\documentclass[11pt,reqno]{amsart}
\usepackage{cite}
\usepackage{enumitem}
\usepackage{etoolbox} 
\usepackage{geometry,graphicx}
\usepackage{amsmath,amssymb,subfig,amsfonts} 
\usepackage{bbm}
\usepackage{color}
\usepackage[table]{xcolor}
\usepackage{hyperref, cleveref}
\usepackage{parskip}
\usepackage{dsfont} 
\usepackage{import}

\usepackage{mathtools}          

\theoremstyle{plain}
\newtheorem{theorem}{Theorem}[section]
\newtheorem{prop}[theorem]{Proposition}
\newtheorem{lem}[theorem]{Lemma}
\newtheorem{coro}[theorem]{Corollary}

\theoremstyle{definition}
\newtheorem{definition}[theorem]{Definition}
\newtheorem{remark}[theorem]{Remark}
\AtEndEnvironment{remark}{\null\hfill\exend}
\newtheorem{question}[theorem]{Question}
\newtheorem{example}[theorem]{Example}
\AtEndEnvironment{example}{\null\hfill\exend}

\newcommand{\Z}{{\mathbb Z}}

\newcommand{\R}{{\mathbb R}}
\newcommand{\N}{{\mathbb N}}
\newcommand{\C}{{\mathbb C}}
\newcommand{\mc}{\mathcal}
\newcommand{\A}{\mc A}
\newcommand{\T}{\mc T}

\newcommand{\dd}{{\mathrm{d}}}

\newcommand{\sub}{\varrho}
\newcommand{\orb}{\operatorname{orb}}

\newcommand{\bbo}{\mathds{1}} 
\newcommand{\Id}{\mathrm{Id}}
\newcommand{\Exp}{\mathrm{Exp}}
\newcommand{\Act}{\mathrm{Act}}

\newcommand{\exend}{\hfill \ensuremath{\Diamond}}

\begin{document}

\title{Substitutions on Compact Alphabets}

 \author{Neil Ma\~nibo, Dan Rust}
\address{School of Mathematics and Statistics, The Open University, \newline
 \hspace*{\parindent}Walton Hall, Milton Keynes, MK7
6AA, UK
}
\email{neil.manibo@open.ac.uk, dan.rust@open.ac.uk }

\author{James J. Walton}
\address{Mathematical Sciences Building, University of Nottingham, \newline
\hspace*{\parindent}University Park, Nottingham, NG7 2RD, UK}
\email{Jamie.Walton@nottingham.ac.uk}

\address{
Fakult\"at f\"ur Mathematik, Universit\"at Bielefeld, \newline
\hspace*{\parindent}Postfach 100131, 33501 Bielefeld, Germany}
\email{cmanibo@math.uni-bielefeld.de}

\keywords{substitutions, infinite alphabets, positive operators, quasi-compactness, unique ergodicity}
\subjclass[2020]{37B10, 47B65, 52C23
}
\begin{abstract}
We develop a systematic approach to continuous substitutions on compact Hausdorff alphabets. Focussing on implications of irreducibility and primitivity, we highlight important features of the topological dynamics of their (generalised) subshifts. We then reframe questions from ergodic theory in terms of spectral properties of a corresponding substitution operator. This requires an extension of standard Perron--Frobenius theory to the setting of Banach lattices. As an application, we identify computable criteria that guarantee quasi-compactness of the substitution operator. This allows unique ergodicity to be verified for several classes of examples. For instance, it follows that every primitive and constant length substitution on an alphabet with an isolated point is uniquely ergodic, a result which fails when there are no isolated points.
\end{abstract}

\maketitle

\section{Introduction}\label{SEC:intro}
Substitutions on infinite alphabets and  tilings with infinite local complexity (ILC) have been steadily gaining attention in the study of symbolic dynamics and aperiodic order \cite{DOP:self-induced, PFS:fusion-ILC, Frett-Richard, SS:discrepancy, RY-profinite, SS:multiscale}. Indeed, infinite alphabets must naturally be considered when recoding non-uniformly recurrent sequences by return words \cite{EG-schroedinger}, or when performing the balanced-pair and overlap algorithms for a non-Pisot substitution \cite{HS:pisot}. In the context of automatic sequences, constant length substitutions on infinite alphabets are the natural setting for regular sequences---sequences admitting a finitely-generated $k$-kernel \cite[Thm.~11]{AS-regular-II}---as well as profinite automatic sequences e.g., those defined over the $p$-adic integers $\Z_p$ \cite[Sec.~4]{RY-profinite}. It has also been shown that $(X,T)$ is a self-induced minimal Cantor system if and only if it is conjugate to a substitution subshift on a zero-dimensional alphabet whose substitution is primitive, recognisable and aperiodic \cite[Thms.~24 and 25]{DOP:self-induced}.

Given their ubiquity, it is important to develop a systematic approach to study such substitutions and their subshifts in terms of their dynamics, ergodic theory and spectral theory. Notable studies in this direction include  those by Queff\'{e}lec \cite{Queffelec}, Ferenczi \cite{Ferenczi}, Frank and Sadun \cite{PFS:fusion-ILC}, and Durand, Ormes and Petite \cite{DOP:self-induced}. The aim of this work is both to establish a framework for further study and to better understand in which ways the infinite alphabet landscape differs from the familiar world of finite alphabets. One novelty of our approach is the incorporation of techniques from the theory of positive operators on Banach lattices.

To obtain more general results in the infinite alphabet setting, it is necessary to impose some standing assumptions. Our alphabets will be equipped with a compact Hausdorff topology with the substitution being a continuous map. These assumptions are natural ones to impose so as to retain some of the structure enjoyed in the finite alphabet setting. Interestingly, most of our results do not require the alphabet to be equipped with a metric, generalising the results and often making the proofs conceptually simpler.

Alphabets, or labels, with compact Hausdorff topologies are also prevalent in the setting of higher dimensional ILC tiling substitutions. These include the Conway--Radin pinwheel substitution \cite{R:pinwheel}, where infinitely many tiles appear up to translation, but up to rigid motion there is just one tile, so that the prototile space is homeomorphic to a disjoint union of two circles.

Our main focus will be on symbolic substitutions, which generate bi-infinite words.
Such symbolic substitutions can be made geometric, as so-called \emph{stone inflations} of \(\R\), by associating to each letter \(a\) an interval tile of length \(\ell(a) \geq 0\) (although we often require that each \(\ell(a) > 0\)). A stone inflation first inflates the support of a tile by a fixed inflation factor \(\lambda >1\) and then decomposes the inflated support exactly into the substituted tiles, meaning here that \(\lambda \ell(a) = \ell(a_1) + \ell(a_2) + \cdots + \ell(a_n)\), where the substitution of \(a\) is \(\sub(a) = a_1 a_2 \cdots a_n\). If such a function \(\ell\) exists, and is continuous on the alphabet, we call it a natural length function. We find conditions on the substitution that imply existence (and uniqueness, up to scaling) of a natural length function but also examples that do not permit one (Examples \ref{exp: non-growing, no length} and \ref{exp: growing, no length}). This is different to the finite alphabet setting, where substitutions always permit natural length functions.

Many of the above cited works mention the difficult problem of establishing unique ergodicity for substitution dynamical systems in the infinite alphabet setting. It is sometimes possible to show this for a particular example, but finding general sufficient criteria is another matter. Beyond the constant-length setting \cite{Queffelec,RY-profinite}, little is known.

One of the principal aims of this work is to establish the first general criterion for unique ergodicity without requiring the substitution to have constant length. This entails considering the substitution operator $M$ on the space $C(\mathcal{A})$ of continuous real-valued functions on the alphabet. The space $C(\mathcal{A})$ is a Banach lattice, allowing us to benefit from known results on positive, mean ergodic and power convergent operators on Banach lattices \cite{Abdelaziz, Eisner, Karlin, SchaeferBook, Schaefer60}. Even with the abundance of literature on such operators, we emphasise that the infinite-dimensional situation is far from straightforward. In particular, primitivity no longer guarantees unique ergodicity; compare \cite[Prop.~31]{DOP:self-induced}. Moreover, $M$ is  a compact operator essentially only when the alphabet is finite (Proposition \ref{prop: non-compact}), and hence the classical Kre\u{\i}n--Rutman theorem \cite{KR-compact} for compact operators is not applicable. We mention this specific theorem as it is one of the simplest generalisations of Perron--Frobenius theory in the infinite dimensional setting. 

However, similar results generalise to quasi-compact operators, where the spectral radius \(r\) is a pole of the resolvent and outside of the essential spectrum (see \cite[Thm.~4.1.4]{Nieberg-BL} and the following remark). More generally, unique ergodicity follows from the weaker condition of strong power convergence of the operator \(T = M/r\), meaning that \(T^n f\) converges strongly for each \(f \in C(\A)\); the result below follows from Theorem \ref{thm:inv measures - tiling} and Corollary \ref{cor: inv measures - subshift} (see also Theorem~\ref{thm: primitive+qc => uniquely ergodic}):

\begin{theorem}
Let $\sub$ be a primitive substitution on a compact Hausdorff alphabet with corresponding subshift $X_\sub$ and substitution operator $M$. Let $r$ be the spectral radius of $M$. If $T = M/r$ is strongly power convergent, then $X_\sub$ is uniquely ergodic.
\end{theorem}

We show how strong power convergence of \(T\) follows from certain conditions on the substitution; see, for example, Theorem \ref{thm: CL UE} in the constant length case. In the non-constant length case, the stronger condition of quasi-compactness can be verified for large classes of examples by applying a computable criterion detailed in Theorem \ref{thm: quasi-compact subs}. This condition appears to be typically satisfied when the substitution is primitive and the alphabet has at least one isolated point. For example, it easily follows from Theorem \ref{thm: quasi-compact subs} that every primitive and constant length substitution on an alphabet with an isolated point defines a quasi-compact operator and thus has uniquely ergodic hull, see Example \ref{ex: CL quasi-compact}. Examples from \cite{DOP:self-induced} show that there are counter-examples to this in alphabets with no isolated points. In the non-compact setting of substitutions on countable alphabets, sufficient conditions for unique ergodicity are provided in \cite{DFMV}.

Our examples show that a subtle range of behaviours can occur in the infinite alphabet setting: whilst some have \(T\) quasi-compact and uniformly power convergent, others demonstrate that it is not uncommon for \(T\) to fail to be uniformly power convergent (Example \ref{ex: CL not qc}), even when \(T\) is strongly power convergent so that the associated subshift is uniquely ergodic. Moreover, the spectral radius $r$ of $M$, which is the infinite-dimensional analogue of the Perron--Frobenius eigenvalue for finite matrices, need not be an algebraic number; see Remark~\ref{rem: trans inflation}.

The structure of this paper is as follows. In Section \ref{SEC:dynamics} we present general properties satisfied by subshifts over compact alphabets. In Section \ref{SEC:substitutions} we define notions related to substitutions and we present immediate consequences of continuity. We also discuss the associated substitution subshift and investigate its language. The substitution operator is introduced in Section \ref{SEC:tile lengths}, and its operator-theoretic and spectral properties are discussed. Section~\ref{SEC:invariant} deals with invariant measures on the subshift and its relation to strong power convergence of the (normalised) substitution operator $T$. Here, we exploit a correspondence established in \cite{PFS:fusion-ILC} in the fusion tiling formalism. Applications of our results, including sufficient conditions for unique ergodicity of the subshift and representative examples are discussed in Section~\ref{SEC:applications}. Section~\ref{SEC:discrepancy} deals with asymptotic behaviour of discrepancy estimates for certain substitutions whose corresponding operator has a spectral gap. To simplify the exposition, and because several results consider the existence of natural length functions, we mostly work in the setting of one-dimensional and symbolic substitutions. However, in Section \ref{sec: higher dimensions}, we briefly outline which results here generalise to higher dimensions. In particular, Theorem \ref{thm: CL UE} on a criterion for constant length substitutions to have unique ergodicity naturally extends to higher dimensions, which provides a straightforward proof of unique ergodicity of the dynamical system associated with the Conway--Radin pinewheel tilings. Numerous examples are provided throughout the text to demonstrate our results.

\section{Topology and dynamics}\label{SEC:dynamics}
Most of the results in this section are routine or simple exercises, and are likely already known in the wider literature, for instance in works such as \cite{A:auslander}. Nevertheless, these results may be unfamiliar to the reader in the general setting of infinite alphabets and lay the groundwork for our main focus, substitutions on compact Hausdorff alphabets. We therefore include them, often without proof, in order to keep our work as self-contained as possible.

Let $\mc A$ be a compact Hausdorff space that we call an \emph{alphabet} and whose elements we call \emph{letters}. Let $\mc A^+ = \bigsqcup_{n \geq 1} \mc A^n$ denote the set of all finite (non-empty) \emph{words} over the alphabet $\mc A$, where $\mc A^n$ has the product topology and $\mc A^+$ is topologised as the disjoint union. Hence $\mc A^n$ is a clopen subset of $\mc A ^+$ for each $n \geq 1$. For ease of notation we write
\[
u_1 u_2 \cdots u_n \coloneqq (u_1, u_2, \ldots, u_k).
\]
Let $\mc A^\ast = \mc A^+ \sqcup \mc A^0 = \mc A^+ \sqcup \{\varepsilon\}$, where $\varepsilon$ is the \emph{empty word}.
\emph{Concatenation} is a binary operation $\mc A^\ast \times \mc A^\ast \to \mc A^\ast$ given by $(\varepsilon, u) \mapsto u$, $(u, \varepsilon)\mapsto u$ and
\[
(u_1 \cdots u_n, v_1 \cdots v_m) \mapsto u_1 \cdots u_n v_1 \cdots v_m,
\]
where $u = u_1 \cdots u_n \in \mc A^n$ and $v = v_1 \cdots v_m \in \mc A^m$. We write $uv$ as shorthand for the concatenation of $u$ and $v$, which is a continuous operation. If there is a $j \geq 0$ such that $u_i = v_{j+i}$ for all $1 \leq i \leq n$, then we call $u$ a \emph{subword} of $v$ and write $u \triangleleft v$. We have the continuous projection map $p_i \colon \mc A^n \to \mc A$ to the $i$th letter of a word.

Let $\mc A^{\Z}$ denote the set of bi-infinite sequences over $\mc A$ with the product topology, which is compact by Tychonoff's theorem. We use a vertical line $|$ to denote the position between the $-1$st and $0$th element of a bi-infinite sequence, and so we write $w = \cdots w_{-2} w_{-1} | w_0 w_1 \cdots$. For $j \leq k$, we let $w_{[j,k]}$ denote the subword $w_j\cdots w_k$. We define the projection $p_{[j,k]} \colon \mc A^\Z \to \mc A^{k-j+1}$ by $p_{[j,k]}(w) = w_{[j,k]}$. This is also clearly a continuous function.

The function $\sigma \colon \mc A^{\Z} \to \mc A^{\Z}$ given by $\sigma(x)_i = x_{i+1}$ is a homeomorphism called the \emph{(left) shift} map. The pair $(\mc A^{\Z}, \sigma)$ is called the \emph{full shift} over the alphabet $\mc A$. Let $| \cdot | \colon \mc A^* \to \N_0 = \N \cup \{0\}$ be the \emph{word-length function} mapping $u \mapsto n$ for every $u \in \mc A^n$. The word-length function is continuous because $| \cdot |^{-1}(\{n\}) = \left\{ u \in \mc A^* \mid |u|=n\right\} = \mc A^n$ is open in $\A^\ast$ for every $n$.

\begin{definition}
Let $X \subseteq \mc A^\Z$. If $\sigma(X) = X$,  we say that $X$ is \emph{shift-invariant} or \emph{$\sigma$-invariant}. If $X$ is a non-empty, closed, shift-invariant subspace of $\mc A^\Z$ then we call $X$ a \emph{subshift}. We call $X$ a \emph{minimal} subshift if $X$ contains no proper subshift $\varnothing \neq Y \subsetneq X$.
\end{definition}

\begin{definition}
Let $\mc{L} \subseteq \A^\ast$. For \(n \in \N_0\), we let $\mc{L}^n = \mc{L} \cap \A^n$. We call $\mc{L}$ a \emph{language} if each $\mc{L}^n \subseteq \A^n$ is non-empty, closed (equivalently, $\mc{L} \subseteq \A^\ast$ is closed), and $\mc{L}$ is closed under taking subwords. The \emph{subshift associated with $\mc{L}$} is the subset
\[
X_\mc{L} \coloneqq \{w\in \A^\Z \mid w_{[j,k]} \in \mc{L} \text{ for all } j \leq k\}.
\]
\end{definition}

\begin{prop} \label{prop: language -> subshift}
For any language $\mc{L}$, we have that $X_\mc{L}$ is a subshift.
\end{prop}

\begin{definition}
Let $w \in \mc A^\Z$ and let $\orb(w) \coloneqq \{\sigma^n(w) \mid n \in \Z\}$ denote the orbit of $w$ under the shift map $\sigma$. We let $X_w \coloneqq  \overline{\orb(w)} \subseteq \mc A^\Z$ denote the \emph{orbit closure} of $w$. Let $\mc W^n(w)$ denote the set of all $n$-letter subwords appearing in $w$. We write $\mc W(w) \coloneqq \bigsqcup_{n \geq 0} \mc W^n(w)$, seen as a subspace of $\mc A^\ast$. We define $\mc{L}^n(w) = \overline{\mc{W}^n(w)}$ and similarly the \emph{language} of $w$ as $\mc{L}(w) = \overline{\mc{W}(w)}$.
\end{definition}

\begin{remark} \label{rem: language partition}
Of course, we have that
\[
\mc{L}(w) = \bigsqcup_{n \geq 0} \mc{L}^n(w)
\]
and $\mc{L}^n(w) = \mc{L}(w) \cap \A^n$, since each of the $\mc A^n$ are disjoint clopen subsets of $\mc A^\ast$.
\end{remark}

\begin{definition}
For a subshift $X \subseteq \mc A^\Z$ we let $\mc L^n(X)$ denote the set of all $n$-letter subwords appearing in elements of $X$, that is, $\mc{L}^0(X) = \{\varepsilon\}$ and
\[
\mc L^n(X) \coloneqq \left\{u \in \mc A^n \mid \exists w \in X, \:\: j \leq k \ \text{ such that }\ u = w_{[j,k]}\right\}.
\]
We write $\mc L(X) \coloneqq \bigsqcup_{n \geq 0} \mc L^n(X)$, seen as a subspace of $\mc A^\ast$, and call $\mc L(X)$ the \emph{language} of $X$.
\end{definition}

The following results are useful and left as exercises to the reader.
\begin{prop} \label{prop: are languages}
If $w \in \A^\Z$ then $\mc{L}(w)$ is a language and, for a subshift $X \subseteq \A^\Z$, we have that $\mc{L}(X)$ is a language. 
\end{prop}

\begin{prop} \label{prop: language of orbit}
For $w \in \A^\Z$ we have that $X_w$ is a subshift with $\mc{L}(X_w) = \mc{L}(w)$.
\end{prop}

\begin{remark} \label{rem: all letters}
We assume in this section, without loss of generality, that all letters of $\A$ appear in some word of the subshift $X$. Indeed, by shift-invariance, the set of all letters that appear is given by $\mc{L}^1(X) = p_{[0,0]}(X) \subseteq \A$. Since $X \subseteq \A^\Z$ is closed, we have that $X$ is compact and thus so is $\mc{L}^1(X)$ by continuity of $p_{[0,0]}$. So we may assume that all letters appear by restricting the alphabet to $\mc{L}^1(X)$, which is still compact and Hausdorff.
\end{remark}

Recall that a topological dynamical system $(X,f)$ is called \emph{topologically transitive} if for all non-empty open sets $U, V \subseteq X$, there exists an $n \in \Z$ such that $f^n(U) \cap V \neq \varnothing$.

The following lemmas are routine.

\begin{lem}\label{LEM:DO-implies-TT}
Let $X \subseteq \mc A^\Z$ be a subshift. If there exists an element $w \in X$ with a dense orbit, then the subshift $X$ is topologically transitive.
\end{lem}

\begin{lem}
Let $X \subseteq \mc A^\Z$ be a subshift. If there exists an element $w \in X$ with dense orbit then $\mc A$ is separable.
\end{lem}

Since $X$ is compact Hausdorff, by the Baire category theorem we have that $X$ is Baire and thus of second category. Hence, (by an identical proof to \cite[Prop.~1.1]{S:baire}) we have a converse to Lemma \ref{LEM:DO-implies-TT}:

\begin{prop}\label{PROP:do-iff-tt}
Let $X \subseteq \A^\Z$ be a subshift. There exists an element $w \in X$ with a dense orbit if and only if the subshift $X$ is topologically transitive and $\A$ is separable.
\end{prop}

\begin{definition}
Let $w \in \A^\Z$ be a bi-infinite sequence over $\A$. We say that $w$ is \emph{repetitive} if for all $n \geq 1$ and every non-empty open set $U \subseteq \mc L^n(w)$, there exists $N = N(n,U) \geq 1$ such that every $v \in \mc L^N(w)$ contains a subword in $U$.
\end{definition}

The above definition may be interpreted as saying that any given finite word appears, up to some specified tolerance (determined by the open set $U$), with uniformly bounded gaps (within any word of size $N$). The lemma below (properties 2 and 3) says that this is equivalent, for any given tolerance, to \emph{all} $n$-letter words being found in all $N$-letter words, for some $N$ depending on the tolerance.

\begin{samepage}
\begin{prop} \label{prop: rep all words}
Let $w \in \A^\Z$ be a bi-infinite sequence over $\A$ and let us write $\mc{L}^n = \mc{L}^n(w)$. The following are equivalent:
\begin{enumerate}
	\item $w$ is repetitive;
	\item for all $n \geq 1$ and every finite collection $\mathcal{U} = \{U_i\}_{i=1}^\ell$ of non-empty open subsets $U_i \subseteq \mc{L}^n$, there exists an $N = N(n,\mathcal{U}) \geq 1$ such that for all $v \in \mc L^N$ and for every $1 \leq i \leq \ell$, there is a subword of $v$ in $U_i$;
	\item for all $n \geq 1$ and every open subset $U \subseteq \mc{L}^n \times \mc{L}^n$ containing the diagonal, there exists some $N = N(n,U) \geq 1$ so that, for every $u \in \mc{L}^n$ and $v \in \mc{L}^N$, there is a subword $u' \triangleleft v$ with $(u,u') \in U$.	
\end{enumerate}
\end{prop}
\end{samepage}

\begin{proof}
The implications ($1 \implies 2$) and ($3 \implies 1$) are trivial.

($2 \implies 3$):
Let $U \subseteq \mc{L}^n \times \mc{L}^n$ be open and contain the diagonal. It is well known (and not hard to prove) that, by compactness, we may find $V \subseteq \mc{L}^n \times \mc{L}^n$ satisfying the same, with $V = V^{-1}$ (where $V^{-1}$ consists of pairs $(x,y)$, for $(y,x) \in V$) and $V \circ V \subseteq U$ (where $(u,w) \in V \circ V$ if there is some $(u,v) \in V$ and $(v,w) \in V$).

For $u \in \mc{L}^n$ let $V_u \subseteq \mc{L}^n$ denote the set of $u' \in \mc{L}^n$ with $(u,u') \in V$. Then let $\{V_{u_i}\}_{i=1}^\ell$ be a finite subcover of $\{V_u\}_{u \in \A}$. Take $N \geq 1$ according to property 2. Let $u \in \mc{L}^n$ and $v \in \mc{L}^N$ be arbitrary and $i$ such that $u \in V_{u_i}$. Hence, $(u,u_i) \in V$. By property 2, we may find some $u' \in V_{u_i}$ with $u' \triangleleft v$. Then $(u_i,u') \in V$ so $(u,u_i) \circ (u_i,u') = (u,u') \in V \circ V \subseteq U$, as required.
\end{proof}

\begin{remark} \label{rem: rep metric}
An alternative characterisation to the above is via open subsets of $\A$ (rather than of $\mc{L}^n$), or open sets containing the diagonal in $\A \times \A$, by comparing words letter-by-letter to define the `tolerance'.

In this case, property 3 is easy to interpret when $\A$ has a metric $d$. Then we can take as $U$ the pairs of words $(u,u')$ so that $d(u_i,u_i') < \epsilon$ for each pair of letters $u_i \triangleleft u$ and $u_i' \triangleleft u'$. The property then says: for every $n \geq 1$ and $\epsilon > 0$ there is some $N = N(n,\epsilon) \geq 1$ so that, for every $u \in \mc{L}^n$ and $v \in \mc{L}^N$, there is some subword $u' \triangleleft v$ with $d(u,u') < \epsilon$. This definition in the metric setting has appeared previously, for instance in the work of Frettl\"{o}h and Richard \cite{Frett-Richard}.
\end{remark}

The following is well known for general invertible topological dynamical systems \cite[Rem.~2.12]{GH-topdyn} and so we state it without proof.

\begin{prop}\label{PROP:minimal-rep}
Let $X \subseteq \A^\Z$ be a subshift. The subshift $X$ is minimal if and only if every element $w \in X$ has a dense orbit.
\end{prop}
Likewise, the following lemmas are simple exercises.

\begin{lem}\label{LEM:dense-implies-rep}
Let \(X \subseteq \A^\Z\) be a subshift. If \(X\) is minimal, every \(w \in X\) has the same language and is repetitive, with the same value of \(N = N(n,U)\) for each \(n \geq 1\) and non-empty open set \(U \subseteq \mc{L}^n(w)\).
\end{lem}

\begin{lem}\label{LEM:rep-implies-min}
If $w \in \A^\Z$ is repetitive, then $X_w$ is minimal.
\end{lem}

Putting together Proposition \ref{PROP:minimal-rep},  Lemma \ref{LEM:dense-implies-rep} and Lemma \ref{LEM:rep-implies-min} gives us the following.
\begin{coro}\label{CORO:minimal-equiv}
Let $w \in \A^\Z$. The following are equivalent:
\begin{enumerate}
\item $X_w$ is minimal;
\item every element $x \in X_w$ has a dense orbit;
\item $w$ is repetitive;
\item every element $x \in X_w$ is repetitive.
\end{enumerate}
\end{coro}

\section{Substitutions}\label{SEC:substitutions}
Our main focus will be on subshifts coming from substitutions, to which we now turn. The theory of substitutions on finite alphabets is well developed \cite{BG:book,F:book,Queffelec}, and the reader is encouraged to familiarise themselves with that setting in order to compare and contrast our results with the classical theory. Where appropriate, we highlight these differences and subtleties.

\begin{definition}
Let $\sub \colon \A \to \A^+$ be a continuous function.
We call such a function a \emph{substitution} on $\A$.
We say $\sub$ is a substitution of \emph{constant length} $n$ if $\sub(\A) \subseteq \A^n$.
\end{definition}

Note that continuity of $\sub$ is automatically satisfied when $\A$ is finite (and thus given the discrete topology). To give an idea of the variety of substitutions that satisfy this general definition, consider the following examples on infinite alphabets:

\begin{example}\label{ex:non-CL}
 Take $\A = \N_\infty = \N_0 \cup \{\infty\}$, the one-point compactification of the natural numbers $\N_0$, with the substitution $\sub \colon \A \to \A^+$ given by
\[
\sub \colon \left\{
\begin{array}{rl}
0\       \mapsto & 0\ 0\ 0\ 1          \\
n\       \mapsto & 0\  n\!-\!1\  n\!+\!1    \\
\infty\  \mapsto & 0\ \infty\ \infty.
\end{array}\right. 
\]
\end{example}

\begin{example}\label{ex:CL-S1}
Take $\A = S^1 = \{z \in \C \mid |z| = 1\}$ as the unit circle in the complex numbers. For fixed $\alpha \in \A$ we have the substitution $\sub \colon \A \to \A^+$
\[
\sub \colon z \mapsto z\ \alpha z,
\]
where the image of $z$ contains itself and $\alpha z$, which is the product of $z$ with $\alpha$. 
\end{example}

\begin{example}
For an iterated function system $\{f_1, \ldots, f_d\}$ on a compact space $X$ with attractor $Y$, let $\A = Y$, with the substitution $\sub \colon \A \to \A^+$ given by
\[
\sub \colon a \mapsto f_{w_1(a)}(a)\ \cdots\ f_{w_k(a)}(a),
\]
where the word $w_1(a) \cdots w_k(a) \in \{1, \ldots, d\}^+$ is locally constant on $Y$.
\end{example}

Example \ref{ex:CL-S1} above is a constant length substitution because $|\sub(z)| = 2$ for all $z \in \A$.
In fact, we will see that we have no choice but to make the substitution constant length when the alphabet is $\A = S^1$.
The next result shows that topological properties of the alphabet, namely connectedness, can heavily restrict properties of substitutions on that alphabet.
We see also that, even though our alphabets are potentially infinite, continuity of the substitution and compactness of the alphabet ensure that the images of letters under substitution are uniformly bounded in length.
\begin{prop}\label{PROP:const-length}
Let $\sub \colon \A \to \A^+$ be a substitution.
There exists $k \geq 1$ such that for all $a \in \A$, $|\sub(a)| \leq k$.
If the alphabet $\A$ is connected, then $\sub$ is constant length.
\end{prop}
\begin{proof}
The function $| \cdot | \circ \sub \colon \A \to \N$ is continuous as it is the composition of two continuous functions.
It follows that $|\sub(\A)|$ is compact by the compactness of $\A$ and so is bounded.
If $\A$ is connected, then so is $|\sub(\A)|$ and hence there exists an $n\geq 1$ such that for all $a \in \A$, $|\sub(a)| = n$.
\end{proof}

Similarly, a substitution is of course constant length on any connected subset of the alphabet. Given Proposition \ref{PROP:const-length}, we may define the \emph{length} of a substitution by $\displaystyle |\sub| \coloneqq \max_{a \in \A} |\sub(a)| \in \N$.

It is often the case that checking continuity of a substitution is easier if one can do it component-wise. Indeed, Durand, Ormes and Petite gave an alternative equivalent definition of a substitution \cite{DOP:self-induced} (there called a \emph{generalised substitution}). Recall that $p_i \colon \A^n \to \A$ is the canonical projection function onto the $i$th component of the product $\A^n$.

\begin{prop}\label{PROP:continuous}
A function $\sub \colon \A \to \A^+$ is continuous if and only if for every $n \geq 1$ the subspace $\sub^{-1}(\A^n) \subseteq \A$ is open, and for every $1 \leq i \leq n$, the composition $p_i \circ \sub \colon \sub^{-1}(\A^n) \to \A$ is continuous.
\end{prop}

\begin{proof}
Suppose that $\sub\colon \A \to \A^+$ is continuous.
As $\A^n$ is an open subset of $\A^+$, and $\sub$ is continuous, it follows that $\sub^{-1}(\A^n)$ is open in $\A$.
Then, as $p_i$ is continuous for every $1 \leq i \leq n$, and restricting continuous functions to subspaces preserves continuity, it follows that $p_i \circ \sub \colon \sub^{-1}(\A^n) \to \A$ is continuous.

For the other direction, it is enough to show that
\[
\sub|_{\sub^{-1}(\A^n)} \colon \sub^{-1}(\A^n) \to \A^+
\]
is continuous for every $n \geq 1$, as the (disjoint) sets $\sub^{-1}(\A^n)$ are open in $\A$. For any particular $n$, the universal property of products says that $\sub|_{\sub^{-1}(\A^n)}$ is continuous if and only if $p_i \circ \sub \colon \sub^{-1}(\A^n) \to \A$ is continuous for every $1 \leq i \leq n$ and so we are done.
\end{proof}
By Proposition \ref{PROP:const-length}, one is then only required to check the compositions $p_i \circ \sub $ for a finite number of the subalphabets $\sub^{-1}(\A^n)$ in order to confirm continuity of $\sub$.

A substitution naturally extends to a continuous function  $\sub \colon \A^\ast \to \A^\ast$. We define $\sub(\varepsilon) \coloneqq \varepsilon$ and for $u = u_1 \cdots u_n \in \A^n \subseteq \A^+$ we define
\[
\sub(u) \coloneqq \sub(u_1) \cdots \sub(u_n)
\]
by concatenation. Since \(\sub \colon \A^* \to \A^*\) has the same domain and codomain it may be iterated. By continuity of $\sub$, all iterates $\sub^k$ for \(k \in \N\) are also continuous and so are also substitutions. Words of the form $\sub^k(a)$ are called $k$-\emph{superwords}.

\begin{definition}
Let $\sub\colon \A \to \A^+$ be a substitution. We say a word $u \in \A^n$ is \emph{generated} by $\sub$ if there exist $a \in \A$, $k \geq 0$ such that $\sub^k(a)$ contains $u$ as a subword. We say that \(\sub\) \emph{generates arbitrarily long words} if, for each \(n \in \N\), there is a word of length \(n\) generated by \(\sub\). In this case, the \emph{language} of $\sub$ is
\[
\mc{L}(\sub) = \overline{\{u \in \A^\ast \mid u \text{ is generated by } \sub\}}
\]
and we call words in $\mc L(\sub)$ \emph{legal}. The legal words of length $n$ are denoted by $\mc{L}^n(\sub)$. The \emph{subshift associated with $\sub$} is $X_\sub \coloneqq X_{\mc{L}(\sub)}$.
\end{definition}

Note that the language includes closure points of words generated by $\sub$. If closure points were not included then the following lemma, stating that \(\mc{L}(\sub)\) is a language, would fail and thus \(X_\sub\) would also not be guaranteed to be a subshift. 

\begin{lem} \label{lem: closure of language under subst and subwords}
Let \(v \in \mc{L}(\sub)\). Then \(\sub(v) \in \mc{L}(\sub)\) and \(u \in \mc{L}(\sub)\) for any \(u \triangleleft v\). In particular, \(\mc{L}(\sub)\) is a language if \(\sub\) generates arbitrarily long words.
\end{lem}

\begin{proof}
Let $Y$ be the set of generated words, so that $\mc{L}(\sub) = \overline{Y}$. By continuity of $\sub$ we have that $\sub(\mathcal{L}(\sub)) = \sub\left(\overline{Y}\right) \subseteq \overline{\sub(Y)}$. If $u \in Y$ then $u \triangleleft \sub^n(a)$ for some $a \in \A$ and $n \in \N_0$ and thus $\sub(u) \triangleleft \sub^{n+1}(a)$ so that $\sub(Y) \subseteq Y$. Hence $\overline{\sub(Y)} \subseteq \overline{Y} = \mathcal{L}(\sub)$ and thus $\sub(\mathcal{L}(\sub)) \subseteq \mathcal{L}(\sub)$, as required.

The proof for closure under subwords is similar. Let $Y_n$ and $Y_m$ denote the sets of generated words of length $n > m$. Given some $0 \leq j \leq n-m$, let $p \colon \A^n \to \A^m$ be the continuous map that takes the $j$th subword of length $m$. Then $p(\mc{L}^n(\sub)) = p\left(\overline{Y_n}\right) \subseteq \overline{p(Y_n)}$ by continuity. Clearly a subword of a generated word is generated, so $\overline{p(Y_n)} \subseteq \overline{Y_m} = \mathcal{L}^m(\sub)$ and thus $p(\mc{L}^n(\sub)) \subseteq \mc{L}^m(\sub)$. Since $n > m$ and $j$ were arbitrary, we see that $\mc{L}(\sub)$ is closed under taking subwords.

By definition, \(\mc{L}(\sub)\) is closed, and by the above is closed under taking subwords. If \(\sub\) generates arbitrarily long words then each \(\mc{L}^n(\sub) \neq \varnothing\) and thus is a language.
\end{proof}

\begin{coro} \label{cor: sub of subshift}
Let $\sub \colon \A \to \A^+$ be a substitution. Then $\sub(X_\sub) \subseteq X_\sub$.
\end{coro}

\begin{proof}
Let $w \in X_\sub$, so every subword of $w$ is legal. Let $u \triangleleft \sub(w)$. Then $u \triangleleft \sub(v)$ for some $v \triangleleft w$, the latter implying $v \in \mc{L}(\sub)$. By Lemma \ref{lem: closure of language under subst and subwords} we have $\sub(v) \in \mc{L}(\sub)$. By the same lemma, since $u \triangleleft \sub(v)$, we have $u \in \mc{L}(\sub)$. Since $u \triangleleft \sub(w)$ was arbitrary, we see that $\sub(w) \in X_\sub$, as required.
\end{proof}

\begin{coro} \label{cor: sub is subshift}
Let $\sub \colon \A \to \A^+$ be a substitution that generates arbitrarily long words. The associated space $X_\sub \subseteq \A^\Z$ is a subshift (in particular, \(X_{\sub} \neq \varnothing\)). Conversely, if there exists some $k \geq 1$ so that $|\sub^n(a)| < k$ for all $n \in \N$ and \(a \in \A\), then $X_\sub = \varnothing$.
\end{coro}

\begin{proof}
If $|\sub^n(a)| < k$ for all $a \in \A$ and $n \in \N$ then $\mc{L}^k(\sub) = \varnothing$ and hence $X_\sub = \varnothing$. Otherwise, \(\mc{L}(\sub)\) is a language by Lemma \ref{lem: closure of language under subst and subwords} and $X_\sub$ is a subshift by Proposition \ref{prop: language -> subshift}.
\end{proof}

\begin{coro}\label{COR:non-empty}
Let $\sub \colon \A \to \A^+$ be a substitution with associated subshift $X_\sub$.
If there exists a letter $a \in \A$ such that $\lim_{n \to \infty} |\sub^n(a)| = \infty$, then $X_\sub \neq \varnothing$ is a subshift.
\end{coro}

One might expect the converse of the above to hold by some compactness argument; it of course holds for finite or connected alphabets $\A$. However, it is possible to engineer an example where this is not the case:

\begin{example} \label{exp: non-growing}
Let $\A = \N_{\infty} \times \N_{\infty}$, where $\N_{\infty}= \N_0 \cup \{\infty\}$ is the one-point compactification of $\N_0$. Define a substitution $\sub \colon \A \to \A^+$ by
\[
\sub(n,m) =
\begin{cases}
(0,0), & \text{if } n = 0, m = 0, \\
(n,m-1), & \text{if } m > 0,\\
(n-1,n)(0,n), & \text{if } n > 0 \text{ and } m = 0,
\\
\end{cases}
\]
where $\infty - 1 \coloneqq \infty$. By Proposition \ref{PROP:continuous}, $\sub$ is continuous. One quickly sees that
\[
\begin{array}{lcll}
{\displaystyle \lim_{i \to \infty}\sub^i(n,m) }& = & (0,0)^{n+1} & \text{ if } n,m \neq \infty;\\
{\displaystyle \lim_{i \to \infty}\sub^i(\infty,m) }& = & (\infty,\infty)(0,\infty) & \text{ if } m \neq \infty; \\
{\displaystyle \lim_{i \to \infty}\sub^i(n,\infty) }& = & (n, \infty). &
\end{array}
\]
For every $k \in \N$, the letter $(k,0)$ eventually grows to length $k+1$ under iterated substitution, meaning that $X_\sub$ is non-empty by Corollary \ref{cor: sub is subshift}. In fact,
\[
X_\sub = \{\cdots (0,0)(0,0)|(0,0)(0,0) \cdots \}.
\]
However, every letter is eventually constant under substitution, so no letter grows without bound. 
\end{example}

\begin{example} \label{exp: language not realised}
It may be that not all letters of $\A$ appear in $X_\sub = X(\mc{L}(\sub))$, even if \( \# \A < \infty\) and the alphabet cannot be reduced to give the same subshift. Indeed, consider the substitution
\[
\sub \colon
\begin{cases}
a \mapsto a, \\
b \mapsto ab.
\end{cases}
\]
It is easy to see that $\mc{L}(\sub) = \{a^n,a^nb \mid n \in \N_0\}$. The only bi-infinite element admitted by this language is the periodic word containing only $a$, and the substitution cannot be restricted to a smaller alphabet that gives the same.

We see that $\mc{L}(X_\sub) = \{a^n \mid n \in \N_0\}$ and (as also for the previous example) it is possible for $\mc{L}(X_\sub) \subseteq \mc{L}(\sub)$ to be a strict inclusion. However, for the sufficiently well behaved substitutions of main interest here (such as primitive substitutions, but also see Proposition \ref{prop: language realised}) the full language, and in particular the full alphabet, is realised by $X_\sub$.
\end{example}

\begin{remark} \label{rem: letter surjective}
For $C \subseteq \A$, let $s(C)$ be the set of letters in substitutes of elements of $C$ i.e.,
\[
s(C) \coloneqq \{a \triangleleft \sub(c) \mid c \in C, a \in \A\}.
\]
One may always restrict to a \emph{letter surjective} substitution, in the sense that $\A = s(\A)$. Indeed, consider the eventual range
\[
\mc{B} = \bigcap_{n \in \N} s^n(\A).
\]
This is a nested intersection of non-empty compact sets and thus $\mc{B} \neq \varnothing$. Since $s(\mc{B}) = \mc{B}$, the substitution is well defined and letter surjective on $\mc{B}$. All elements of the shift $X_\sub$ contain only letters in $\mc{B}$. Indeed, let $a \triangleleft w \in X_\sub$ for $a \in \A$ and let $Y_n$ be the set of words of length $n$ generated by $\sub$. Since $a \triangleleft v$ for $v \in \mc{L}(\sub)$ and $v$ arbitrarily long, we have that $a \triangleleft v \in \overline{Y_n}$ for $n$ arbitrarily large. Since the length of the substitution is bounded (Proposition \ref{PROP:const-length}) we have that all letters of $Y_n$ are in $s^k(\A)$, where $k$ can be made arbitrarily large by choosing $n$ sufficiently large. It follows that $a \in \overline{s^k(\A)} = s^k(\A)$ for $k$ arbitrarily large and hence $a \in \mc{B}$, so all elements of the shift contain only letters of $\mc{B}$. This does not mean, however, that restricting the substitution to $\mc{B}$ does not give a smaller subshift, as seen in the next example.
\end{remark}

\begin{example}
Consider the finite substitution
\[
\sub \colon
\begin{cases}
a \mapsto bc, \\
b \mapsto bb, \\
c \mapsto cc.
\end{cases}
\]
One may restrict the substitution to its eventual range $\mc{B} = \{b,c\}$, which generates the subshift of two periodic elements of all $b$s and all $c$s. However, $\sub^{n+1}(a) = b^{2^{n}}c^{2^{n}}$ and thus we have the non-periodic element $\cdots bbbccc \cdots \in X_\sub$ in the subshift of the original substitution.
\end{example}

The following result identifies those substitutions whose languages (and in particular full alphabets) are fully realised in the subshift. A result of this form is only of interest when \(\sub\) generates arbitrarily long words since, otherwise, (1) and (2) fail whilst the subshifts in the equalities of (3) and (4) are not defined in this case.

\begin{prop} \label{prop: language realised} Let $\varrho$ a substitution that generates arbitrarily long words. 
The following are equivalent:
\begin{enumerate}
	\item for every $n \in \N$ and non-empty open subset $U \subseteq \A$, there is some word generated by $\sub$ of length $2n+1$ whose central letter is in $U$;
	\item for all $a \in \A$ and $n \in \N$, there is a legal word of length $2n+1$ whose central letter is $a$;
	\item $\mc{L}^1(X_\sub) = \mc{L}^1(\sub) = \A$;
	\item $\mc{L}(X_\sub) = \mc{L}(\sub)$.
	\end{enumerate}
\end{prop}

\begin{proof}
($1 \implies 2$): Let $a \in \A$ be arbitrary. By assumption, for each open set $U \subseteq \A$ containing $a$ we may construct a word $v_U$ of length $2n+1$ that is generated by $\sub$ and has a letter of $U$ at its centre. By compactness of $\A^{2n+1}$ there is a subnet of $(v_U)$ that converges to some word $v$. By construction $v$ has $2n+1$ letters, contains $a$ at its centre and is the limit of a net of words generated by $\sub$. Thus $v$ is in the closure of the generated words, so is legal, as required.

($2 \implies 3$): By definition, $\mc{L}^1(\sub) = \A$ (since $a \triangleleft \sub^0(a) = a$). To see that every letter $a \in \A$ appears in $X_\sub$, by using the assumption we may find a sequence $(v_n)$ of legal words of length $2n+1$ whose central letters are all $a$. By compactness of $\A$ we may choose a subsequence $S_1$ of $(v_n)$ that converges at positions $-1$ and $+1$, say to $a_{-1}$ and $a_1$, respectively. Similarly, we may find a subsequence $S_2$ of $S_1$ that also converges at positions $a_{-2}$ and $a_2$. Continuing in this way, we construct an infinite word $w = \{a_i\}_{i \in \Z}$ and subsequences $S_n$ of legal words that converge to $w_{[-n,n]}$ when restricted to the word of length $2n+1$ about the origin. Since $\mc{L}(\sub)$ is closed we see that $w_{[-n,n]} \in \mc{L}(\sub)$. As these cover all of $w$, it follows that $w \in X_\sub$, since then any finite subword of $w$ is a subword of some $w_{[-n,n]}$, which is legal. As $a = w_0$ we have that $a \in \mc{L}^1(X_\sub)$, as required.

($3 \implies 4$): We show that $\sub^n(a) \in \mc{L}(X_\sub)$ for each $a \in \A$ and $n \in \N_0$. By assumption, for each $a \in \A$ there is some $w \in X_\sub$ so that $a \triangleleft w$. Then $\sub^n(a) \triangleleft \sub^n(w)$, and $\sub^n(w) \in X_\sub$ by Corollary \ref{cor: sub of subshift}. We see that $\mc{L}(X_\sub)$ contains all generated words. Since languages are always closed (Proposition \ref{prop: are languages}) we see that $\mc{L}(\sub) \subseteq \mc{L}(X_\sub)$. The reverse inclusion is trivial, by definition of $X_\sub$.

($4 \implies 1$): Let $U \subseteq \A$ be non-empty and open and $n \in \N$ arbitrary. Take any $a \in U$. Choose any $w \in X_\sub$ with $w_0 = a$, which we may do using shift-invariance and the fact that, in particular, our assumption implies that $\mc{L}^1(X_\sub) = \mc{L}^1(\sub) = \A$. Then $w_{[-n,n]}$ is legal, of length $2n+1$ and has $w_0 = a$. Since the legal words are in the closure of the generated words, it follows that we may find a generated word $v$ with $v_0 \in U$.
\end{proof}

\begin{remark} \label{rem: language realised => letter surjective}
If the equivalent conditions of the above result are satisfied then $\sub$ is letter surjective. Indeed, in this case any given $a \in \A$ is legal, so for arbitrary open $U \subseteq \A$ containing $a$ we can find some superword $\sub^k(b) = \sub(\sub^{k-1}(b))$ containing a letter of $U$. In particular, there is some $b_U \in \A$ with $\sub(b_U)_i \in U$. By compactness of $\A$, there is a convergent subnet of $(b_U)$, which by continuity has limit $b$ satisfying $a \triangleleft \sub(b)$. However, being letter surjective is not sufficient, as demonstrated by Example \ref{exp: language not realised}.
\end{remark}

\begin{remark}
In the case that $\A$ is infinite, it is not necessarily true that we may find any given $a \in \A$ in the interior of a word generated by $\sub$ in the above result. For example, consider $\A = \N_0 \cup \{\infty\}$ and
\[
\sub \colon
\left\{
\begin{array}{rl}
n\       \mapsto & 0 \ n\!+\!1 \\
\infty\  \mapsto & 0 \ \infty.
\end{array}
\right.
\]

It is not hard to see that every letter of $\A$ appears in some word of $X_\sub$. In particular, we have the fixed point $x= \sub^\infty (\infty\,|\,0) \in X_\sub$; that is, we can build the point $x$ as the limit of nested subwords $(u^{(n)})_{n\geq 0}$ centred at the origin by defining $u^{(0)} = \infty\,|\,0$ and $u^{(n)} = \sub^n(\infty)\,|\,\sub^n(0)$ for $n \geq 0$. However, $\infty$ appears only as the final term of superwords of the form $\sub^n(\infty)$.
\end{remark}

The above proposition gives necessary and sufficient conditions for the language of the subshift to agree with the legal words. The next result concerns this set of legal words, which, a priori, requires taking a closure of the generated words for each $\mc{L}^n$. Fortunately, as long as all letters grow without bound under substitution, then, once we know the 2-letter legal words, we can construct all other legal words by substituting these and the letters $\A$.

\begin{lem}\label{LEM:legal-words}
Let $\sub \colon \A \to \A^+$ be a substitution such that $|\sub^n(a)| \to \infty$ as $n \to \infty$ for every $a \in \A$. For a subset $U \subseteq \A^\ast$, let $S_n(U) \subseteq \A^n$ denote the subwords of length $n$ of words from $U$. Then
\begin{equation} \label{eq: legal words}
\mc{L}^n(\sub) = \bigcup_{j=0}^P S_n(\sub^j(\A \cup \mc{L}^2(\sub)))
\end{equation}
where $P$ may be taken as any number with $|\sub^P(a)| \geq n$ for all $a \in \A$.
\end{lem}

\begin{proof}
Firstly, there exists some $P \in \N$ with $|\sub^P(a)| \geq n$ for all $a \in \A$. Indeed, consider
\[
\A_m = \{a \in \A \mid |\sub^m(a)| \leq n-1\} \subseteq \A.
\]
Then $\A_m = (|\,\cdot\,| \circ \sub^m)^{-1}\{1,\ldots,n-1\}$ and hence is compact. Since words grow under application of $\sub$, we have that $\A_m \supseteq \A_{m+1}$. If each $\A_m \neq \varnothing$, then by Cantor's intersection theorem there is some $x \in \A_m$ for each $m$. This contradicts the substitution growing on $x$. It follows that some $\A_P = \varnothing$ or, equivalently, $|\sub^P(a)| \geq n$ for each $a \in \A$.

We define
\[
Z =  \bigcup_{j=0}^\infty S_n(\sub^j(\A \cup \mc{L}^2(\sub))),
\]
which agrees with the union of the lemma except that we take an infinite union. Clearly $Z$ contains all length $n$ words generated by $\sub$. Moreover, since the substitution of a legal word is legal (Lemma \ref{lem: closure of language under subst and subwords}), we have that $Z \subseteq \mc{L}^n$. It follows that $\overline{Z} = \mc{L}^n$.

We claim that the above union defining \(Z\) does not change by truncating at $j=P$. Indeed, suppose that $v \in \mc{L}^n(\sub)$ with $v \triangleleft \sub^k(u)$ for $k > P$ and $u \in \A \cup \mc{L}^2(\sub)$. Then $v \triangleleft \sub^P(\sub^{k-P}(u))$. As any word of the form $\sub^P(a)$ has length at least $n$, it follows that there is some $1$ or $2$-letter subword $u' \triangleleft \sub^{k-P}(u)$ with $v \triangleleft \sub^P(u')$. Since substitutions and subwords of legal words are legal, we have that $u' \in \A \cup \mc{L}^2(\sub)$ and thus $v \in S_n(\sub^P(\A \cup \mc{L}^2(\sub)))$.

It follows that $Z$ is the finite union in the statement of the lemma. Since $\A \cup \mc{L}^2(\sub)$ is compact, the same follows for $Z$, which is thus closed. Hence $\mc{L}^n = \overline{Z} = Z$, as required.
\end{proof}

\begin{remark} \label{rem: two-letter legals}
Technically the $\A$ term of Equation \ref{eq: legal words} needs to be included in general, since it is possible that a letter does not appear in any $2$-letter legal word. However, in examples where the substitution grows without bound, if every legal letter also appears in the subshift (i.e., it satisfies the equivalent conditions of Proposition \ref{prop: language realised}) then it must appear in a legal $2$-letter word and thus we may simplify the statement to
\[
\mc{L}^n(\sub) = \bigcup_{j=0}^P S_n(\sub^j(\mc{L}^2(\sub))).
\]
In fact, it may be simplified further in this case. In Proposition \ref{prop: supertiling existence} we will see that every element of $X_\sub$ is the image under substitution of another, up to a shift. It easily follows that every legal $1$ or $2$-letter word is a subword of the substitute of another. Hence the above is a nested union and so we may write instead:
\[
\mc{L}^n(\sub) = S_n(\sub^P(\mc{L}^2(\sub))).
\]
\end{remark}

The following result shows that for every element of $X_\sub$, there is a corresponding `superword' decomposition (substitutive preimage up to an appropriate shift):

\begin{prop}\label{prop: supertiling existence}
Let $\sub \colon \A \to \A^+$ be a substitution with associated subshift $X_\sub$.
For every $w \in X_\sub$, there exists an element $x \in X_\sub$ and an integer $0 \leq i \leq |\sub(x_0)|-1$ such that $\sigma^i(\sub(x)) = w$.
\end{prop}

\begin{proof}
We present an analogue of the usual proof from the finite alphabet setting. Let $w \in X_\sub$ and $v_n = w_{[-n,n]}$. Since $v_n$ is legal there is a net of generated words converging to \(v_n\). Thus, for each \(n\) we have a directed set \(\varLambda_n\), superwords $\sub^{k(n,m)}(t(n,m))$ for $k(n,m) \in \N$ and $t(n,m) \in \A$ and subwords $u(n,m) \triangleleft \sub^{k(n,m)}(t(n,m))$ so that $u(n,m) \to v_n$ as $m \to \infty$ in $\varLambda_n$.

We may consider the subwords $u(n,m)$ as centred over the origin, so that convergence $u(n,m) \to v_n$ as $m \to \infty$ holds letter-wise, at each position. Similarly, we may position the subwords $q(n,m) \coloneqq \sub^{k(n,m) - 1}(t(n,m))$ over the origin so that $u(n,m) \triangleleft \sigma^{i(n,m)}\sub(q(n,m))$, for $0 \leq i(n,m) < \ell$, where $\ell$ is the number of letters in the substitute of the origin letter of $q(n,m)$. Indeed, we may shift $q(n,m)$ so that $u(n,m)_0$ is in the image of the origin letter of $q(n,m)$ after substitution. We may pass to subnets of each $\varLambda_n$ so that $i(n,m) = i_n$ is constant in $m$, and then restrict to values of $n \in \N_0$ so that all remaining $i_n = i$ are constant.

We thus construct subwords $q(n,m)$ so that letters at positions $[-n,n]$ of $\sigma^i \sub(q(n,m))$ converge to those of $w$. Using compactness, choose a subnet $S_0$ of $q(n,m)$ that converges at position $0$, say to $x_0$. By continuity we have that $0 \leq i \leq |\sub(x_0)|-1$. Similarly, we may choose a subnet $S_1$ of $S_0$ that also converges at positions $-1$ and $1$, say to letters $x_{-1}$ and $x_1$, respectively. Inductively define subnets in this way so that $S_n$ converges at all positions $j \in [-n,n]$ to letters  $x_j$.

By construction (and continuity), the bi-infinite element $w' = \{x_j\}_{j \in \Z}$ satisfies $\sigma^i(\sub(w')) = w$. We now need to show that every subword of $w'$ is legal. Since, by Lemma \ref{lem: closure of language under subst and subwords}, subwords of legal words are legal, it suffices to show that $w'_{[-n,n]}$ is legal for each $n$. By construction we have the net $S_n$, which converges letter-wise on positions $[-n,n]$ to $w_{[-n,n]}$. This net consists of subwords of superwords $\sub^{k(n, m) - 1}(t(n,m))$, so it follows that $w_{[-n,n]}$ is in the closure of the generated words and thus is legal, as required.
\end{proof}

\subsection{Primitivity}
The following definition is adapted from the work of Durand, Ormes and Petite \cite{DOP:self-induced}, modified so that the condition only needs to be checked for a single power $p$. This definition is essentially the same as the one given by Frank and Sadun \cite{PFS:fusion-ILC}, just in a slightly different setting.

\begin{definition}\label{DEF:prim}
Let $\sub \colon \A \to \A^+$ be a substitution.
We say $\sub$ is \emph{primitive} if, for every non-empty open set $U \subseteq \A$, there exists a $p = p(U) \geq 0$ such that, for all $a \in \A$, some letter of $\sub^p(a)$ is in $U$.
\end{definition}

This is easily seen to be equivalent to the definition in \cite{DOP:self-induced}:

\begin{lem}\label{LEM:prim-higher-powers}
The substitution $\sub$ is primitive if and only if, for every non-empty open set $U \subseteq \A$, there exists a $p = p(U) \geq 0$ such that, for all $a \in \A$ and for all $j \geq p$, some letter of $\sub^j(a)$ is in $U$.
\end{lem}

\begin{proof}
Suppose $\sub$ is primitive and let $U \subseteq \A$ be a given open set. Let $p$ be such that for all $a \in \A$, some letter of $\sub^p(a)$ is in $U$ and let $j \geq p$. Let $a$ be a letter in $\A$ and let $b$ be the first letter of $\sub^{j-p}(a)$. Then $\sub^p(b) \triangleleft \sub^p(\sub^{j-p}(a)) = \sub^j(a)$. By primitivity, $\sub^p(b) \triangleleft \sub^j(a)$ contains a letter in $U$. The other direction is trivial.
\end{proof}

The following is an analogue of Proposition \ref{prop: rep all words} on repetitivity, in the sense that it gives an equivalent primitivity condition described by finding \emph{all} letters of $\A$ within \emph{all} $p$-superwords, up to some given tolerance:

\begin{prop} \label{prop: prim all letters}
The following are equivalent:
\begin{enumerate}
	\item $\sub$ is primitive;
	\item for every finite collection $\mc{U} = \{U_i\}_{i=1}^\ell$ of open subsets $U_i \subseteq \A$, there exists a $p = p(\mc{U}) \geq 0$ such that for all $a \in \A$ and $1 \leq i \leq \ell$, some letter of $\sub^p(a)$ is in $U_i$;
	\item for every open $U \subseteq \A \times \A$ containing the diagonal, there is some $p \geq 0$ so that, for all $a$, $b \in \A$, we have that $b' \triangleleft \sub^p(a)$ for some $b' \in \A$ with $(b,b') \in U$.
\end{enumerate}
\end{prop}

\begin{proof}
(1 $\implies$ 2): 
This follows trivially from Lemma \ref{LEM:prim-higher-powers}.

(2 $\implies$ 3):
Let $U \subseteq \A \times \A$ be open and contain the diagonal and take $V \subseteq \A \times \A$ with $V = V^{-1}$ and $V \circ V \subseteq U$. Similar to before we take the open cover $\{V_a\}_{a \in \A}$ of $\A$, where $a' \in V_a$ if $(a,a') \in V$. By compactness we may find a finite subcover $\{V_{a_i}\}_{i=1}^\ell$. Take $p \geq 0$ according to property 2 and let $a$, $b \in \A$. We have that $a \in V_{a_i}$ for some $i$, so $(a_i,a) \in V$ and thus also $(a,a_i) \in V$. By property 2, we have that $a' \triangleleft \sub^p(b)$ for some $a' \in V_{a_i}$ so that $(a_i,a') \in V$. We have that $(a,a') = (a,a_i) \circ (a_i,a') \in V \circ V \subseteq U$, as required.

(3 $\implies$ 1):
Let $U \subseteq \A$ be an arbitrary, non-empty open subset and let $V = (\A \times \A) \setminus (\{a\} \times (\A \setminus U))$, where $a \in U$ is arbitrary. Then $V \subseteq \A \times \A$ is open and contains the diagonal. Take $p \geq 0$ according to property 3 and let $b \in \A$ be arbitrary. Then there is some $a' \triangleleft \sub^p(b)$ with $(a,a') \in V$. The latter implies that $a' \in U$, as required.
\end{proof}

\begin{remark} \label{rem: prim metric}
Similar to Remark \ref{rem: rep metric}, Property 3 above has an intuitive interpretation when $\A$ is a metric space: for all $\epsilon > 0$ there is some $p = p(\epsilon) \geq 0$ so that, for all $a \in \A$, every letter $b \in \A$ is $\epsilon$-close to a letter of $\sub^p(a)$.
\end{remark}

In the following result (and elsewhere), to avoid trivialities, we assume that $\sub$ is not the trivial substitution $\sub(a) = a$ on a one-letter alphabet, which technically would be primitive according to Definition \ref{DEF:prim}.

\begin{prop}\label{PROP:prim-growing}
Let $\sub \colon \A \to \A^+$ be a primitive substitution. Then for all $n \geq 1$, there exists a $p \geq 0$ such that for all $a \in \A$, $|\sub^p(a)| \geq n$.
Consequently, for all $a \in \A$,  one has ${\displaystyle \lim_{i \to \infty}|\sub^i(a)| = \infty}$.
\end{prop}

\begin{proof}
If \(\A\) is a singleton then, since \(\sub\) is not the trivial substitution of length \(1\), \(\sub(a)\) is given by \(n\) copies of \(a\) for the unique \(a \in \A\), so that \(|\sub^k(a)| = 2^k\). Otherwise, \(\A\) contains at least two distinct points and thus, since \(\A\) is Hausdorff, two non-empty disjoint open subsets \(U_1\), \(U_2 \subset \A\). By Proposition \ref{prop: prim all letters}, there exists \(p \in \N\) so that \(\sub^p(a)\) contains a letter in both \(U_1\) and \(U_2\) for all \(a \in \A\), in particular \(|\sub^p(a)| \geq 2\) and hence \(|\sub^{kp}(a)| \geq 2^k\).
\end{proof}

As a consequence of Corollary \ref{COR:non-empty}, $X_\sub$ must then also be non-empty.

\begin{coro}\label{COR:irred-non-empty}
If $\sub$ is primitive, then $X_\sub \neq \varnothing$ is a subshift.
\end{coro}

It is clear that Property 1 of Proposition \ref{prop: language realised} holds for primitive substitutions. Indeed, take an arbitrary non-empty open subset $U \subseteq \A$ and $a \in \A$. For some $n \geq 1$ we have \(|\sub^n| \geq 3\), so \(\sub^n(a) = xvy\) for \(x\), \(y \in \A\) and \(v \in \A^+\). Consider $j \geq p$ and $\sub^{n+j}(a) = \sub^j(x) \sub^j(v) \sub^j(y)$. By primitivity, $\sub^j(v)$ contains a letter $u \in U$. By taking $j$ sufficiently large, we may ensure that $|\sub^j(x)|$, $|\sub^j(y)| \geq n$, so that there is a word $v' \triangleleft \sub^j(a)$ of length $2n+1$ centred around $u$ and is generated by $\sub$, as required. Since all letters grow without bound, Remark \ref{rem: two-letter legals} has the following consequence.

\begin{coro} \label{cor: primitive 2 letters generate}
Let $\sub$ be a primitive substitution. Then for all $n \in \N_0$,
\[
\mc{L}^n(\sub) = \mc{L}^n(X_\sub) = S_n(\sub^P(\mc{L}^2(\sub))),
\]
where $S_n(U)$ denotes the $n$-letter subwords of words from $U \subseteq \A^\ast$ and $P$ is any number satisfying $|\sub^P(a)| \geq n$ for all $a \in \A$.
\end{coro}

For any non-empty open set of legal \emph{words} $U$ (rather than just a set of letters), it will be useful to find a power $p$ such that every $p$-superword $\sub^p(a)$ contains a word in $U$ as a subword.

\begin{prop}\label{PROP:uniform-legal-word-generation}
Let $\sub$ be primitive and $U \subseteq \mc{L}^n(\sub)$ be non-empty and open. Then there exists a power $p \geq 0$ such that, for every $a \in \A$, there is a subword of $\sub^p(a)$ in $U$.
\end{prop}

\begin{proof}
Since $U$ is open, we can find a generated word $u \in U$ so that $u = p_i(\sub^k(b))$, where $p_i$ is the projection to the length $n$ subword beginning at the index $i$. Again using that $U$ is open, and continuity of $p_i \circ \sub^k$ (which is well defined on an open subset of $\A$), we have an open subset $V \subseteq \A$ for which $p_i \circ \sub^k(v) \in U$ for all $v \in V$. By primitivity, there is some $p \geq 0$ so that $\sub^p(a)$ contains an element of $V$ for all $a \in \A$. Then, for arbitrary $a \in \A$, we have $\sub^{p+k}(a)$ contains a word of $\sub^k(V)$, whose $i$th subword of length $n$ belongs to $U$, as required.
\end{proof}

Primitivity of $\sub$ is a strong condition and gives us that $X_\sub$ is minimal, hence every element has a dense orbit and is repetitive.

\begin{theorem}\label{THM:primitive-implies-minimal}
Let $\sub \colon \A \to \A^+$ be a substitution with associated subshift $X_\sub$. If $\sub$ is primitive, then $X_\sub$ is a minimal subshift.
\end{theorem}

\begin{proof}
Let \(U \subseteq \mc{L}^n(\sub)\) be any open subset. By Proposition \ref{PROP:uniform-legal-word-generation}, there is some \(p \geq 0\) such that, for all \(a \in \A\), there is a subword of \(\sub^p(a)\) in \(U\). Take any \(w \in X_\sub\). By Proposition \ref{prop: supertiling existence}, there exists some \(w' \in X_\sub\) such that \(\sub^p(w') = \sigma^i(w)\), where \(0 \leq i \leq |\sub^p|-1\). So the word \(\sub^p(w_0)\), and hence a word in \(U\), appears within a uniformly bounded distance of the origin in \(w\) and hence \(w\) is repetitive. Thus, by Corollary \ref{CORO:minimal-equiv}, \(X_w\) is minimal. It is clear that \(X_w = X_\sub\), since the above shows that any non-empty open \(U \subseteq \mc{L}^n(\sub)\) contains a subword of any \(w \in X_\sub\).
\end{proof}

It is interesting to note, therefore, that by the above and Propositions \ref{PROP:do-iff-tt} and \ref{PROP:minimal-rep}, if the alphabet $\A$ is non-separable, then all substitutions on $\A$ are non-primitive.

\begin{coro}\label{COR:prim-implies-sep}
If $\sub$ is primitive, then $\A$ is separable.
\end{coro}

\begin{remark}
In the case when $\A$ is zero-dimensional and metrisable, Theorem~\ref{THM:primitive-implies-minimal} is already known  under an equivalent assumption \cite{DOP:self-induced}. In this work, it was shown that the family of primitive substitution subshifts is in one-to-one correspondence with minimal self-induced Cantor systems. 
The self-induced property implies the dichotomy result that, for this family, the topological entropy $h_{\text{top}}(X_{\sub})$ is either $0$ or $\infty$; see \cite[Prop.~3]{DOP:self-induced}. 
\end{remark}

\subsection{Realising set substitutions with compact alphabets}
Suppose that $S$ is a set, and we have a function $\sub \colon S \to S^+$. It is natural to ask whether we may extend $\sub$ to a continuous substitution on a compact Hausdorff alphabet. In the finite case this is clear, by using the discrete topology on $S$. Combinatorial substitutions with $S$ countable have already been studied, for example in the work of Ferenczi \cite{Ferenczi}. As we have seen, such a substitution on a compact alphabet would necessarily be of bounded length. Given this necessary restriction, we may in fact always find such a compactification:

\begin{theorem}
Let $\sub \colon S \to S^+$ with $|\sub| < \infty$. Then there is a compact Hausdorff alphabet $\A$, a dense inclusion $\iota \colon S \hookrightarrow \A$ and a substitution $\overline{\sub} \colon \A \to \A^+$ so that $\overline{\sub}(\iota a) = \iota \sub(a)$ for all $a \in S$. Moreover, we may choose $\iota$ so that $\iota(S)$ is the set of isolated points of $\A$.
\end{theorem}

\begin{proof}
Since $L \coloneqq |\sub| < \infty$, we may write $\sub \colon S \to S \sqcup S^2 \sqcup \cdots \sqcup S^L$. Equip $S$ with the discrete topology and let $\iota \colon S \hookrightarrow \beta S$ be its Stone--\v{C}ech compactification. Then $\iota$ naturally embeds $S^n$ into $(\beta S)^n$ for each $n \in \N$, so we may regard $\sub$ as a map
\[
\sub \colon S \to \beta S \sqcup (\beta S)^2 \sqcup \cdots \sqcup (\beta S)^L.
\]
This is continuous, since all maps from discrete spaces are. As a finite disjoint union of compact Hausdorff spaces, we have that $\beta S \sqcup (\beta S)^2 \sqcup \cdots \sqcup (\beta S)^L$ is compact Hausdorff. Hence, by the universal property of the Stone--\v{C}ech compactification, we may extend $\sub$ to a continuous map
\[
\overline{\sub} \colon \beta S \to \beta S \sqcup (\beta S)^2 \sqcup \cdots \sqcup (\beta S)^L.
\]
This is nothing other than a continuous substitution of maximal length $L$ on the compact, Hausdorff alphabet $\beta S$. This compactification makes $\iota(S) \subseteq \beta S$ a dense subspace of isolated points, as required.
\end{proof}

Note that the approach in the above proof, using the Stone--\v{C}ech compactification, allows us to compactify any continuous substitution on a topologised alphabet that is Hausdorff (which will contain the original alphabet as a dense subspace). Although the above answers our question on realising arbitrary (infinite) set substitutions with continuous, compact substitutions, it is far from ideal for typical examples, since the Stone--\v{C}ech compactification is usually unwieldy. For example, $\beta \N_0$ has cardinality $2^\mathfrak{c}$, where $\mathfrak{c}$ is the cardinality of the continuum. In practice there are often more obvious compactifications. For example, the substitution $\sub\colon n \mapsto 0 \ n\!+\!1$ on $\N_0$ may be clearly extended to the one-point compactification, by defining $\infty \mapsto 0 \ \infty$, and it is easily checked that this is primitive; see \cite{Ferenczi,PFS:fusion-ILC}.

\section{The substitution operator and natural length functions}\label{SEC:tile lengths}

For a substitution \(\sub\), a continuous and non-zero \(\ell \colon \A \to \R_{\geq 0}\) is called a \emph{natural length function} for \(\sub\) if, for all $a \in \A$,
\begin{equation} \label{eq: natural length}
\sum_{b \triangleleft \sub(a)} \ell(b) = \lambda \cdot \ell(a),
\end{equation}
where the notation $\sum_{b \triangleleft \sub(a)}$ enumerates over each $b$-entry of $\sub(a)$, including multiplicities. Here, \(\lambda \geq 0\) is called the \emph{inflation factor}. In practice, we really want \(\lambda > 1\) in order to represent an inflation, and \(\ell(a) > 0\) for all \(a \in \A\) in order to represent tiles as closed intervals with positive length. For a primitive (or, more generally, irreducible) substitution, both of these properties derive from Equation \ref{eq: natural length} (in addition to non-negativity and non-triviality of \(\ell\)), see Theorem \ref{thm: irreducible => unique nlf}. However, the results will be easier to develop in the above slightly generalised setting.

Geometrically, we thus have an associated inflation rule where all letters $a \in \A$ are assigned to a tile of length $\ell(a)$, depending continuously on their location in $\A$, so that under substitution these tiles are inflated by a factor of $\lambda$ and then perfectly dissected into other tiles, each with original lengths defined by \(\ell\). Fixed points of the substitution $\sub$ (if they exist) then give rise to self-similar tilings of $\R$.

\subsection{The substitution operator}

\begin{definition}
We let $E = C(\A)$ denote the Banach space of continuous functions $f \colon \A \to \R$, which is equipped with norm $\|f\| \coloneqq \sup\{|f(a)| \mid a \in \A\}$. Call an element $f \in E$ \emph{positive} (resp.\ \emph{strictly positive}) if $f(a) \geq 0$ (resp.\ $f(a) > 0$) for all $a \in \A$. Let $K$ denote the \emph{positive cone} of positive elements, and $K_{>0}$ denote the set of strictly positive elements. Given $f$, $g \in E$, we write $f \leq g$ if $f(a) \leq g(a)$ for all $a \in A$ i.e., $g - f \in K$. Given this partial order, $E$ is a Banach lattice; see \cite{SchaeferBook}. 
 We have the \emph{order interval} $[f,g] = \{z \in E \mid f \leq z \leq g\}$.

The dual of $E$ is the Banach lattice $E'$ of continuous homomorphisms $\phi \colon E \to \R$, which has norm $\|\phi\| \coloneqq \sup_{\|f\|\leq 1} |\phi(f)|$. We sometimes write $\langle \phi,f \rangle \coloneqq \phi(f)$. For an operator $M \colon E \to E$, the dual operator $M'$ is defined by $(M'\phi)(f) \coloneqq \langle \phi,Mf \rangle$. The dual cone $K' \subset E'$ is defined as the set of $\phi \in E'$ for which $\phi(f) \geq 0$ whenever $f \in K$.
\end{definition}

By the Riesz--Markov--Kakutani representation theorem, there is a bijection between continuous linear functionals $\phi$ on $E = C(\A)$ and regular signed finite Borel measures $\mu$ on $\A$, where we identify
\[
\langle \phi , f \rangle \longleftrightarrow \int_{\A} f \mathrm{d} \mu,
\]
for continuous `test functions' $f \in E$. Positive functionals on \(E\) (which are necessarily continuous) i.e., the elements of \(K'\), may be identified with the (unsigned) regular finite Borel measures on \(\A\).

\begin{remark} \label{rem: regular measures}
Because of the above correspondence, all measures here will be assumed to be regular. For example, when we speak of \emph{unique ergodicity} in Section~\ref{SEC:ergodicity}, it is meant that there is a unique \emph{regular} Borel probability measure. Note that a Borel probability measure on a compact Hausdorff alphabet $\A$ is automatically regular in most cases of interest, including when $\A$ has a countable base for its topology, or when $\A$ is metrisable.
\end{remark}

Given any `length function' $f \in E$ (not necessarily one that satisfies Equation~\eqref{eq: natural length}), it is still meaningful to talk of the `length' of the substitute of a tile $a \in \A$, by summing the lengths of the letters of $\sub(a)$. So, we define
\begin{equation} \label{eq: substitution operator}
M_{\sub}:=M \colon E \to E, \qquad (M f)(a) \coloneqq \sum_{b \triangleleft \sub(a)} f(b), 
\end{equation}
where $\sub(a)$ is seen as a multiset (so we sum including multiplicities as in Equation~\eqref{eq: natural length}).

To find a natural length function is thus to find a positive eigenvector of $M$, which we call the \emph{substitution operator}. Note that $M$ only depends on the substitution considered as a multi-valued map, i.e., it does not depend on the order of letters of each superword.

\begin{remark}
We note that an analogous operator called the \emph{transition map} was introduced in \cite{PFS:fusion, PFS:fusion-ILC} in the geometric setting for fusion tilings. Before this, such an operator has already been used to determine statistical properties of generalised pinwheel tilings in \cite{S:pinwheel}.
\end{remark}

\begin{example}
Let $\A$ be finite and with the discrete topology. A basis for $E$ is given by the indicator functions $\bbo_a$ for $a \in \A$, where $\bbo_a(a) = 1$ and $\bbo_a(b) = 0$ for $a \neq b \in \A$. Then
\[
(M \bbo_a)(b) = \sum_{c \triangleleft \sub(b)} \bbo_a(c) = \sum_{a \triangleleft \sub(b)} \bbo_a(a) = M_{ab}
\]
where $M_{ab}$ is the number of occurrences of $a$ in $\sub(b)$. So we may write
\[
M \bbo_a = \sum_{b \in \A} M_{ab} \cdot \bbo_b
\]
and hence, with respect to this basis, $M$ is represented by the matrix $(M)_{ba}$. It follows that the substitution operator \(M\) is represented by the transpose of the \emph{Abelianisation} or \emph{substitution matrix} of the substitution in the finite letter setting.
\end{example}

It is easy to see that $M^n$ is the operator that replaces a length function $f$ with its sum over $n$-superwords:

\begin{lem} \label{lem: powers of operator}
For $n \in \N$ we have $\displaystyle (M^n f)(a) = \sum_{b \triangleleft \sub^n(a)} f(b)$. That is, $M_{\sub^n}=\big(M_{\sub}\big)^n$.

\end{lem}

\begin{proof}
The above is true for $n=1$ by definition of $M$. Suppose that it holds for $n \in \N$. Then
\[
(M^{n+1} f)(a) = M \left(b \mapsto \sum_{c \triangleleft \sub^n(b)} f(c) \right)(a) = \sum_{b \triangleleft \sub(a)}\sum_{c \triangleleft \sub^n(b)} f(c) = \sum_{c \triangleleft \sub^{n+1}(a)} f(c),
\]
so the result also holds for $n+1$ and thus for all $k \in \N$ by induction.
\end{proof}

Clearly $M$ is linear and \emph{positive}, which is to say that $M(K) \subseteq K$. Moreover $M$ is a bounded operator, since, for all $a \in \A$, we have $|M(f)(a)| = |\sum_{b \triangleleft \sub(a)} f(b)| \leq \max|\sub(a)| \cdot \max|f| = |\sub| \cdot \|f\|$, where the length \(|\sub|\) of the substitution is bounded by Proposition \ref{PROP:const-length}. In fact, we have the following formula for the operator norm of $M^n$. Let $\bbo$ be the constant function $a \mapsto 1$.

\begin{coro} \label{coro: operator norm}
For $n \in \N$, the operator norm of $M^n$ is given by
\[
\|M^n\| = \|M^n(\bbo)\| = |\sub^n| .
\]
\end{coro}

\begin{proof}
As $M^n$ is a positive operator, for $\|f\| \leq 1$, the norm $\|M^n f\|$ is maximised by the constant function $\bbo$, for which $(M^n \bbo)(a) = |\sub^n(a)|$ by Lemma \ref{lem: powers of operator}.
\end{proof}

In the finite setting, the existence of a positive non-zero length function follows from Perron--Frobenius theory. Surprisingly, this does not hold for all substitutions on compact alphabets:

\begin{example} \label{exp: non-growing, no length}
Consider again the substitution of Example \ref{exp: non-growing} on the alphabet $\A = \N_{\infty} \times \N_{\infty}$, for which $(n,m) \mapsto (n,m-1)$ for $m > 0$, $(n,0) \mapsto (n-1,n)(0,n)$ for $n > 0$ and $(0,0) \mapsto (0,0)$. It is not hard to show that this substitution has no continuous, positive, non-zero length function \(\ell\). Dropping continuity of \(\ell\), one finds that \(M \ell = \lambda \ell\) implies that \(\lambda = 1\) and
\[
\ell(n,m) =
\begin{cases}
(n+1)\alpha_1,       & \mbox{if } n, m \neq \infty, \\
\alpha_2,            & \mbox{if } n = \infty, m = 0, \\
\alpha_3,            & \mbox{if } n = 0, m = \infty, \\
\alpha_2 - \alpha_3, & \mbox{if } n = m = \infty, \\
\beta_n,             & \mbox{if } m = \infty,
\end{cases}
\]
where \(\alpha_1\), \(\alpha_2 \geq \alpha_3\) and \(\beta_n\) are arbitrary non-negative numbers, with at least one strictly positive. Note that \(\ell\) is unbounded if \(\alpha_1 \neq 0\), and if \(\alpha_1 = 0\) then \(\ell = 0\) on a dense subset of \(\A\).
\end{example}

\begin{example} \label{exp: growing, no length}
The substitution in the above example was primarily motivated by its pathological non-growth property, perhaps making it less surprising that it does not have a continuous length function. However, it can be easily modified to be growing in length on all letters but still admitting no continuous length function. We double up the substitution on \((0,0)\) and triple it on all other letters: on \(\N_0 \times \N_0\) we define
\[
\sub \colon
\left\{
\begin{array}{rll}
(0,0)\ \mapsto & (0,0) (0,0),           &                    \\
(n,m)\ \mapsto & (n,m-1)(n,m-1)(n,m-1)  & \text{ for } m > 0, \\
(n,0)\ \mapsto & (n-1,n)(n-1,n)(n-1,n)(0,n)(0,n)(0,n) & \text{ otherwise,}
\end{array}
\right.
\]

which uniquely defines the substitution \(\sub\) on \(\A = \N_\infty \times \N_\infty\) by continuity. Since all letters in \(\N_0 \times \N_0\) eventually map to \((0,0)\), it is clear that \(\ell(0,0) \neq 0\), otherwise \(\ell = 0\) by continuity, so also \(\ell(n,m) \neq 0\) for all \((n,m) \in \N_0 \times \N_0\). Since \((0,0) \mapsto (0,0)(0,0)\), the inflation constant must be \(\lambda = 2\). But for \((n,m) \in \N_0 \times \N_0\), with \(m>0\), we have
\[
2 \ell(n,m+1) = \lambda \ell(n,m+1) = \ell(n,m) + \ell(n,m) + \ell(n,m) \implies \ell(n,m+1) = \frac{3}{2}\ell(n,m) \neq 0,
\]
so \(\ell\) is unbounded, contradicting continuity. Thus, this substitution admits no continuous and non-zero length function.
\end{example}

Although the above shows that not all substitutions have a natural length function, it seems reasonable to conjecture this for primitive substitutions:

\begin{question}
Let $\sub$ be a primitive substitution on a compact, Hausdorff alphabet. Does $\sub$ admit a continuous natural length function?
\end{question}

We are not currently able to resolve the above question in full generality. However, we will find some conditions under which this holds. In the remainder of this subsection, we will consider general properties of the Banach space $E$, the operator $M$ and implications of primitivity.

\begin{remark} \label{rem: tychonoff}
Because \(\A\) is compact Hausdorff, it is Tychonoff, i.e., Hausdorff and completely regular. This means that for any closed set \(C \subseteq \A\) and \(a \notin C\), there exists a continuous function \(f \colon \A \to [0,1]\) with \(f(a) = 1\) and \(f(C) = \{0\}\), which will be used in several proofs below.
\end{remark}

\begin{prop} \label{prop: primitive <=> st. positive iterates}
The substitution $\sub$ is primitive if and only if, for all non-zero $f \in K$, there exists some $p \in \N$ with $M^p(f)\in K_{>0}$.
\end{prop}

\begin{proof}
Let $\sub$ be primitive and $0 \neq f \geq 0$. Let $U = f^{-1}(0,\infty)$, which is open and non-empty. By primitivity, there exists some $p \in \N$ so that for any $a \in \A$ we have that $\sub^p(a)$ contains some $b \in U$, which is to say that $f(b) > 0$. Since $f \geq 0$ and $f(b) > 0$ for some $b \in \sub^n(a)$, it follows from Lemma \ref{lem: powers of operator} that $(M^p f)(a) > 0$ for all $a \in \A$.

Conversely, suppose the given condition on $M$ holds and that $U \subset \A$ is open and non-empty. Let \(C = \A \setminus U\), \(a_0 \in U\) be arbitrary and \(f \colon \A \to [0,1]\) be continuous with \(f(a_0) = 1\) and \(f(C) = \{0\}\) (Remark \ref{rem: tychonoff}). By assumption there exists some $p \in \N$ so that $M^p(f) > 0$. By Lemma \ref{lem: powers of operator}, for every $a \in \A$ we have that
\[
\sum_{b \triangleleft \sub^p(a)} f(b) > 0.
\]
Since $f \geq 0$ and $f(b) = 0$ for all $b \notin U$, this implies that $b \in U$ for some $b \triangleleft \sub^p(a)$, as required.
\end{proof}

We recall properties of positive cones for general Banach lattices from \cite{Schaefer60}. A positive cone $K \subset E$ is called \emph{proper} if $K \cap (-K) = \{0\}$, \emph{generating} if $E = K-K$ and \emph{normal} if $\|x+y\| \geq \|y\|$ for all $x$, $y \in K$. The ordered Banach space $(E,K)$ is said to have the \emph{decomposition property} \cite{Abdelaziz} if, for all $x$, $y$ and $z \in K$ with $z \leq x+y$, there exist $b_1$, $b_2 \in K$ with $z = b_1 + b_2$ and $b_1 \leq x$, $b_2 \leq y$. 

\begin{definition}
A subset $A \subset E$ is called \emph{total} if the linear span of $A$ is dense in $E$.
\end{definition}

\begin{definition} \label{def: qi}
A point $f \in E$ is called \emph{quasi-interior} to $K$ if the order interval $[0,f]$ is a total subset of $E$.
\end{definition}

For the next results, we go back to our setting where $E=C(\A)$ and $K$ is the positive cone of non-negative functions. The proof of the following is routine, so we omit it.

\begin{lem} \label{lem: cone properties}
We have that $K \subset E$ is a closed, proper, generating and normal cone, with interior $\mathrm{int}(K) = K_{>0}$. The Banach space $(E,K)$ has the decomposition property.
\end{lem}

Every interior point of $K$ is quasi-interior \cite{Schaefer60}, and the converse holds if $\mathrm{int}(K) \neq 0$, which is clearly the case here by Lemma \ref{lem: cone properties}. For a proof of the following lemma, we refer the reader to \cite{SchaeferBook}.

\begin{lem} \label{lem: qi <=> int}
A point $f \in E$ is quasi-interior to $K$ if and only if $f \in \mathrm{int}(K)$. 
\end{lem}

\begin{coro} \label{cor: primitive <=> qi iterates}
The substitution $\sub$ is primitive if and only if, for every non-zero $f \in K$, one has that $M^p(f)$ is quasi-interior to $K$ for some $p \in \N$.
\end{coro}

We recall from \cite{Schaefer60} that an operator $M$ is \emph{strongly positive} if for each $0 \neq x \in K$, there is some $n = n(x)\in \N$ such that $M^n(x)$ is quasi-interior to $K$. Then Proposition~\ref{prop: primitive <=> st. positive iterates} may be restated as follows.

\begin{coro} \label{cor: strongly positive}
The substitution $\sub$ is primitive if and only if $M$ is strongly positive.
\end{coro}

The strongly positive condition implies the weaker property of the operator $M$ (and every power \(M^n\)) being \emph{(ideal) irreducible}, meaning that, for every ideal \(I\) of \(E\) with \(M(I) \subseteq I\), we have that \(I=\{0\}\) or \(I=E\). For the notion of an ideal of a Banach lattice, we refer the reader to \cite[II.2]{SchaeferBook}, although note \cite[III.1 Exp.~1]{SchaeferBook} that for \(E = C(X)\), with \(X\) a compact Hausdorff space, we have a bijective correspondence
\[
\text{closed subsets } C \text{ of } X \longleftrightarrow \text{ideals } I_C = \{f \in E \mid f(C) = \{0\} \}.
\]
As we will see shortly, irreducibility is a powerful property for ensuring essential uniqueness and positivity of natural length functions.

\begin{definition} \label{def: irreducible sub}
Let $\sub$ be a substitution on a compact alphabet $\mathcal{A}$. We say that $\sub$ is \emph{irreducible} if it cannot be restricted to a strictly smaller closed and non-empty subset of $\A$, that is, there does not exist a non-empty closed $C\subsetneq \A$ such that $\sub(c)$ contains only letters in $C$, for all $c\in C$.
\end{definition}

\begin{prop} \label{prop: irreducible M}
A substitution $\sub$ is irreducible if and only if the substitution operator $M$ is irreducible.
\end{prop}

\begin{proof}
Suppose that $\sub(c) \subseteq C^+$ for all \(c \in C\), for some closed subset \(C \subseteq \A\). Let \(f \in I_C\), so that \(f(c) = 0\) for all \(c \in C\). Then \((Mf)(c) = f(c_1) + \cdots + f(c_n) = 0\), where \(c_1 \cdots c_n = \sub(c) \in C^+\) for \(c \in C\). Thus \(M(I_C) \subseteq I_C\). If \(M\) is irreducible, we must have that \(I_C = \{0\}\) or \(E\), so that \(C = \A\) or \(\varnothing\), as required.

Conversely, suppose that $M(I_C) \subseteq I_C\). We claim that \(\sub\) may be restricted to \(C\). Indeed, if we have $b \triangleleft M^n(c)$, for $c \in C$ but $b \in \A \setminus C$, then consider a function \(f \colon \A \to [0,1]\) with \(f(c) = 0\) for all $c \in C$ and \(f(b) = 1\) (Remark \ref{rem: tychonoff}). Then $f \in I_C$ but $M(f) \notin I_C$, since \((Mf)(c) = f(c_1) + \cdots + f(c_n) \geq 1\), where \(c_1 \cdots c_n = \sub(c)\) contains a copy of \(b\). It follows that \(\sub\) restricts to \(C\). If \(\sub\) is irreducible, it then follows that \(I_C = \{0\}\) or \(E\), as required.
\end{proof}

\begin{remark}
If \(\sub\) is primitive, then clearly \(\sub\) is irreducible. Likewise, since powers of primitive substitutions are primitive, we see that each \(\sub^k\) is irreducible. Note that Proposition \ref{prop: irreducible M} is consistent with the finite alphabet setting where a substitution is sometimes called irreducible if for every pair of letters $a,b\in\A$, there exists a power $k$ such that $b$ is in $\sub^{k}(a)$, which is equivalent to the substitution matrix being irreducible. 
\end{remark}

\begin{prop} \label{prop: reduction to irreducible}
Every substitution \(\sub\) restricts to an irreducible substitution on a non-empty, closed sub-alphabet.
\end{prop}

\begin{proof}
This is analogous the the proof that every dynamical system has a minimal subsystem. Consider the family \(\Pi\) of non-empty closed sub-alphabets \(C \subseteq \A\), which are also closed under substitution. Then \(\Pi \neq \varnothing\) (since \(\A \in \Pi\)) and partially ordered under set inclusion. Moreover, every chain \((C_i)\) has a lower bound, namely \(C = \bigcap C_i\). Indeed, \(C\) is closed and non-empty by Cantor's intersection theorem. If \(c \in C\) then \(c\) belongs to each \(C_i\); since these are each closed under substitution, it must hold that \(\sub(c) = c_1 \cdots c_n\) with each \(c_j\) belonging to every \(C_i\) and thus to \(C\), so \(C\) is also closed under substitution. Hence, \(\Pi\) has a minimal element \( C \subseteq \A\) by Zorn's Lemma. But this means that \(C\) is a closed alphabet and \(\sub\) is irreducible over \(C\), as required.
\end{proof}

\subsection{The spectrum of \texorpdfstring{$M$}{}}
The spectrum of the operator $M$ is the set
\[
\sigma(M)=\left\{\lambda\in\C\mid \lambda\mathbb{I}-M \text{ is not invertible}\right\}.
\]
We call $R(\lambda)=(\lambda\mathbb{I}-M)^{-1}$ the resolvent operator, which is a holomorphic (operator-valued) function on $\C\setminus \sigma(M)$. We call $\lambda_0\in\sigma(M)$ a \emph{pole of the resolvent} (of order $p \in \N$) if there exists a punctured neighbourhood of $\lambda_0$ such that the  coefficient corresponding to $-p$ of the Laurent expansion of $R(\lambda)$ is non-zero and all other coefficients for $n<-p$ are zero \cite[Sec.~8]{Taylor-FA}. A pole is called simple if it is of order $p=1$. It is often very useful to know that \(\lambda \in \sigma(M)\) is a pole of the resolvent; in particular, this implies that \(\lambda\) is an eigenvalue (see, for instance \cite[Thm.~5.8-A]{Taylor-FA}):

\begin{lem} \label{lem: pole => eigenvalue}
If \(\lambda \in \sigma(M)\) is a pole of the resolvent then \(\lambda\) is isolated in \(\sigma(M)\) and is an eigenvalue of \(M\).
\end{lem}

For the rest of the paper, we let $r = r(M)$ denote the \emph{spectral radius} of $M$, defined as
$r = \sup\{|\lambda| \mid \lambda \in \sigma(M)\}$. This is the natural generalisation of the Perron--Frobenius leading eigenvalue in the finite alphabet setting.  Gelfand's formula gives us that
\begin{equation} \label{eq: gelfand}
r = \lim_{n \to \infty} \sqrt[n]{\|M^n\|}\ \text{ and }\ r \leq \sqrt[n]{\|M^n\|} \ \text{ for all } n \in \N.
\end{equation}

\begin{lem} \label{lem: spectral bounds}
We have the following for all $n \in \N$:
\[
\min_{a \in \A} |\sub^n(a)| \leq r^n \leq \max_{b \in \A}|\sub^n(b)|.
\]
\end{lem}

\begin{proof}
The upper bound follows from Equation~\eqref{eq: gelfand} and Corollary \ref{coro: operator norm}. For the lower bound, suppose that $c \coloneqq \min_{a \in \A} |\sub^n(a)|$. We claim that $c^k \leq \|M^{kn}\|$ for all $k \in \N$, which is to say that each $kn$-superword contains at least $c^k$ letters. By definition this holds for $k=1$. Suppose it holds for some given $k$. We have $|\sub^{(k+1)n}(a)| = |\sub^n(\sub^{kn}(a))|$ for all $a \in \A$ so, since $\sub^{kn}$ contains at least $c^k$ letters, each of which substitutes under $\sub^n$ to at least $c$ letters (by the definition of $c$), we obtain the claimed upper bound for all $k$ by induction. Applying Equation \ref{eq: gelfand}:
\[
r = \lim_{k \to \infty} \sqrt[k]{\|M^k\|} = \lim_{k \to \infty} \sqrt[kn]{\|M^{kn}\|} \geq \sqrt[kn]{c^k} = \sqrt[n]{c}.
\]
\end{proof}

Note, in particular, that $r \geq 1$ for any substitution. The \emph{peripheral spectrum} of $M$ is the subset of the spectrum of maximum modulus:
\[
\sigma_\mathrm{per}(M) = \{\lambda \in \sigma(M) \mid |\lambda| = r\}.
\]
We say that $\lambda \in \sigma(M)$ is in the \emph{point spectrum} if $\lambda$ is an eigenvalue of $M$ and let $\sigma^\mathrm{p}_\mathrm{per}(M)$ denote the set of elements in the peripheral point spectrum. Since $M$ is a positive operator, the spectral radius $r$ will always belong to the peripheral spectrum \cite[Prop.~1]{Schaefer60}, although it is not always in the peripheral point spectrum (Examples \ref{exp: non-growing, no length} and \ref{exp: growing, no length}). For the rest of the paper, we let $T$ denote the operator 
\[
T \coloneqq M/r. 
\]
Likewise, its dual operator is denoted by $T'=M'/r\colon E^{\prime}\to E^{\prime}$. Such normalisation is typical in operator theory, as many spectral and convergence results require $r(T)=1$.

\begin{definition}
We call $\phi \in E'$ an \emph{eigenmeasure} if $T' \phi = \phi$ with $\phi \in K'$ and $\|\phi\| = 1$. 
\end{definition}

Note that eigenmeasures always exist (see Remark \ref{rem: eigenmeasures exist}).

\subsection{Uniqueness and consequences of a natural length function}

Natural length functions will be of most use when strictly positive. These always have as inflation factor the spectral radius:

\begin{prop} \label{prop: strictly positive nlf => lambda = r}
For any substitution \(\sub\), if \(\ell \in K_{>0}\) is a strictly positive eigenvector of \(M\) then its eigenvalue is the spectral radius \(\lambda = r \geq 1\). 
\end{prop}

\begin{proof}
Since \(\ell > 0\), by compactness there are \(c\), \(C > 0\) with \(c \bbo \leq \ell \leq C \bbo\). Iteratively applying the substitution operator, \(cM^n(\bbo) \leq \lambda^n \ell \leq CM^n(\bbo)\) for all \(n \in \N\). Taking norms and applying Corollary \ref{coro: operator norm},
\[
c\|M^n\| \leq \lambda^n \|\ell\| \leq C\|M^n\|, \text{ thus } c^{1/n} \sqrt[n]{\|M^n\|} \leq \lambda \|\ell\|^{1/n} \leq C^{1/n}\sqrt[n]{\|M^n\|}.
\]
Since \(c\), \(C\) and \(\|\ell\| > 0\), the lower and upper bounds have limits equal to \(r\) by Gelfand's formula and \(\lambda = r\).
\end{proof}

We now turn to irreducible substitutions, where we see that natural length functions must be strictly positive and essentially unique:

\begin{theorem} \label{thm: irreducible => unique nlf}
Suppose that \(\sub\) is irreducible and has a natural length function \(\ell \in K\). Then \(\ell \in K_{>0}\), has inflation factor \(\lambda = r\) the spectral radius and every other natural length function is a scalar multiple of \(\ell\). Moreover, \(\ell\) is the only eigenvector of \(M\) with eigenvalue \(r\), up to scalar multiplication.
\end{theorem}

\begin{proof}
First we prove that every positive eigenvector of \(M\), for an irreducible substitution, is strictly positive. Indeed, let \(C \coloneqq \{a \in \A \mid \ell(a) = 0\}\). Then \(C\) is closed, by continuity of \(\ell\). Moreover, for \(c \in C\) with \(\sub(c) = c_1 c_2 \cdots c_n\),
\[
0 = \lambda \ell(c) = \ell(c_1) + \ell(c_2) + \cdots + \ell(c_n),
\]
so that each \(\ell(c_i) = 0\), as \(\ell \in K\). It follows that \(\sub\) maps letters of \(C\) into words over \(C\). Since \(\ell \neq 0\), we have that \(C \neq \A\) and thus \(C = \varnothing\), by irreducibility of \(\sub\), so that \(\ell(a) > 0\) for all \(a \in \A\). 

Then every natural length function for \(\sub\) is strictly positive and thus, by Proposition \ref{prop: strictly positive nlf => lambda = r}, has inflation factor \(\lambda = r\). Suppose that \(\ell\) and \(\ell'\) are two such eigenvectors of \(M\). By strict positivity, compactness and scaling \(\ell\) if necessary, we may assume without loss of generality that \(\ell' \leq \ell\). We define
\[
c \coloneqq \sup \left\{\xi \geq 0 \mid \ell - \xi \cdot \ell' \geq 0\right\}.
\]
Again by compactness and continuity, we have that \(\ell'' \coloneqq \ell - c \cdot \ell' \in K\) and \(\ell''(a) = 0\) for some \(a \in \A\). We also have \(M\ell'' = r \ell''\). Since, by the above, all eigenvectors are strictly positive, we must have that \(\ell'' = 0\), that is \(\ell = c \ell'\), as required.

Finally, suppose that \(\ell'\) is another (not necessarily positive) eigenvector of \(M\), with eigenvalue \(r\). For sufficiently large \(c > 0\) we have that \(\ell' + c\ell\) is strictly positive. Since this is still an eigenvector with eigenvalue \(r\), by the above \(\ell' + c\ell = C\ell\) for some constant \(C\). But then \(\ell' = (C-c)\ell\), as required.
\end{proof}

As well as essential uniqueness and strict positivity of natural length functions for irreducible substitutions, we are also guaranteed that their inflation factors have \(\lambda = r > 1\), unless the substitution is non-growing for trivial reasons:

\begin{prop} \label{prop: irreducible => growing}
Let \(\sub\) be an irreducible substitution and suppose that \(|\sub(b)| \geq 2\) for some \(b \in \A\). Then, for some \(n \in \N\), we have that \(|\sub^n(a)| \geq 2\) for all \(a \in \A\) and, in particular, \(r > 1\).
\end{prop}

\begin{proof}
We first show that, for all \(a \in \A\), we have \(|\sub^n(a)| \geq 2\) for some \(n \in \N\) (but perhaps with \(n\) depending on \(a\)). If not, then we have a sequence \((a_n)_n\) in \(\A\), with \(a_1 = a\) and \(a_{n+1} = \sub(a_n)\). Consider the closure \(C\) of the set \(\{a_n\}_{n \in \N}\). Since \(|\sub(a_n)|=1\) for any \(n\), the same holds for all \(c \in C\), by continuity. By construction \(C \neq \varnothing\) and because \(|\sub(b)| \geq 2\) we have that \(C \neq \A\). Finally, again by continuity, \(\sub\) maps letters of \(C\) to other letters in \(C\), contradicting irreducibility.

Let \(X_n = \{a \in \A \mid |\sub^n(a)| = 1\}\). By continuity of \(\sub^n\), each \(X_n\) is closed, and clearly \(X_{n+1} \subseteq X_n\). Suppose that \(X_n \neq \varnothing\) for all \(n \in \N\). Then by Cantor's intersection theorem, there exists some \(a \in \A\) that is a member of each \(X_n\). But this contradicts \(|\sub^n(a)| \geq 2\) for all \(a \in \A\), for sufficiently large \(n\) depending on \(a\). Thus, \(X_n = \varnothing\) for some \(n\), so that \(|\sub^n(a)| \geq 2\) for all \(a \in \A\) and \(n\) not depending on \(a\). By Lemma \ref{lem: spectral bounds} we have that \(r > 1\).
\end{proof}

\begin{remark} \label{rem: one letter growing}
Henceforth, we will always assume that our substitutions satisfy \(|\sub(a)| \geq 2\) for some \(a \in \A\), and hence \(r > 1\) for any irreducible substitution.
\end{remark}

Next we see that existence of a natural length function ensures certain bounds for letter growth of superwords, and occurrences of letters in superwords belonging to open subsets:

\begin{lem} \label{lem: supertile lengths}
Suppose that \(\sub\) is irreducible and admits some natural length function \(\ell\) (necessarily with inflation constant \(\lambda = r > 1\)). There are constants $\alpha$, $\beta>0$ so that, for all $a \in \A$,
\[
\alpha \lambda^n \leq |\sub^n(a)| \leq \beta \lambda^n.
\]
If \(\sub\) is primitive, then for any open subset $U \subset \A$, there exist $\alpha(U)$, $\beta(U) > 0$ and $p \in \N$ so that, for all $a \in \A$,
\begin{equation} \label{eq: bounds for appearances of letters}
\alpha(U) \lambda^n \leq \#\left\{b \triangleleft \sub^n(a) \mid b \in U\right\} \leq \beta(U) \lambda^n \ \text{ for all } n > p.
\end{equation}
\end{lem}
Note that $\#\left\{b \triangleleft \sub^n(a) \mid b \in U\right\}$ counts the occurrences of $b$ in $\sub^n(a)$ from the set $U$ with multiplicities.
\begin{proof}
By Theorem \ref{thm: irreducible => unique nlf}, we have that $\ell > 0$ has inflation constant $\lambda = r$. We have that $r > 1$ by Proposition \ref{prop: irreducible => growing} (and Remark \ref{rem: one letter growing}).

By compactness we have constants $c$, $C > 0$ so that $c\bbo \leq \ell \leq C\bbo$. Thus,
\[
c |\sub^n(a)| = M^n(c \cdot \bbo)(a) \leq M^n(\ell)(a) = \lambda^n \ell(a) \leq \lambda^n C, \ \mathrm{ hence } \ |\sub^n(a)| \leq \left(\frac{C}{c}\right) \lambda^n.
\]
Analogously we have an upper bound of $\left(\frac{c}{C}\right) \lambda^n \leq |\sub^n(a)|$, as required.

Now suppose that \(\sub\) is primitive and $U \subset \A$ is open. There exists a function $f \in E$ with $f \geq 0$, $\|f\| = 1$ and $f(b) = 0$ for all $b \notin U$ (Remark \ref{rem: tychonoff}). By Proposition \ref{prop: primitive <=> st. positive iterates}, $h \coloneqq M^p(f) \in K_{>0}$ for some $p \in \N$. So, by compactness, $h \geq \kappa \bbo$ for some constant $\kappa > 0$. Hence, for all $a \in \A$,
\[
\#\{b \triangleleft \sub^{n+p}(a) \mid b \in U\} \geq (M^{n+p}(f))(a) = (M^n(h))(a) \geq \kappa M^n(\bbo)(a) \geq \kappa \alpha \lambda^n,
\]
where the first inequality again follows from Lemma \ref{lem: powers of operator}. The lower bound of Equation~\eqref{eq: bounds for appearances of letters} follows, whilst the upper bound follows trivially from that for $|\sub^n(a)|$.
\end{proof}

Since $\|T^n\|=\frac{1}{r^n}\|M^n\|$, we get the following immediate consequence.

\begin{coro}
Let $\sub$ be irreducible and suppose it admits a natural length function. Then the operator $T$ is power bounded, i.e., for some $\beta > 0$, $\|T^n\|\leqslant\beta$  for all $n\in\mathbb{N}$.
\end{coro}

Note that the examples in Ex.~\ref{exp: non-growing, no length} and \ref{exp: growing, no length} that both do not admit a length function are both reducible
(with $C=\{(0,0)\}$ as the closed subalphabet). 

\begin{remark}
Whilst we are not always assured the existence of a natural length function, there is always at least a system of approximate natural length functions. For any substitution $\sub$, there exists a sequence $(x_n)_{n \in \N}$ of $x_n \in K$, with each \(\|x_n\| = 1\), for which $(T-rI)x_n \to 0$ uniformly as $n \to \infty$ (i.e., $\|(T-rI)x_n\| \to 0$ as $n \to \infty$). This result follows from \cite[Prop.~4.1.1]{Nieberg-BL} (see, for instance, the proof of \cite[Prop.~1]{Deng-Du}, which also applies to Banach lattices).
\end{remark}

\section{Existence of natural length functions and unique ergodicity}\label{SEC:invariant}
In this section we will find conditions that assure the existence of a natural length function. One can associate to this a \emph{geometric inflation} rule $\sub$ and build the corresponding tiling dynamical system. 

To motivate this, assume that $\sub$ admits a natural length function $\ell > 0$. We then build the corresponding \emph{geometric hull} $\Omega$ (also known as the tiling space) and equip it with the continuous translation action by $\R$. The resulting geometric self-similarity is a useful structure. For instance, one may then apply the general machinery of fusion tilings of Frank and Sadun \cite{PFS:fusion-ILC} to relate translation-invariant measures on $(\Omega,\R)$ to sequences of measures on the alphabet $\A$ that are compatible with the dual operator $T'$. These can then be related (see Theorem \ref{thm:inv measures - tiling}) to shift-invariant measures on $X_\sub$ by viewing $(\Omega,\R)$ as a suspension flow. In this way, we may find conditions for unique ergodicity in terms of the substitution operator. This will then be applied in Section \ref{SEC:applications} to also find sufficient conditions in terms of the substitution.

For completeness, we recall some of the notions above regarding the geometric realisation of tiling spaces from a subshift and length function; compare \cite[Sec.~5]{BG:book}:

\begin{definition}\label{def: geometric hulls}
Let \((X,\sigma)$ be a subshift over a compact alphabet $\A$ and \(\ell \colon \A \to \R_{>0}\) be continuous. This defines \emph{prototiles} \((I_a,a)\), given by intervals \(I_a \coloneqq [0,\ell(a)]\) also carrying the label \(a \in \A\). A \emph{tile} is a translate \((I_a + x,a)\) of a prototile. A \emph{tiling} \(\mc T = \{\mathfrak{t}_i \mid i \in \Z\}\) is a set of tiles \(\mathfrak{t}_i = (I_{a_i}+x_i,a_i))\) whose supports cover \(\R\), where distinct tiles intersect at most on their boundaries and, given that the tiles are indexed in order (that is, each \(\mathfrak{t}_i\) is to the immediate left of \(\mathfrak{t}_{i+1}\)), the associated bi-infinite sequence \(\cdots a_{-1}|a_0 a_1 \cdots \in \A^\Z\) of labels of tiles is an element of \(X\) (by shift-invariance, the choice of indexing giving the central letter is inconsequential). The \emph{geometric hull} (or \emph{tiling space}) is the set of all such tilings, which carries a natural (compact, Hausdorff) topology which, loosely speaking, considers two tilings to be `close' if they have patches covering large intervals about the origin whose tiles may be paired to be close in support and with close labels in \(\A\). More explicitly, there is a natural identification \(\Omega \cong Y \coloneqq (X_\rho \times \R) /\sim\) with the equivalence relation generated by \((x,\ell(x_0))\sim (\sigma(x),0)\); elements of the form \([(x,0)]\) are those tilings \(\mc T\) with (proto)tile \((I(x_0),x_0)\) containing the origin on its left endpoint and the other tiles following the sequence \(x\) in the obvious way. Any other \([(x,t)]\) is identified with the shift \(\mc T-t\) of this tiling. We have a continuous \(\R\)-action on \(\Omega\), by translating tilings, and the above identification defines a topological conjugacy with \(Y\), with action \(t \cdot [(x,s)] = [(x,s-t)]\). We call \((Y,\R)\) the \emph{suspension flow} over \((X_{\sub},\sigma)\) with continuous \emph{roof function} \(f(x) \coloneqq \ell(x_0)\). 
\end{definition}

\begin{remark}
Note the importance of \(\ell\) being continuous in this definition. If \(\ell\) were not continuous, the identification \(Y \to \Omega\) would not be continuous, for instance. One might imagine elements \(x\), \(y \in X\) whose tiles are all pairwise close in \(\A\) but so that the tilings associated to \([(x,0)],\) \([(y,0)] \in Y\) have tiles with close labels \(x_0\) and \(y_0\) yet very different supports.
\end{remark}

\begin{remark}
In this more general setting, the \emph{transversal} $\Xi$ of $\Omega$ consists precisely of the points identified with elements of the form $[(x,0)]$ in $Y$. This is unambiguous as the length function $\ell$ is \emph{uniformly bounded} away from 0.
Unlike in the FLC setting, $\Xi$ need not be totally disconnected. 
\end{remark}

\subsection{Operator convergence, length functions, and eigenmeasures}
In this section, we consider convergence notions for $T=M/r$, and the consequences for the length function and $T^{'}$-eigenmeasures. As already noted in Remark \ref{rem: regular measures}, all measures will be assumed to be regular and we identify them with elements of $K'$ by the Riesz--Markov--Kakutani representation theorem.

\begin{remark}
In the finite alphabet case, the measures $\mu$ here are in correspondence with row vectors, with one coordinate for each element of $\A$, which `integrate' a weighted sum of letters by the inner product. Then $\mu(M f) = \lambda \mu(f)$ for all \(f \in E\) (that is, \(\mu T = \mu\)) just means that $\mu$ is a left eigenvector of the substitution operator $M$ (i.e., a right eigenvector of the substitution matrix). In the uniquely ergodic case, the entries of this row vector may be considered as the relative frequencies of each tile.
\end{remark}

\begin{definition} \label{def: mean ergodic}
Given a bounded linear operator $T \colon E \to E$ with $r(T)=1$ we define its $n$th Ces\`{a}ro mean as
\[
A_n(f) = \frac{1}{n} \sum_{i=0}^{n-1} T^i(f) .
\]
We call \(T\) \emph{Ces\`{a}ro bounded} if \(\|A_n\|\) is bounded. We call $T$ \emph{mean ergodic} if
\[
P(f) \coloneqq \lim_{n \to \infty} (A_n T)(f) \in E
\]
exists for every $f \in E$. We say that $T$ is \emph{uniformly ergodic} if the above sequence of operators converges in the uniform operator topology \cite{Eisner}. That is, for every $\epsilon > 0$ and for sufficiently large \(n \in \N\), we have $\|A_n(f)-P(f)\| < \epsilon$ for every $f \in E$ with $\|f\| \leq 1$; equivalently\footnote{By the Uniform Boundedness Principle, $P$ is also a bounded operator on $E$.}, $A_n \to P$ in the Banach space of bounded operators on $E$.

Similarly, we say that $T$ is \emph{strongly power convergent} if $P(f) \coloneqq \lim_{n\to\infty} T^n(f)$ exists for all $f\in E$. We call $T$ \emph{uniformly power convergent} \cite{K:power-conv} if $T^n$ converges (necessarily to $P$) in the uniform operator topology.
\end{definition}

\begin{remark} \label{rem: power convergent => mean ergodic}
It is not hard to see that if \(T\) is uniformly (resp.\ strongly) power convergent then it is power bounded and uniformly (resp.\ mean) ergodic. Similarly, \(P \colon E \to E\) is clearly a projection to the subspace of \(E\) fixed by \(T\). Note that convergence properties extend those of primitive non-negative matrices on finite-dimensional vector spaces, i.e., \((M/r)^n\to P\) where \(r\) is the Perron--Frobenius eigenvalue; see \cite{Queffelec}. In the finite-dimensional case, any power bounded operator (thus, any operator with spectral radius \(1\)) is uniformly ergodic.
\end{remark}

Obviously constant length substitutions always admit natural length functions, taking \(\ell = \bbo\). The theorem below shows that in many cases of interest (such as for primitive substitutions), natural length functions exist when the peripheral point spectrum is non-empty. Throughout, we will say that \(f \in E\) is \emph{essentially unique with property} \texttt{P}, if every other element of \(E\) with property \texttt{P} is a scalar multiple of \(f\).

\begin{theorem} \label{thm: nlf from non-void pps or pole}
Let \(\sub\) be irreducible, with renormalised substitution operator \(T = M/r\). Suppose one of the following holds:
\begin{enumerate}[label=(\Alph*)]
	\item $T'$ admits an eigenmeasure $\phi$ satisfying $\phi(f) > 0$ for all $f \in K$ with $f \neq 0$ and $\sigma_\mathrm{per}^p(T) \neq \varnothing$;
	\item $1$ is a pole of the resolvent of $T$;
	\item $\sub$ is primitive and $\sigma_\mathrm{per}^p(T) \neq \varnothing$.
\end{enumerate}
Then \(\sub\) admits a natural length function \(\ell \in K\). We have that \(\ell\) is an essentially unique eigenvector of \(M\) with eigenvalue \(r\). We also have that \(\ell \in K_{>0}\) and is an essentially unique eigenvector of \(M\) in \(K\). Moreover, if C holds then \(r > 1\) (resp.\ \(1\)) is the only element in the peripheral spectrum of \(M\) (resp.\ \(T\)).
\end{theorem}

\begin{proof}
Note that $T$ is irreducible if and only if $\sub$ is (Proposition \ref{prop: irreducible M}), as irreducibility is preserved under normalisation. Condition A implies the result directly by \cite[Thm.~V.5.2]{SchaeferBook}. Condition B implies the result by case (ii) of the corollary following \cite[Thm.~V.5.2]{SchaeferBook}.

Next, we show that condition C implies condition A. The positive cone $K$ is normal and has non-empty interior. It then follows from \cite[Cor., p.~1015]{Schaefer60} that the dual $T' \colon E' \to E'$ of $T$ admits an eigenmeasure $\phi \in K'$; see also \cite{Schaefer67}. Suppose that $\sub$ is primitive. By Proposition \ref{prop: primitive <=> st. positive iterates} and compactness, for any non-zero $f \geq 0$ there is some $p \in \N$ and $c > 0$ for which $(T^p f )(a) > c$ for all $a \in \A$. Then
\[
\langle \phi, f \rangle = \langle (T')^p \phi, f \rangle = \langle \phi,T^p f \rangle > 0 .
\]
Indeed, $\phi \neq 0$ and hence is non-trivial on $\bbo$ (and thus on any $f \in K_{>0}$) since each such function has \(f \geq \kappa \bbo\) for some \(\kappa > 0\). Supposing that $\sigma_\mathrm{per}^p(T) \neq \varnothing$, we have that A holds, as required. Moreover, $1$ is then the only element of the peripheral point spectrum by \cite[Prop.~V.5.6]{SchaeferBook}.

That \(\ell\) has eigenvalue \(r\), is essentially unique with eigenvalue \(r\), essentially unique in \(K\) and \(\ell \in K_{>0}\) all follows from Theorem \ref{thm: irreducible => unique nlf}.
\end{proof}

\begin{remark}
Note that primitivity cannot be dropped in the final statement of the above theorem, as demonstrated by the irreducible substitution \(\sub(a) = bb\), \(\sub(b) = aa\) on the finite alphabet \(\A = \{a,b\}\).
\end{remark}

\begin{remark} \label{rem: eigenmeasures exist}
By \cite[Cor., p.~1015]{Schaefer60}, as used in the proof of Theorem \ref{thm: nlf from non-void pps or pole}, every substitution $\sub$ admits at least one eigenmeasure \(\phi\). It is easily seen from the proof of Theorem \ref{thm: nlf from non-void pps or pole} that \(\phi > 0\) (that is, \(\phi(f) > 0\) for all non-zero \(f \in K\)) if \(\sub\) is primitive.
\end{remark}

Conditions A and C of Theorem \ref{thm: nlf from non-void pps or pole} beg the question of when the peripheral point spectrum of \(T\) is non-empty. 
Before we proceed, we need the following auxiliary results. Recall that a sequence \((x_n)_n\) in a Banach space \(E\) is said to \emph{converge weakly} to \(x\) if \(\mu(x_n-x) \to 0\) as \(n \to \infty\) for all \(\mu \in E'\). 
For linear subspaces $F\subseteq E$ and $G\subseteq E^{'}$, we say that $F$ \emph{separates points in} $G$, if, for every non-zero $\mu\in G$, there exists $f\in F$ such that $\mu(f)\neq 0$.

\begin{lem}[{\cite[Lem.~8.16]{Eisner}}]\label{LEM:ME-CB}
Let $T$ be a bounded operator on a Banach space $E$. If $T$ is mean ergodic, then $T$ is necessarily Ces\`{a}ro bounded and $\frac{1}{n}T^nf\to 0$ (i.e., strongly) for every $f\in E$. 
\end{lem}

\begin{lem}[{\cite[Thm.~8.20]{Eisner}}]
\label{LEM:CB-FIX}
Let T be Ces\`{a}ro bounded and suppose $\frac{1}{n}T^nf\to 0$ weakly for every $f\in E$. $T$ is mean ergodic if, and only if, $\textnormal{fix}(T)$ separates $\textnormal{fix}(T^{'})$. 
\end{lem}

The following gives a seemingly rather weak general condition that ensures the existence of a natural length function.

\begin{prop} \label{prop: mean ergodic => nlf}
A substitution \(\sub\) with \(T\) mean ergodic admits a natural length function with stretching factor \(\lambda = r\).
\end{prop}

\begin{proof}
The operator \(T\) is necessarily bounded (since it comes from a substitution) hence mean ergodicity and Lemma~\ref{LEM:ME-CB} imply that \(T^n f/n \to 0\) strongly, and thus weakly, for every \(f \in E\) and  that \(T\) is Ces\`{a}ro bounded. 
From Lemma~\ref{LEM:CB-FIX}, it follows that \(\mathrm{fix}(T)\) separates \(\mathrm{fix}(T')\). Since eigenmeasures always exist (Remark \ref{rem: eigenmeasures exist}), we must have that \(\mathrm{fix}(T)\) is non-trivial and thus \(P(f) \neq 0\) for some \(f \in E\).

We now show that there must be a \emph{positive} fixed vector of \(T\). As a pointwise limit of the positive operators \(A_n\), the projection \(P\) to the fixed points of \(T\) must also be a positive operator. Write \(f = f_+ - f_-\), where \(f_+\), \(f_- \in K\). Since \(P(f_+ - f_-) = P(f_+) - P(f_-) = f \neq 0\), at least one of the positive functions \(P(f_+)\) or \(P(f_-)\) is non-zero. Since $f \in \text{fix}(T)$, one has $Tf_+ - Tf_- = Tf = f = f_+ - f_-$. Since the decomposition into the positive and negative parts is unique, one necessarily has $P(f_+) = Tf_+ = f_+$ and $P(f_-) = Tf_- = f_-$. By the discussion above, at least one of $P(f_{+})$ or $P(f_{-})$ is a non-zero positive element fixed by \(T\), so that \(\sub\) admits a natural length function with stretching factor \(r\).
\end{proof}

A reducible substitution need not admit a strictly positive natural length function, a simple example being the substitution \(\sub(a) = aaa\), \(\sub(b) = bb\) over \(\A = \{a,b\}\). However, as we saw in Theorem \ref{thm: irreducible => unique nlf}, for an irreducible substitution any natural length function must have stretching factor \(\lambda = r\) and must be strictly positive and unique, up to rescaling. The result below shows that irreducibility is a necessary and sufficient condition for this strict positivity and uniqueness, and analogously for eigenmeasures:

\begin{prop}\label{prop: mean ergod + irred => unique nlf and eigenmeasure}
Let \(\sub\) be a substitution with \(T\) mean ergodic. Then the following are equivalent:
\begin{enumerate}
	\item \(\sub\) is irreducible
	\item \(\mathrm{fix}(T)\) is one dimensional and spanned by a strictly positive \(\ell \in K_{>0}\).
	\item \(\mathrm{fix}(T')\) is one dimensional and spanned by a strictly positive \(\mu \in K'_{>0}\), where \(K'_{>0}\) is the set of \(\mu \in E'\) for which \(\mu(f) > 0\) for all non-zero \(f \in K\).
\end{enumerate}
\end{prop}

\begin{proof}
Since \(T\) is mean ergodic, \(\mathrm{fix}(T) \neq \{0\}\) by Proposition \ref{prop: mean ergodic => nlf}. Then the semi-group of positive operators \(S = \{T^n\}\) satisfies the conditions of \cite[Prop.~III.8.5]{SchaeferBook}, giving precisely the above equivalent conditions.
\end{proof}

\subsection{Fusion tilings, invariant measures, and power convergence}\label{sec:fusion}

Frank and Sadun's notion of \emph{fusion tilings} \cite{PFS:fusion-ILC} allows for a generalised setting in which we can see that invariant measures on the dynamical system $(\Omega,\R)$ are in one-to-one correspondence with sequences of \emph{volume-normalised} and \emph{transition-consistent} measures on the sets $\mathcal{P}_n$ of  $n$-supertiles; see the discussion below for definitions. In this section, we recall this correspondence and we reformulate volume-normalisation and transition-consistency in our setting. 
With this, one can show that invariant measures are in one-to-one correspondence with sequences
\((\mu^{ }_n)_{n=0}^\infty\) of elements in \(K'\) with \(T'\mu^{ }_n = \mu^{ }_{n-1}\), see Proposition~\ref{prop:fusion inv meas} below. 

This allows one to apply the results in the previous section to determine conditions for unique ergodicity.
We show that a sufficient condition for this is strong power convergence of $T$ and irreducibility (which also guarantees the existence, uniqueness and strict positivity of \(\ell\)). This extends the classical result for primitive substitutions over finite alphabets. We recall some definitions below regarding the general fusion framework in \cite{PFS:fusion,PFS:fusion-ILC}.

We assume that our substitution has natural length function \(\ell > 0\) with inflation factor \(\lambda = r > 1\). This defines geometric (labelled) prototiles, see Definition \ref{def: geometric hulls}. The set of these prototiles is denoted by \(\mathcal{P}_0\). Substitution naturally acts on prototiles; the \(n\)th iterate of substitution applied to the prototile associated to \(a \in \A\) has support \([0,r^n \ell(a)]\), where \(r^n \ell(a)\) is also the sum of lengths of the constituent prototiles (since \(\ell\) is a natural length function). Geometric substitution (also denoted  by \(\sub\)) also acts on tiles, finite patches and tilings in the obvious way.

An \emph{\(n\)-supertile} is a patch of the form \(\sub^n(p)\) for a prototile \(p \in \mathcal{P}_0\), which also carries the label \(a \in \A\) of \(p\) (this labelling is usually unnecessary, namely when the substitution \(\sub \colon \A \to \A^+\) is injective, but is needed otherwise). Since supertiles grow without bound, the geometric substitution rule is automatically van Hove, i.e., the ratio of the measure of the boundaries of \(n\)-supertiles and the volume of the supertile goes to zero as \(n \to \infty\). We let \(\mathcal{P}_n\) denote the space of \(n\)-supertiles, which we can tacitly identify with \(\A\). In particular, any measurable subset \(I \subseteq \A\) can be seen as a subset of \(\mathcal{P}_n\) for any \(n\). Thus, we may also canonically identify any \(f \in C(\mathcal{P}_n)\) with a function in \(E \coloneqq C(\mathcal{A})\), which we will often do without comment.

Let \(\Omega = \Omega^{(0)}\) denote the geometric hull associated to \(X_\sub\), with length function \(\ell\). We similarly consider the spaces \(\Omega^{(n)}\) of the same elements of \(\Omega\), but whose tiles are translates of elements in \(\mathcal{P}_n\) instead of \(\mathcal{P}_0\). One can then define the subdivision map \(\omega^{(n)} \colon \Omega^{(n)} \to \Omega^{(n-1)}\), which simply breaks every \(n\)-supertile into the constituent \((n-1)\)-supertiles; up to rescaling, this corresponds simply to the substitution map \(\sub \colon \Omega \to \Omega\). It quickly follows from Proposition \ref{prop: supertiling existence} that, for each tiling \(\mc{T} \in \Omega^{(n)}\), there is some \(\mc{T}' \in \Omega^{(n+1)}\) with \(\omega^{(n+1)}(\mc{T}') = \mc{T}\) (equivalently, \(\sub \colon \Omega \to \Omega\) is surjective) and we call \(\mc{T}'\) a \emph{supertiling} of \(\mc{T}\). We say that \(\sub\) is \emph{recognisable} if the choice of supertiling is unique, that is, \(\sub \colon \Omega \to \Omega\) is injective (equivalently, the subdivision map \(\omega^{(n)}\) is injective for all \(n \in \N\), see~\cite[Sec.~2]{PFS:fusion-ILC}).

\begin{remark}
There exists substitutions over compact alphabets that are aperiodic but are \emph{not} recognisable; see \cite[Ex.~28]{DOP:self-induced} for an example. The substitution above is an example of a profinite substitution \cite{RY-profinite}, and is in fact conjugate to the dyadic odometer (and hence uniquely ergodic).
\end{remark}

The following is the restriction of the definitions in \cite[Sec.~4]{PFS:fusion-ILC} to our setting.
 
\begin{definition}
A regular Borel measure $\mu^{ }_n$ on $\mathcal{P}_n$ (which we may identify with a (positive) element of \(K'\subset C(\mc{P}_n)' \)) is called \emph{volume-normalised} if one has 
\begin{equation}\label{eq: volume normalised}
\int_{\mathcal{A}} \ell^{(n)}(a)\,\dd\mu^{ }_n(a) \coloneqq \mu^{ }_n(\ell^{(n)}) = 1 ,
\end{equation}
where $\ell^{(n)}(a) \coloneqq M^{n}\ell(a) = r^n \ell(a)$ is simply the geometric length of the $n$-supertile corresponding to $a\in\mathcal{A}$. Let $\mathcal{M}^{(n)}$ be the set of volume-normalised measures on $\mathcal{P}_n$. A pair $(\mu^{ }_n,\mu^{ }_N)\in\mathcal{M}^{(n)}\times\mathcal{M}^{(N)}$ is called \emph{transition consistent} if one has 
\begin{equation}\label{eq: trans-cons}
 \mu^{ }_n = (M^{\prime})^{N-n}(\mu^{ }_N),\ \text{ equivalently }\ \mu^{ }_n = \mu^{ }_N\circ M^{N-n} ,
\end{equation} 
where $M$ and $M'$ are the substitution and the dual substitution operators, respectively.
\end{definition}

The following is a correspondence result for tiling spaces generated by fusion rules. For us, the version for $d=1$ suffices.

\begin{theorem}[{\cite[Thm.~4.2]{PFS:fusion-ILC}}]\label{thm: ILC-fusion-invmeas}
Let $\mathcal{R}$ be a van Hove and recognisable fusion rule in $\mathbb{R}^d$, where $\A$ is a compact metric space. Let $\Omega_{\mathcal{R}}$ be the corresponding tiling space with $\mathbb{R}^d$-action. Then there is a one-to-one correspondence between  the set of $\mathbb{R}^d$-invariant measures on $\Omega_{\mathcal{R}}$  and the space of sequences $\{\rho_n\}$ of transition-consistent and volume-normalised measures on $\mathcal{P}_n$. 
\end{theorem}

We now continue to reinterpret Eqs.~\eqref{eq: volume normalised} and \eqref{eq: trans-cons}. 
Using instead notation similar to \cite{PFS:fusion-ILC}, for any continuous function $f \in  E \coloneqq C(\A)$,
\[
\mu^{ }_n(f)=\int_{\mathcal{A}}f(a)\,\dd\mu^{ }_n(a)=
\int_{\mathcal{A}} M^{N-n}f(a)\,\dd\mu^{ }_N(a)=(\mu^{ }_N\circ M^{N-n})(f) .
\]
Note that \(\mu^{ }_N (M^{N-n} f)\), for \(f \colon \mc{P}_n \to \R\), is simply given by considering \(f\) as a function defined over \(N\)-supertiles, by first taking the function \(\mc{P}_N \to \R\) which, on an \(N\)-supertile \(a\), sums \(f\) over \(a\)'s constituent \(n\)-supertiles (which follows from the definition of \(M^{N-n}\)), and then integrating this with \(\mu^{ }_N\). Again, we may identify each \(\mc{P}_n\) with \(\A\), which rescales \(n\)-supertiles by \(r^{-n}\). Due to this rescaling, let us alternatively identify a continuous function \(f \colon \mc{P}_n \to \R\) with the function \(\widetilde{f} \colon \A \to \R\), where \(\widetilde{f}(a) \coloneqq r^{-n} f(a)\) (where, for \(f(a)\), we are in fact identifying \(a\) with its \(n\)-supertile), and similarly identify a measure \(\mu^{ }_n\) on \(\mathcal{P}_n\) with \(\widetilde{\mu}^{ }_n = r^{-n}\mu^{ }_n = f \mapsto \widetilde{\mu}^{ }_n(f)\). Then the volume normalisation and transition consistency conditions in Eqs.~\eqref{eq: volume normalised} and \eqref{eq: trans-cons}  become the conditions
\begin{align}
&\widetilde{\mu}^{ }_n(\ell) = 1, \quad \text{ for all } n \label{eq:vol-norm-tilde}\\ 
&\widetilde{\mu}^{ }_n = (T')^{N-n}(\widetilde{\mu}^{ }_N) \quad \text{ for all }  n < N. \label{eq:trans-con-tilde}
\end{align}
Note that, since \(T(\ell) = \ell\), volume normalisation follows from \(\widetilde{\mu}^{ }_0(\ell) = 1\) and transition consistency. Let \(\widetilde{\mc{M}}^{(n)}\) denote the measures on \(\mathcal{P}_n\) that are volume normalised with the above identification, i.e., those measures \(\mu \in K'\) with \(\mu(\ell) = 1\). For a fixed $n<N$ the following is defined in \cite{PFS:fusion-ILC}:
\[
\Delta_{n,N} = (M')^{N-n}\mc{M}^{(N)} .
\]
Then, by Eqs.~\eqref{eq:vol-norm-tilde} and \eqref{eq:trans-con-tilde}, we may identify this with
\begin{equation} \label{eq: del}
\widetilde{\Delta}_{n,N} = (T')^{N-n} \widetilde{\mc{M}}^{(N)} = \left\{\widetilde{\mu} \in (T')^{N-n}(K') \mid \widetilde{\mu}(\ell) = 1\right\} .
\end{equation}
Let us now drop the tildes and use the above as our definition of \({\Delta}_{n,N}\):
\[
\Delta_{n,N} \coloneqq \left\{\mu \in (T')^{N-n}(K') \mid \mu(\ell) = 1\right\} .
\]
We then define
\begin{equation}\label{eq:inv-lim-sub}
\Delta_n \coloneqq \bigcap_{N \geq n} \Delta_{n,N} \ \text{ and } \ \Delta_\infty \coloneqq \varprojlim \Delta_n = \left\{\{\mu_n\} \mid \mu_n \in \Delta_n, T'(\mu_n) = \mu_{n-1}\right\},
\end{equation}
using the notation \(\{\mu_n\}\) for a sequence of measures, as in \cite{PFS:fusion-ILC}. 

\begin{remark}
In Theorem~\ref{thm: ILC-fusion-invmeas}, it is assumed in its original formulation in \cite{PFS:fusion-ILC} that each \(\mc{P}_n\) is a compact metric space, equivalently here that \(\A\) is metrisable. However, the assumption of metrisability is easily shown to be unnecessary. Since laying out the details here would be time-consuming and largely just represent very similar arguments, we leave it as a simple exercise to the reader to check that the results and arguments in \cite{PFS:fusion-ILC} naturally extend to the case of a general compact, Hausdorff alphabet \(\A\) with no metric.
\end{remark}

By the previous discussion, the next result is a restatement of Theorem~\ref{thm: ILC-fusion-invmeas} for substitutions on compact Hausdorff alphabets. 

\begin{prop}\label{prop:fusion inv meas}
For \(\sub\) recognisable there is a bijection between the set of translation-invariant Borel probability measures on $\Omega$ and elements of $\Delta_\infty$ in Eq.~\eqref{eq:inv-lim-sub}. 
\end{prop}

\begin{theorem}\label{thm:unique ergod FS}
For \(\sub\) recognisable, the tiling dynamical system $(\Omega,\R)$ is uniquely ergodic if and only if $\Delta_\infty$ is a singleton. If \(\sub\) is not recognisable, $(\Omega,\R)$ is uniquely ergodic if \(\Delta_\infty\) is a singleton.
\end{theorem}

\begin{proof}
The first statement follows trivially from Proposition~\ref{prop:fusion inv meas} in the recognisable case. If \(\sub\) is not recognisable, we may consider an abstract hull where recognisability is forced, as follows. In defining \(\Omega\), one may modify the definition of tilings in the geometric hull to be instead hierarchies of tilings \(\T = (\T_n)_{n=0}^\infty\), where each \(\T_n \in \Omega^{(n)}\) and \(\sub(\T_n) = \T_{n-1}\). In other words, we define elements of the hull as instead sequences of tilings, supertilings, \(2\)-supertilings, and so on, so that the tiles of \(\T_n\) may be grouped into those of \(\T_{n+1}\) consistently with respect to the substitution. So we consider the `extended geometric hull' \(\Omega'\) of such elements, which we may identify with
\[
\Omega' = \varprojlim (\Omega^{(0)} \xleftarrow{\sub} \Omega^{(1)} \xleftarrow{\sub} \cdots),
\]
which also provides the topology on \(\Omega'\); translation acts in the obvious way. Substitution \((\T_n)_{n=0}^\infty \mapsto (\sub(\T_n))_{n=0}^\infty = (\T_{n+1})_{n=0}^\infty\) then acts as a homeomorphism on \(\Omega'\), building recognisability into the system. The proof of Proposition~ \ref{prop:fusion inv meas} from \cite{PFS:fusion-ILC} essentially only requires recognisability to identify elements of the hull with consistent supertiling sequences, so Proposition~\ref{prop:fusion inv meas} applies to \(\Omega'\) and we may identify its invariant measures with \(\Delta_\infty\). But \(\Omega'\) naturally factors onto \(\Omega\), by the map \((\T_n)^{ }_n \mapsto \T_0\), and factors of uniquely ergodic dynamical systems are still uniquely ergodic, so the result follows.
\end{proof}

\begin{remark} \label{rem: Delta ER}
Clearly each \(\Delta_n = \Delta_0\) is the eventual range of \(K'\) under the map \(T'\), restricted to \(\mu \in K'\) with \(\mu(\ell) = 1\). Since \(T'\) must be surjective on this eventual range, we have that \(\Delta_\infty\) is a singleton if and only if \(\Delta_0\) is. Then if \((\Omega,\R)\) is uniquely ergodic, the unique invariant measure corresponds to a unique volume normalised eigenmeasure of \(T'\). However, the converse is not true:
\end{remark}

\begin{example}
Consider the irreducible (but non-primitive) substitution 
\[
\sub \colon
\begin{cases}
a\mapsto bb, \\
b\mapsto aa.
\end{cases}
\]
The corresponding operator is \(T\colon \R^2 \to \R^2 \colon (x,y) \mapsto (y,x)\), which is Markov, meaning that \(T(\bbo) = \bbo\) (as the substitution is constant length). Then the fixed subspace of \(T'\) is one-dimensional and generated by \((\frac{1}{2},\frac{1}{2})\). However, $\Delta_\infty$ is 2-dimensional, consisting of the sequences \((\alpha,\beta) \leftarrow (\beta,\alpha) \leftarrow (\alpha,\beta) \leftarrow \cdots,\) where \(\alpha+\beta = 1\) for \(\alpha\), \(\beta \geq 0\) (\(X_\sub\) consists of two periodic points, of all \(a\)s and all \(b\)s). The operator $T$ is irreducible and mean ergodic but is not strongly power convergent since $T^{2n}=I$ and $T^{2n+1}=T$.
\end{example}

Our main result of this section is the following sufficient condition for unique ergodicity in terms of the normalised substitution operator $T$.

\begin{theorem}\label{thm:inv measures - tiling}
Let \(\sub\) be an irreducible substitution on a compact Hausdorff alphabet $\A$ with \(T\) strongly power convergent. Then \(\sub\) admits a natural length function \(\ell \in K\) with inflation factor \(\lambda = r > 1\). Moreover, this is an essentially unique natural length function and \(\ell(a) > 0\) for all \(a \in \A\). With respect to this length function, $(\Omega,\R)$ is uniquely ergodic.
\end{theorem}

\begin{proof}
By Proposition \ref{prop: mean ergod + irred => unique nlf and eigenmeasure} we have an essentially unique natural length function \(\ell > 0\), with inflation factor \(\lambda = r\), and the projection operator $P$, to which $T^n$ converges strongly (though not necessarily uniformly), must be of rank $1$.

Suppose that a volume normalised and transition consistent sequence \(\{\mu_n\}\) is given and let $f \in E = C(\A)$ be arbitrary. From transition consistency,
\[
\mu^{ }_0(f)=(T')^n(\mu^{ }_n)(f)=\mu^{ }_n(T^nf).
\]
Since $T$ is strongly power convergent, $\lim_{n\to\infty} T^nf \to Pf = c\ell$ for some $c\in \R$. It follows that 
\[
\mu^{ }_0(f) = \mu^{ }_n(T^nf) = \mu^{ }_n(c\ell+v_n) = c\mu^{ }_n(\ell) + \mu^{ }_n(v_n) = c + \mu^{ }_n(v_n),
\]
where $v_n \to 0$ uniformly over \(\A\) as $n \to \infty$. Since \(\ell > 0\) we have that \(\kappa \ell \geq \bbo\) for some \(\kappa > 0\), and since each \(\mu^{ }_n\) is positive we have that \(\|\mu^{ }_n\| = \mu^{ }_n(\bbo) \leq \kappa \mu^{ }_n(\ell) = \kappa\) for all \(n\). Thus, since \(v_n \to 0\) uniformly as \(n \to \infty\) and the norms of the \(\mu^{ }_n\) are bounded, it follows that \(\mu^{ }_n(v_n) \to 0\) as \(n \to \infty\) and \(\mu^{ }_0(f) = c\). Since the volume normalised and transition consistent sequence was arbitrary, \(\Delta_0\) is a singleton, so \(\Delta_\infty\) is too (Remark \ref{rem: Delta ER}), thus \((\Omega,\R)\) is uniquely ergodic by Theorem \ref{thm:unique ergod FS}.
\end{proof}

There is an explicit form for the measures of cylinder sets in $\Omega$ in terms of a sequence $\left\{\mu^{ }_n\right\} \in \Delta_\infty$, which is given in terms of measures on sets of patches that generate the topology on $\Omega$. For the exact formulation, we refer the reader to \cite[Sec.~4]{PFS:fusion-ILC}.

\subsection{Invariant measures and suspension flows}\label{sec:suspension}

To end this section, we go back to the relation between the $\R$-invariant measures on $(\Omega,\R)$ and the $\sigma$-invariant measures on $(X_{\sub},\sigma)$ via the suspension flow $(Y,\mathbb{R})$.

Let $\mathcal{M}(Y)$ be the set of $\R$-invariant probability measures on $Y$ and let $\mathcal{M}(X)$ be the set of $\sigma$-invariant measures on the subshift. Since the roof function $f$ is bounded away from zero, it is well known \cite{AK-flows,BRW-suspension} that there is a one-to-one correspondence $L\colon\mathcal{M}(X)\to  \mathcal{M}(Y)$, given by
\[
\widetilde{\mu} \coloneqq L(\mu)=\frac{(\mu\times \mu^{ }_{\text{Leb}})|_{Y}}{\int_X f(x)\,\dd\mu(x)}.
\]
This leads to the following immediate consequence. 

\begin{coro}\label{cor: inv measures - subshift}
Let $\sub$ be a substitution on a compact Hausdorff alphabet that is recognisable and admits a natural length function $\ell$ that is strictly positive and has inflation factor \(\lambda > 1\). Then there is a one-to-one correspondence between the elements of $\Delta_{\infty}$ and the set $\mathcal{M}(X_{\sub})$ of $\sigma$-invariant probability measures on $X_{\sub}$. If \(\sub\) is not recognisable, then \(X_\sub\) is uniquely ergodic if \(\Delta_\infty\) is a singleton.
\end{coro}

\section{Applications}\label{SEC:applications}

Using the operator-theoretic perspective above, we will establish two classes of substitutions that are uniquely ergodic. We begin with examples that have a type of quantitative coincidence property, that appears to apply very generally to primitive substitutions on alphabets \(\A\) containing an isolated point. Afterwards, we will consider constant length substitutions whose columns generate uniformly equicontinuous semigroups.

\subsection{Quasi-compact substitutions and unique ergodicity}\label{SEC:ergodicity}
In this section we will give a simple condition on $\sub$ that ensures the operator $T$ to be \emph{quasi-compact}, allowing us to deduce strong properties of the substitution.

\begin{definition}
An operator $C\colon E\to E$ is called \emph{compact} if the image of the unit ball under $C$ is relatively compact. 
\end{definition}

\begin{prop} \label{prop: non-compact}
If \(\# \A = \infty\) and \(\sub\) is letter surjective then \(M\) is not compact.
\end{prop}

\begin{proof}
Since \(\A\) is Hausdorff and \(\#\A = \infty\) there exists an infinite sequence \(U_1\), \(U_2\), \ldots of non-empty, disjoint open subsets of \(\A\). Since \(\sub\) is letter surjective, for each \(U_i\) we may pick some \(p_i \in U_i\) and \(a_i \in \A\) with \(p_i \triangleleft \sub(a_i)\). For each \(i \in \N\) we may find a continuous map \(g_i \colon \A \to [0,1]\) with \(g_i(p_i)=1\) and \(g_i(a) = 0\) for \(a \notin U_i\) (Remark \ref{rem: tychonoff}). For \(n \in \N\), define \(f_n \colon \A \to [0,1]\) by \(f_n(a) \coloneqq g_1(a) + \cdots + g_n(a)\). Then each \(f_n\) belongs to the unit ball of \(E\) but we claim that \((M(f_n))_n\) has no convergent subsequence. Indeed, given arbitrary \(n \in \N\), we note that \(\|M(f_n) - M(f_i)\| \geq 1\) for \(i < n\), since
\[
(M(f_n) - M(f_i))(a_n) = M(f_n - f_i)(a_n) = M(g_{i+1})(a_n) + \cdots + M(g_n)(a_n) \geq M(g_n)(a_n) \geq 1.
\]
The last inequality follows from the fact that we have \(p_n \triangleleft \sub(a_n)\), so that \(M(g_n)(a_n) = g_n(x_1) + \cdots + g_n(x_k) \geq 1\), where \(\sub(a_n) = x_1 x_2 \cdots x_k\) and \(g_n(x_i) = 1\) for \(x_i = p_n\).
\end{proof}

Recall from Remark \ref{rem: letter surjective} that any substitution can be restricted to a letter surjective substitution (although perhaps not without reducing the subshift), and from Remark \ref{rem: language realised => letter surjective} that any substitution that realises the whole alphabet in the subshift is letter surjective. So essentially all infinite substitutions of interest here have non-compact substitution operators. However, we are able to identify cases where $T$ satisfies a weaker condition called \emph{quasi-compactness}.
Quasi-compactness is a very powerful property here, since it ensures the existence of a continuous length function and unique ergodicity in the primitive case, as we will see below.

\begin{definition}
An operator $T$ with $r(T)=1$ is called \emph{quasi-compact} if there exists some compact operator $C$ and power $n \in \N$ for which $\|T^n - C\| < 1$.
\end{definition}

The above definition can be reformulated as follows. Let \(\mathcal{B}(E)\) be the Banach algebra of bounded linear operators on \(E\). We have the Banach subalgebra \(\mathcal{K}(E) \leqslant \mathcal{B}(E)\) of compact operators and the \emph{Calkin algebra} \(\mathcal{B}(E) / \mathcal{K}(E)\). Now, there are several notions of `essential spectrum' for the operator \(T\), which are inequivalent (but nonetheless give the same notion of essential spectral radius). One is that \(\lambda \in \C\) is in the essential spectrum if \(T - \lambda\Id\) is non-invertible in the Calkin algebra, equivalently \cite{Arv02}, \(T - \lambda \Id\) is not Fredholm (an operator is Fredholm if its range is closed and both its kernel and cokernel are finite-dimensional). An alternative (which gives a different essential spectrum in general) is to take the Browder spectrum \cite{Bro61}, that is, those \(\lambda \in \C\) for which one of the following holds: \(\text{range}(T-\lambda\Id)$ is not closed in $E$, \(\text{dim}\left(\bigcup_{r\geqslant 0} \text{ker}\left((T-\lambda\Id)^r\right)\right)=\infty\) or $\lambda$ is an accumulation point of $\sigma(T)$. In any case, the \emph{essential spectral radius} \(r_\mathrm{ess}(T)\) is then the supremum of moduli of elements in the chosen essential spectrum. Equivalently, there is a Gelfand formula: one may define the operator norm \(\|[T]\|_\mathrm{Cal}\) in the Calkin algebra as the infimum of operator norms \(\| T - C \|\) over all compact operators \(C\). Then \(r_\mathrm{ess}(T) = \lim_{n \to \infty} \sqrt[n]{\|[T^n]\|_\mathrm{Cal}}\), which is a decreasing sequence. Then clearly, for \(r(T) = 1\), we have that \(T\) is quasi-compact if and only if \(r_\mathrm{ess}(T) < 1\).

From \cite[Lem.~17]{Bro61} (see also \cite[Thm.~1]{Lay67}):

\begin{lem} \label{lem: browder spectrum}
We have that \(\lambda \in \sigma(T)\) lies outside the Browder spectrum if and only if \(\lambda\) is a pole of the resolvent \((\lambda\mathbb{I}-T)^{-1}\) of finite rank.
\end{lem}

\begin{lem} \label{lem: qc => length function}
Let \(T\) be a positive, quasi-compact operator with \(r(T)=1\). Then \(1 \in \sigma(T)\) is a pole of the resolvent of \(T\) of finite rank and thus \(T\) has a non-trivial fixed vector.
\end{lem}

\begin{proof}
Since \(T\) is positive, \(r(T) = 1 \in \sigma(T)\), which cannot be in the Browder spectrum because \(r_\mathrm{ess}(T) < 1\) by quasi-compactness. By Lemma \ref{lem: browder spectrum}, \(1\) is a pole of the resolvent of finite rank, which is an eigenvalue by Lemma \ref{lem: pole => eigenvalue}.
\end{proof}

If \(T\) is additionally irreducible, then in our setting of \(E = C(\A)\) and by Theorem \ref{thm: nlf from non-void pps or pole}, the fixed subspace of \(T\) in Lemma~\ref{lem: qc => length function} is one-dimensional and spanned by some \(\ell > 0\). If each power $\sub^k$ is irreducible, which includes the case that $\sub$ is primitive, then we have the following stronger result:

\begin{prop} \label{prop: primitive+qc => length}
Suppose that $\sub^k$ is irreducible for each $k \in \N$ (for example, $\sub$ is primitive) and that $T$ is quasi-compact. Then $\{r(T)\} =\{1\} = \sigma_\mathrm{per}(T) = \sigma_\mathrm{per}^p(T)$ and $T$ is uniformly power convergent. If $\sub$ is primitive then $r(T)=1$ is a simple pole of the resolvent.
\end{prop}

\begin{proof}
As discussed above, $T$ admits a fixed vector in \(K_{>0}\). Then $1$ is the only element of the peripheral spectrum, and $T$ is uniformly power convergent, by \cite[Prop.~5]{Abdelaziz}, whose results apply since $(E,K)$ has the decomposition property (Lemma \ref{lem: cone properties}). If $\sub$ is primitive then $1$ is a simple pole of the resolvent, by \cite[Thm.~11]{Karlin}: the results of \cite[Sec.~5]{Karlin} are given for `strictly positive operators', but all hold more generally for strongly positive operators (which holds for $\sub$ primitive, by Proposition \ref{prop: primitive <=> st. positive iterates}), as stated in the introduction of that section.
\end{proof}

\begin{theorem} \label{thm: primitive+qc => uniquely ergodic}
Suppose that $\sub$ is primitive or, more generally, that \(\sub^k\) is irreducible for all \(k \in \N\). If $T$ is quasicompact, the following properties hold.
\begin{enumerate}
	\item $\{r(T)\} = \{1\} = \sigma_\mathrm{per}(T) = \sigma_\mathrm{per}^p(T)$;
	\item $T$ is uniformly power convergent;
	\item $M$ admits an essentially unique eigenvector in $K$, which is a strictly positive natural length function with stretching factor $\lambda = r > 1$;
	\item the tiling dynamical system $(\Omega,\R)$ is uniquely ergodic.
	\item the subshift $(X_{\sub},\sigma)$ is uniquely ergodic.
\end{enumerate}
\end{theorem}

\begin{proof}
The statements (1), (2) follow from Proposition~\ref{prop: primitive+qc => length}, whereas (3) follows from Theorem \ref{thm: nlf from non-void pps or pole}. The statement (4) follows from (2),(3) and Theorem \ref{thm:inv measures - tiling} and then (5) follows from (4) and Corollary~\ref{cor: inv measures - subshift}.
\end{proof}

\begin{remark}
The previous theorem is reminiscent of the \emph{Ruelle--Perron--Frobenius theorem} for Markov shifts $\Sigma_{A}$ \cite{Bow08}. In this setting, the operator in question is the transfer operator $L_{\phi}$ for a H\"older continuous potential $\phi\colon \Sigma_{A}\to \R$. In the case of topologically mixing Markov shifts over a finite alphabet, the unique equilibrium measure is given by $h\, \dd \mu$, where $L_{\phi}h=\lambda h$ and $L_{\phi}'\mu=\lambda \mu$, where $\lambda>0$. Here, $\lambda$ is the Perron--Frobenius eigenvector of the transition matrix $A$ of the Markov shift, and $h$ and $\mu$ are the corresponding (essentially unique) eigenfunction (resp.\ eigenfunctional) of $L_{\phi}$ (resp. of $L_{\phi}')$.

The case of countable Markov shifts (CMS) is more involved, and the extension of the above result relies on recurrence properties of $A$ (which is now an infinite $0{-}1$ matrix) and, more generally, those of the potential function $\phi$ considered; see \cite{Kit98,Sar99} for background. The main difficulty in the study of CMSs is the non-compactness of the shift space. 

A family of substitutions over (parametrised) compactifications of $\N_0$ are studied in \cite{FGM}. One can choose the parameters so that the restriction $M|_{V}$ of the substitution operator to the subspace $V=\left\langle\left\{\bbo_{n}\right\}_{n\in\N_0}\right\rangle\subset C(\A)$ can be seen as a transition matrix $A$ of a CMS. It would be interesting to see whether, in this subclass, there are connections between properties of the substitutive subshift with operator $M$ and the CMS with transition matrix $A=M|_{V}$. We suspect that $\sub$ satisfies the condition of Theorem~\ref{thm: primitive+qc => uniquely ergodic} if, and only if, $\Sigma_{A}$ satisfies a Ruelle--Perron--Frobenius theorem and admits a (unique) equilibrium measure. For results along this vein in the setting of substitutions on countable (discrete topology) alphabets, we refer the reader to \cite{DFMV}. In particular, \cite[Thms.~3.23\, 3.33]{DFMV} provide conditions for the existence of a unique shift invariant probability measure.
\end{remark}

Note that primitivity and quasi-compactness are both \emph{abelian} properties, whence
if the result above holds for $\sub$, it also holds for any substitution derived from $\sub$ by rearranging letters in the images $\varrho(a)$. We now introduce the following combinatorial criteria (which are abelian in the sense of the description above) that guarantee $T$ to be quasi-compact. For convenience, we let $r=r(M)$ for the rest of the paper.

\begin{theorem} \label{thm: quasi-compact subs}
Let $\sub$ be an arbitrary substitution. For $P \subseteq \A$ and $k \in \N$ we define
\[
C_k(P) \coloneqq \max_{a \in \A} \#\{b \triangleleft \sub^k(a) \mid b \notin P\}.
\]
For finite \(P \subseteq \A\) and \(k \in \N\),
\begin{enumerate}
	\item if \(P\) consists of only isolated points then \(C_k(P)\) is an upper bound for the essential spectral radius of \(M^k\);
	\item otherwise, \(2C_k(P)\) is an upper bound for the essential spectral radius of \(M^k\).
\end{enumerate}
In particular, if \(C_k(P) < r^k\) and \(P\) consists of only isolated points, or if \(2C_k(P) < r^k\), then $T$ is quasi-compact.
\end{theorem}

\begin{proof}
We first consider the case that $P$ is a finite set of isolated points. Consider the operator $V \colon E \to E$ given by
\[
(Vf)(a) =
\begin{cases}
f(a) & \text{ if } a \in P, \\
0    & \text{ otherwise.}
\end{cases}
\]
Clearly $V$ is a compact operator, since it maps onto the subspace of functions supported on $P$, which is of finite dimension equal to $\#P$. We define $C \coloneqq T^k \circ V$. This is a compact operator, since the composition of any bounded operator with a compact operator is compact. Then
\[
((T^k-C)(f))(a) = (T^k(I - V)(f))(a) = \frac{1}{r^k}\left( \sum_{b \triangleleft \sub^k(a) \text{ with } b \notin P} f(b)\right).
\]
Then over all $\|f\| = 1$, the norm of $(T^k-C)(f)$ clearly maximised by the constant function $f = \bbo$, for which
\[
\|(T^k-C)(\bbo)\| = \frac{1}{r^k}\max_{a \in \A} \left(\sum_{b \triangleleft \sub^k(a) \text{ with } b \notin P} 1\right) = \frac{C_k(P)}{r^k}.
\]
Since \(r_\mathrm{ess}(T^k)\) is bounded above by the spectral radius of \(T^k-C\), which in turn is bounded above by \(\|T^k-C\|\), we have that \(C_k(P) / r^k\) is an upper bound for the essential spectral radius of \(T^k\), equivalently \(C_k(P)\) is an upper bound for the essential spectral radius of \(M\). If \(C_k(P) < r^k\), it follows that $T$ is quasi-compact.

Now suppose that $P$ is finite but contains non-isolated points. Then the operator $V$ above is not continuous and needs to be adjusted. Choose open sets $U_p$, one for each $p \in P$, so that $p \in U_p$ and $U_p \cap U_q = \varnothing$ for $p \neq q$. For each $p \in P$, there exists a continuous function $\psi_p \colon \A \to [0,1]$ for which $\psi_p(p) = 1$ and $\psi_p(a) = 0$ for $a \notin U_p$ (Remark \ref{rem: tychonoff}). We define $V \colon E \to E$ by
\[
(Vf)(a) = \sum_{p \in P} f(p) \cdot \psi_p(a) = 
\begin{cases}
f(p) \cdot \psi_p(a) & \ \text{ for } \ a \in U_p, \\
0                 & \ \text{ if } \ a \notin \bigcup_{p \in P} U_p,
\end{cases}
\]
where in the latter equality we use that the $U_p$ are disjoint. It is easy to see that $V$ is continuous, and it is compact since $V(E)$ is contained in the $\# P$-dimensional subspace spanned by the functions $\psi_p$. Thus, $C \coloneqq T^k \circ V$ is also a compact operator and \(T^k-C\) is given by:
\[
((T^k-C)(f))(a) = (T^k(I - V)(f))(a) = \frac{1}{r^k}\sum_{b \triangleleft \sub^k(a)} (f(b) - (Vf)(b)).
\]
Notice that for $b \in P$ we have \(Vf(b) = f(b)\) so that \(f(b) - (Vf)(b) = 0\) and thus the above sum may be taken over \(b \notin P\). For \(b \notin P\) and \(\|f\| \leq 1\), we have that \(|f(b) - Vf(b)| \leq |f(b)| + |Vf(b)| \leq 2\). Hence,
\[
\|(T^k-C)(f)\| \leq \frac{1}{r^k}\max_{a \in \A} \left( \sum_{b \triangleleft \sub^k(a) \text{ with } b \notin P} 2 \right) = \frac{2 C_k(P)}{r^k}.
\]
Then, analogously to the first case, \(2 C_k(P)\) is an upper bound for the essential spectral radius of \(M\) and \(T\) is quasi-compact if  \(2 C_k(P) < r^k\).
\end{proof}

Note that, in the finite alphabet case, $T$ is always compact (i.e., $r_{\text{ess}}(T)=0$), and hence always quasi-compact, so the above result is of interest only in the infinite alphabet case. The bound $r \geq \min_{a \in \A} \sqrt[n]{|\sub^n(a)|}$ for all $n$ from Lemma \ref{lem: spectral bounds} makes the above checkable for many interesting examples:

\begin{example}\label{ex:quasi-compact}
Consider the substitution
\[
\sub \colon
\begin{cases}
0 \mapsto 0\ 1, \\
n \mapsto 0\ n\!-\!1\ n\!+\!1,
\end{cases}
\]
defined on $\A = \N_{\infty}$. Let $P = \{0\}$. Then every $1$-supertile contains at most $2$ letters not in $P$, so we need to show that $r = r^1 > C_1(P) = 2$. From the first power of the substitution alone, we can only conclude from Lemma \ref{lem: spectral bounds} that $r \geq 2$, whereas we need a strict inequality. We therefore consider $\sub^2$. We have $|\sub^2(0)| = |01002| = 5$, $|\sub^2(1)| = |0101013| = 7$ and $|\sub^2(n)| = |0 1 0 (n-2) (n+1) 0 n (n+2)| = 8$ for $n > 1$, so that $\min_{a \in \A} |\sub^2(a)| = 5$ and hence $r \geq \sqrt{5} > 2$, by Lemma \ref{lem: spectral bounds}. It follows that $T$ is quasi-compact. For illustrative purposes, we also note that one could take $P = \{0,1\}$ and $k=2$. We have that each $\sub^2(a)$ contains at most $4$ elements not in $P$, and $4 < 5 \leq r^2$, thus \(C_2(P) = 4 < r^2\).
\end{example}

\begin{example} \label{ex: CL quasi-compact}
Let $\sub$ be any substitution of constant length $L$ for which $\sub^k(a)$ contains a letter in some finite subset $P \subseteq \A$, for any $a \in \A$. If $P$ consists of isolated points then $\sub$ satisfies the conditions of Theorem \ref{thm: quasi-compact subs}. Indeed, in this case $r = L$, and $|\sub^k(a)| = r^k$ for all $a \in \A$. Since, by assumption, each such $\sub^k(a)$ contains at least one occurrence of a letter in $P$, we have that $C_k(P) \leq r^k - 1 < r^k$. 

In fact, the same criterion may be used even if $P$ has non-isolated points in the constant length case. Indeed, suppose each $\sub^k(a)$ contains a letter of $P$, for some finite $P$ and $k \in \N$. Define
\[
P(N) = \left\{a\in \A\colon a\triangleleft v,\,v \in P \cup \sub^k(P) \cup \sub^{2k}(P) \cup \cdots \cup \sub^{Nk}(P)\right\}.
\]
Here, $\sub^{n}(P)=\left\{\sub^{n}(a)\colon a\in P\right\}$. 
Suppose a given word $w$ contains at least $p$ letters in $P(N)$ and at most $n$ letters in its complement, where $|w| = p+n$. Then $\sub^k(w)$ contains at least $rp + n$ letters in $P(N+1)$, since it has $rp$ from applying $\sub^k$ to letters in $P(N)$, and at least one contribution in $P = P(0) \subseteq P(N+1)$ by applying $\sub^k$ to every other letter. So at most $r(p+n) - (rp + n) = (r-1)n$ letters of $\sub^k(w)$ are not in $P(N+1)$, since $w$ has $p+n$ letters. In matrix form,
\[
v
\mapsto
A v
\ \text{ with } \ 
A =
\begin{pmatrix}
r & 1 \\
0 & r-1
\end{pmatrix}
\]
where $v = (p,n)^T$ and $Av$ has first coordinate a lower bound for the number of elements in $P(N+1)$, and second coordinate an upper bound for the number of elements not in $P(N+1)$. It follows that, for each $a \in \A$ and $N \in \N$, the supertile $\sub^{Nk}(a)$ contains at least $(A^N(0,1)^T)_1$ letters in $P(N)$ and at most $(A^N(0,1)^T)_2$ not in $P(N)$. The matrix $A$ has leading Perron--Frobenius eigenvector $(1,0)$, with eigenvalue $r$, and $1/r^N A^N(0,1)^T \to (1,0)$ as $n \to \infty$; that is, the ratio of letters of $\sub^{Nk}(a)$ in $P(N)$ converges to $1$ as $N \to \infty$. In particular, for some $N > 0$, at least half of the letters of each $\sub^{Nk}(a)$ belong to $P(N)$. Since $P(N)$ is a finite set, it follows from Theorem \ref{thm: quasi-compact subs} that $T$ is quasi-compact.
\end{example}

\begin{example}
Let $\A = S^1$ and consider the substitution $\sub(x) = 1 \ \alpha x$. Then we may take $P = \{1\}$ and every $\sub^1(x)$ contains a letter of $P$, so $T$ is quasi-compact by Example \ref{ex: CL quasi-compact}.
\end{example}

\begin{example}
Suppose that $\sub \colon \A \to \A^+$ is primitive and that $\A$ contains at least one isolated point. Then if $\sub$ is constant length, it satisfies the conditions of Theorem \ref{thm: quasi-compact subs}. Indeed, if $a \in \A$ is isolated, then $P \coloneqq \{a\}$ is open, so by primitivity for some $p \geq 1$, each $\sub^p(b)$ contains at least one occurrence of $a$. Then $\sub$ is again covered by Example \ref{ex: CL quasi-compact}.

If $\sub$ is primitive (but not necessarily constant length), then it is similarly easy to show that, for every $b \in \A$, there are at least $\alpha_a \cdot |\sub^k(b)|$ occurrences of any given isolated point $a \in \A$ in each $\sub^k(b)$, for $k$ sufficiently large (depending on \(a\)) and $\alpha_a > 0$ not depending on $k$. However, it is not clear that the condition of Theorem \ref{thm: quasi-compact subs} is satisfied, since a priori it may happen that $|\sub^k(b)|/r^k$ is unbounded. In fact, even if $|\sub^k(b)| < Cr^k$ for all $k \in \N$ and $b \in \A$, it is not immediately clear that there exists a finite subset $P \subset \A$ so that each $\sub^k(b)$ contains sufficiently many occurrences in $P$ so as to apply Theorem \ref{thm: quasi-compact subs}. This raises the following question.  
\end{example}

\begin{question}
Suppose that $\sub$ is primitive and that $\A$ contains at least one isolated point. Is $T$ power bounded? Is $T$ quasi-compact?
\end{question}

One simplification of the above question is to restrict to primitive substitutions for which the isolated points are dense in $\A$, and $\sub$ sends isolated points to words of isolated points. Such substitutions are natural to consider: they arise, for instance, from some combinatorial substitutions on $\N_0$ that admit a compactification to a primitive substitution.

Below, we give an example of a primitive substitution that satisfies the conditions of Theorem~\ref{thm: primitive+qc => uniquely ergodic} and compute the corresponding length function.

\begin{example}
Let $\A = \N_{\infty}$ and consider the substitution defined in Example~\ref{ex:non-CL}:
\[
\sub \colon \left\{
\begin{array}{rl}
0\       \mapsto & 0\ 0\ 0\ 1,            \\
n\       \mapsto & 0\  n\!-\!1\  n\!+\!1, \\
\infty\  \mapsto & 0\ \infty\ \infty.
\end{array}\right.
\]
One can easily see that $\sub$ is primitive.
Using the same arguments as in Example~\ref{ex:quasi-compact}, one can show that the operator $T$ is quasi-compact. 
Here, we give closed formulae for the letter frequencies and for the natural length function (to the inflation factor $\lambda=r)$. 
Let $\nu=(\nu_0,\nu_1,\ldots)^{T}$ be the frequency vector.
Note that the letter frequencies satisfy the linear recurrence relation $\nu_{j}=\lambda\nu_{j-1}-\nu_{j-2}$, for $j\geqslant 2$ with $\nu_{1}=(\lambda-2)\nu_0-1$. From the equation $\sum_{j=0}^{\infty}\nu_j=1$, one gets $\lambda=3+\nu_0$. This implies $\nu_1=\nu_0^2+\nu_0-1$.
The recurrence relation is homogeneous with characteristic polynomial 
$p(x)=x^2-\lambda x+1$. This means the solution of the recurrence relation is of the form 
\[
\nu_{j}=C_{+}\left(\frac{\lambda+\sqrt{\lambda^2-4}}{2}\right)^{j}+C_{-}\left(\frac{\lambda-\sqrt{\lambda^2-4}}{2}\right)^{j}
\]
for some constants $C_{-},C_{+}$. Since the frequencies satisfy $\nu_{j}\leqslant 1$, we immediately get that $C_{+}=0$. From the initial conditions, we get $C_{-}=\nu_0$ and $C_{-}\left(\frac{\lambda-\sqrt{\lambda^2-4}}{2}\right)=\nu^2_0+\nu_0-1$. Solving these two equations simultaneously yields $\nu_0=-1$ or $\nu_0=\frac{1}{\sqrt{2}}$, where the first is obviously an extraneous solution. This then yields $\lambda=3+\frac{1}{\sqrt{2}}$ and
\[
\nu_j=\frac{1}{\sqrt{2}}\left(1-\frac{1}{\sqrt{2}}\right)^{j}
\]
for $j\geqslant 0$ with $\nu_{\infty}=0$.

Next we solve for the lengths $\ell_{j}$. They satisfy the non-homogeneous linear recurrence relation $\ell_{j}=\lambda\ell_{j-1}-\ell_{j-2}-\ell_0$.  If we fix $\ell_{\infty}=1$, we get that $\ell_{0}=\lambda-2=1+\frac{1}{\sqrt{2}}$. Note that the non-homogeneous term is $-\ell_0=-\ell_0\cdot 1^j$, and since $1$ is not a root of $x^{2}-\lambda x-1$, and the two roots are distinct, we have a particular solution given by some constant $C_{p}$. Solving for $C_p$ yields $C_p=\frac{\ell_0}{\lambda-2}=1$. Combining this to the solution of the associated homogeneous recurrence relation leads to
\[
\ell_j=1+\frac{1}{\sqrt{2}}\left(1-\frac{1}{\sqrt{2}}\right)^j \ \ \text{ for all } j \geq 0.
\]
\end{example}

\begin{remark}\label{rem: trans inflation}
Recent work of Frettl\"{o}h, Garber and Ma\~nibo \cite{FGM} has shown that there exists a certain family of substitutions on appropriate compactifications of $\N_0$, which are generalisations of the example above, so that Theorem~\ref{thm: primitive+qc => uniquely ergodic} applies, making the subshifts $X_\sub$ uniquely ergodic. By primitivity, quasi-compactness and Theorem \ref{thm: primitive+qc => uniquely ergodic}, these substitutions have a uniquely defined, strictly positive natural length function, and the family is large enough to realise any inflation factor $\lambda \in [2,\infty)$. This is in contrast to the finite case, where inflation factors necessarily have to be algebraic integers.
\end{remark}

\begin{example}
This example satisfies the conditions of Theorem \ref{thm: quasi-compact subs} whilst being non-constant length and with $\A$ uncountable. Give a different compactification of $\N_0$ by embedding it in the cylinder by the map
\[
f \colon \N_0 \to [0,1] \times S^1, \ f(n) \coloneqq (1/(n+1),\alpha^n).
\]
Then each $n \in \N_0$ is isolated in $\A = \overline{f(\N_0)} \subset [0,1] \times S^1$, where we add accumulation points $A = \{0\} \times S^1 \subset \A$. Identify the isolated points $(1/(n+1),\alpha^n)$ with $n \in \N_0$ and the accumulation points $(0,x) \in \{0\} \times S^1$ with $x \in S^1$. Then
\[
\sub \colon 
\begin{cases}
0      \mapsto 0\;\, 0\;\, 0\;\, 1,          \\
n      \mapsto 0\;\, n\!-\!1\;\, n\!+\!1 \text{ for } n \in \N_0,   \\
x \mapsto 0\;\; \alpha^{-1}x\;\; \alpha x \text{ for } x \in S^1.
\end{cases}
\]
It is not hard to check again that this is primitive. Theorem \ref{thm: quasi-compact subs} applies just as before, so $T$ is quasi-compact.
\end{example}

\subsection{Constant length substitutions}
Queff\'{e}lec has already considered substitutions of constant length in compact metrisable alphabets in \cite{Queffelec}, which she called \emph{compact automata}. In this subsection, we let \(|\sub(a)| = L\) for all \(a \in \A\). This greatly simplifies matters, owing to the following:
\begin{enumerate}
	\item \(\|M\| = r = L\) and \(\|T\| = 1\), that is, \(T\) is a contraction;
	\item \(M(\bbo) = L \bbo\), so \(\bbo\) is a natural length function of \(\sub\);
	\item \(\sub(a) = \sub_1(a) \sub_2(a) \cdots \sub_L(a)\) for continuous functions \(\sub_i \colon \A \to \A\), which we call \emph{columns}.
\end{enumerate}
The last property means that the substitution operator is defined by
\begin{equation} \label{eq: constant length substitution operator}
M(f) = f \circ \sub_1 + f \circ \sub_2 + \cdots + f \circ \sub_L.
\end{equation}
Let \(\Phi\) be the semigroup generated by the columns \(\{\sub_i\}\). We call \(\Phi\) \emph{equicontinuous} if this set of functions is uniformly equicontinuous (note that this notion does not require \(\A\) to be metrisable, since \(\A\) carries a uniformity, as a compact Hausdorff space). The following results are similar to those obtained in \cite[Thm.~12.1, Cor.~12.2]{Queffelec}.

\begin{theorem} \label{thm: CL UE}
Suppose that \(\sub\) is primitive, constant length and generates an equicontinuous semigroup \(\Phi\). Then \(T\) is strongly power convergent and the tiling dynamical system associated with \(\sub\) is uniquely ergodic.
\end{theorem}

\begin{proof}
Let \(f \in K\) be arbitrary. First, note that
\[
T^n(f) = \frac{1}{L^n} \sum_{i=1}^{L^n} f \circ \phi^{(n)}_i,
\]
where the sum is over all possible \(n\)-fold compositions \(\phi^{(n)}_i \in \Phi\) of columns. Since \(\A\) is compact, \(f\) is uniformly continuous and thus the maps \(\{f \circ \phi_i^{(n)}\}\) are uniformly equicontinuous. Since \(T^n(f)\) is a convex combination of these maps, it follows that \(\{T^n(f)\}\) is a uniformly equicontinuous family. Since \(\|T\| \leq 1\), it is also uniformly bounded, so \(\{T^n(f)\}\) is relatively compact in \(E\) by the Arzel\'{a}--Ascoli Theorem and thus
\[
T^{n_k}(f) \to g \text{ as } k \to \infty
\]
for some subsequence \((n_k)_k\) and \(g \in K\). By primitivity, dismissing the trivial case of \(f = 0\), we must have that \(g > 0\). We will show that \(g\) is constant.

Observe that \(g\) is itself also a cluster point of iterates of \(T\) applied to \(g\). Indeed,
\begin{align*}
\|T^{n_{k+1}-n_k}(g) - g\| = \|T^{n_{k+1}-n_k}(T^{n_k}f - v_k) - g\| = \|T^{n_{k+1}}f - T^{n_{k+1}-n_k}(v_k) - g \| \leq \\
\|T^{n_{k+1}}f - g \| + \|T^{n_{k+1}-n_k}(v_k)\| \leq \|T^{n_{k+1}}f - g \| + \|v_k\| \to 0 \text{ as } k \to \infty,
\end{align*}
where \(v_k = T^{n_k}(f) - g \to 0\), so \(T^{n_{k+1}-n_k}(g) \to g\) as \(k \to \infty\). Suppose that \(n_{k+1}-n_k\) is bounded. Then \(T^p(g) = g\) for some \(p\) so that \(T^{np}(g) = g\) for arbitrarily large \(np\). Otherwise, we have that \(T^{j_k}g \to g\) as \(k \to \infty\) for \(j_k \to \infty\). In either case, we see that \(g\) must be constant. Indeed, \(\|T^j g\|\) is monotonically decreasing in \(j\), since \(\|T\| \leq 1\). By primitivity, \(\|T^j  g\| < \|g\|\) for sufficiently large \(j\) and \(g\) non-constant. To see this, suppose that \(U \subset \A\) is an open set small enough so that \(g(a) < \|g\|\) for all \(a \in U\). By primitivity, for sufficiently large \(j\) we have that \(\sub^j(a)\) contains a letter in \(U\) for all \(a \in \A\), which implies that \(T^j(g)(a) < \|g\|\) for all \(b \in \A\), that is, \(\|T^j g\| < \|g\|\). But \(T\) is a contraction so then cannot converge to \(g\), which is a contradiction, so \(g\) must be constant.

Recall that \(T^{n_k}(f) \to g\) as \(k \to \infty\). On the other hand,
\[
\|T^{n+1}(f) - g\| = \|T^{n+1}(f-g)\| \leq \|T^n(f-g)\| = \|T^n(f) - g\|,
\]
since \(T(g) = g\), so that  \(\|T^n(f) - g\|\) is monotonically decreasing. Hence, the whole sequence \(T^n f \to g\) as \(n \to \infty\), as required. If \(f \notin K\) then \(f = f_+ - f_-\), where \(f_+\) and \(f_- \in K\), so by the above \(T^n(f) = T^n(f_+) - T^n(f_-) \to g_+ - g_-\) as \(n \to \infty\) so that \(T\) is strongly power convergent on all of \(E\). By Theorem \ref{thm:inv measures - tiling} the tiling dynamical system \((\Omega,\R)\) or, by Corollary \ref{cor: inv measures - subshift}, the subshift \((X_\sub,\sigma)\), is uniquely ergodic.
\end{proof}

The following substitution (considered in \cite[Exp.~12.1]{Queffelec} in the one-sided shift setting) is an example for which $T$ cannot be quasi-compact, and $r$ cannot be a pole of the resolvent, since $T$ is \emph{not} uniformly power convergent, but the subshift is still uniquely ergodic by the above result:

\begin{example} \label{ex: CL not qc}
Let $\A = S^1 \subseteq \C$ and recall the substitution from Example~\ref{ex:CL-S1}: $\sub(z) = z \ \alpha z$, where $\alpha \in S^1$ is irrational (i.e., \(\alpha = e^{2 \pi \beta i} \) for \(\beta\) irrational). By Kronecker's theorem, via the density of the orbit of $z \mapsto \alpha z$, it is clear that $\sub$ is primitive. Also note that $\Phi=\left\langle \text{id}, \alpha \cdot \right\rangle$ is an equicontinuous semigroup because it is generated by group translations.

We do not have that $T^n \mapsto P$ uniformly as $n \mapsto \infty$, where $P$ is a projection operator to the eigenline spanned by $\ell$. Indeed, first note that an eigenmeasure is clearly provided by the Lebesgue measure on the circle (since each letter map of the substitution is a rotation). Consider a function $f \in E$, $\|f\| = 1$ that has very small integral $c = c_f > 0$ and yet has $f(z) = 1$ for $z = 1$, $\alpha$, $\alpha^2$, \ldots, $\alpha^n$ (by taking a very spiky function this is of course possible; in fact, by taking a function with negative spikes too, we may choose such a function so that $c = 0$). Then by Lemma \ref{lem: powers of operator}, $(M^n(f))(1) = 2^n$ and hence $T^n(f)(1) = 1$, since all elements of $\sub^n(1)$ are of the form $\alpha^n$ for some non-negative $k \leq n$. Assuming uniform power convergence, we would have $T^n f \to c_f \cdot \ell$ uniformly as $n \mapsto \infty$, that is $\|T^n(f) - c_f \ell\| < \epsilon$ for sufficiently large $n$ over all \(\|f\| \leq 1\). However, $\|T^n(f) - c_f \ell\| \geq 1-c_f$, so we do not have such uniform convergence and thus $r$ cannot be a pole of the resolvent, and $M$ cannot be quasi-compact, by Theorem \ref{thm: primitive+qc => uniquely ergodic}. However, whilst \(T\) is not uniformly power convergent, it is strongly power convergent by Theorem \ref{thm: CL UE}.
\end{example}

\begin{remark}
Since Theorem~\ref{thm: ILC-fusion-invmeas} covers fusion tilings in higher dimensions, Theorems~\ref{thm:inv measures - tiling} (on unique ergodicity and strong power convergence), and \ref{thm: CL UE} (for constant-length substitutions) generalises in higher dimensions, under similar assumptions and proofs. This will be addressed in a future work. 
\end{remark}

\section{Discrepancy estimates from a spectral gap}\label{SEC:discrepancy}

Throughout this section, we make the following assumptions on $\sub$:
\begin{itemize}
	\item[(A1):] $T$ is uniformly power convergent;
    \item[(A2):] $T$ is irreducible.
\end{itemize}
The first assumption implies that $1$ is a simple pole of the resolvent \cite[Thm.~2.5]{K:power-conv}. It then follows from Theorem \ref{thm: nlf from non-void pps or pole} that
\begin{itemize}
	\item  $\sigma_\mathrm{per}(T) = \{1\}$;
	\item $\sub$ admits a (unique) natural length function $0 \neq \ell \in K$ with stretching factor $\lambda = r$.
\end{itemize}

Since \(\sigma(M) = r\sigma(T)\) is compact, \(r\) is a pole of the resolvent (and thus isolated) and the only element of the peripheral spectrum of \(M\), it follows that \(M\) has a \emph{spectral gap}, that is, \(\sigma(M)\) has a `second largest element':
\[
r_2 \coloneqq \sup\{|\lambda| \mid \lambda \in \sigma(M), \lambda \neq r\} < r.
\]
We fix a (necessarily unique) eigenmeasure $\mu \in K'$, which we normalise here with $\mu(\ell)=1$.

\begin{remark}
By Theorem \ref{thm: primitive+qc => uniquely ergodic}, the above conditions are satisfied whenever $\sub$ is primitive and  $T$ is quasi-compact. This includes any primitive substitution for which Theorem \ref{thm: quasi-compact subs} applies, so the constructions to follow apply to those relevant examples from the last section.
\end{remark}

The existence of a spectral gap will allow for some control on the discrepancy on the `expected weight' of a weighted selection of tiles across a large supertile relative to the `actual' sum of weights. Given $f \in E$ (a `weight function'), $a \in \A$, and $n \in \N$ we define
\begin{equation} \label{eq: expected}
\Exp(f,a,n) = r^n \ell(a) \cdot \mu(f).
\end{equation}
That is, we simply multiply the length of the $n$-supertile $\sub^n(a)$ by the `average weight' of $f$. Such quantities are prevalent in recent works on bounded displacement equivalence\cite{FGS:BDE,SS:discrepancy}.

\begin{example}
For a primitive substitution on a finite alphabet, $\mu$ is represented by the row vector $f = (f_\alpha,f_\beta,\ldots)$ of frequencies $f_a > 0$ of each letter $a \in \A$. Then the expected number of tiles of type $b$ in $\sub^n(a)$ is
$\Exp(\bbo_b,a,n) = r^n \ell(a) \mu(\bbo_b) = r^n \ell(a) (v \cdot f) = r^n \ell(a) f_b$,
where $v$ is the vector with entry corresponding to $b$ equal to $1$ and all other entries $0$. So we just multiply length of the $n$-supertile by the frequency of occurrence of $b$. More generally, given a weighted selection $v = (v_\alpha,v_\beta,\ldots)$ of tiles, where we have one $v_a \in \R$ for each tile $a \in \A$,
\[
\Exp(f,a,n) = r^n \ell(a) (v \cdot f).
\]
Here, we are identifying $v$ with an arbitrary function $f \in E$, given by $f(a) = v_a$.
\end{example}

We define the \emph{actual} weighted sum of $f$ across an $n$-supertile $\sub^n(a)$ by
\[
\Act(f,a,n) = \sum_{b \triangleleft \sub^n(a)} f(b) = (M^n f)(a),
\]
where the second equality follows from Lemma \ref{lem: powers of operator}. The result below bounds the discrepancy between the expected and actual weight across all possible supertiles. 
For a treatment for primitive substitutions on finite alphabets and applications to questions on balancedness, we refer to \cite{Adam04} and \cite{Adam03}, respectively.

\begin{theorem} \label{thm: discrepancy}
Suppose assumptions \textnormal{(A1)} and \textnormal{(A2)} hold. Then, there is some function $\theta \colon \N \to \R_{>0}$ with $\sqrt[n]{\theta(n)} \to 1$ as $n \to \infty$ so that for all $f \in E$ with $\|f\| \leq 1$, $a \in \A$ and $n \in \N$ we have
\[
|\Exp(f,a,n) - \Act(f,a,n)| \leq \theta(n)(r_2)^n.
\]
In particular, for any $\epsilon > 0$, we have the following for sufficiently large $n$:
\[
|\Exp(f,a,n) - \Act(f,a,n)| \leq (r_2+\epsilon)^n.
\]
\end{theorem}

In the above, recall that \(r_2\) denotes the maximum modulus of element of \(\sigma(M) \setminus \{r\}\). Before proving this result, we require a few technical lemmas regarding the spectrum of the operator restricted to the subspace
\[
\Pi \coloneqq \{f \in E \mid \mu(f) = 0\}.
\]
By uniform power convergence, we have $T^n \to P$ uniformly as $n \to \infty$, where $P$ is a projection to the eigenline $\Lambda = \langle \ell \rangle_\R$. Hence $\Pi = \ker(P)$ is a closed, co-dimension one subspace and we have $E = \Lambda \oplus \Pi$. We let $T_\Pi$ denote the restriction to $\Pi$.

\begin{lem}
Suppose assumptions \textnormal{(A1)} and \textnormal{(A2)} hold. Then, we have the inclusion $\sigma(T_\Pi) \subseteq \sigma(T)$.
\end{lem}

\begin{proof}
Since $P$ is bounded (and thus continuous), its kernel is closed and so $T_\Pi$ is a bounded operator on a Banach space. For such operators, $\tau \in \sigma(T)$ if and only if $T-\tau \mathbb{I}$ is not bijective. Suppose that $\tau \in \sigma(T_\Pi)$. First assume that $T-\tau\mathbb{I}$ is not injective, so that $T(f) = \tau f$ for some $f \in \Pi$. The same $f$ demonstrates that $\tau \in \sigma(T)$.

So suppose instead that $T-\tau\mathbb{I}$ is not surjective, and thus there is some $g \in \Pi$ with $Tf-\tau f \neq g$ for any $f \in \Pi$. Assume, on the other hand, that $\tau \notin \sigma(T)$ so that $Tf-\tau f = g$ for some $f \in E$. We may write $f = a + b $ for $a \in \Lambda$ and $b \in \Pi$. Then
\[
T(a+b)-\tau(a+b) = (Ta-\tau a) + (Tb-\tau b) = g \in \Pi.
\]
But since $\Lambda$ and $\Pi$ are complementary this implies that $Ta - \tau a = 0$. If $a \neq 0$ it follows that $\tau = 1$ (since then $a$ is a non-negative multiple of $\ell$), which we already know is in $\sigma(T)$. Otherwise, we have that $Tb - \tau b = g$, contradicting that $Tb - \tau b \neq g$ for all $b \in \Pi$.
\end{proof}

We get the following consequence of uniform power convergence:

\begin{lem}
We have that $1 \notin \sigma(T_\Pi)$.
\end{lem}

\begin{proof}
Let \(f \in \Pi\) be arbitrary; we have \(T^n f \to P(f) = c \cdot \ell\) as \(n \to \infty\), uniformly in \(f\). Clearly \(c = 0\), since \(\mu(T^n f) = \mu(f) = 0\) for all \(n\) and \(\mu(c \ell) = c\). Thus \(T^n_\Pi \to 0\) uniformly so $\|T_\Pi^n\| < 1$ for some $n \in \N$. It follows from Gelfand's formula that $r(T_\Pi) < 1$, and thus $1 \notin \sigma(T_\Pi)$.
\end{proof}

\begin{lem}
Suppose that $\kappa \in \sigma(T)$ with $\kappa \neq 1$. Then $\kappa \in \sigma(T_\Pi)$.
\end{lem}

\begin{proof}
Let $\kappa \in \sigma(T)$ with $\kappa \neq 1$. First assume that $Tf - \kappa f = 0$ for some $f \in E$ and write $f = a + b$ for $a = c \ell \in \Lambda$ and $b \in \Pi$. Then
\[
Tf - \kappa f = c(1-\kappa)\ell + (Tb - \kappa b) = 0.
\]
Since $c(1-\kappa)\ell \in \Lambda$, $Tb-\kappa b \in \Pi$, and the subspaces $\Lambda$ and $\Pi$ are complementary, it is necessary that $\kappa = 1$ or $c=0$. Since we assume that $\kappa \neq 1$, we have that $c=0$ and thus $f \in \Pi$, so that $\kappa \in \sigma(T_\Pi)$ too.

So suppose instead that $Tf - \kappa f \neq g$ for all $f \in E$ and some $g \in E$. Write $g = a+b = c \ell + b$ for $a \in \Lambda$ and $b \in \Pi$. Suppose that $Tf - \kappa f = b$ for some $f \in \Pi$. Then
\[
T\left(\frac{c}{1-\kappa} \ell + f\right) - \kappa\left(\frac{c}{1-\kappa}\ell + f\right) = c\ell + (Tf - \kappa f) = c\ell + b = g,
\]
contradicting that $Tx - \kappa x \neq g$ for all $x \in E$. Hence $Tf - \kappa f \neq b$ for all $f \in \Pi$ and so $\kappa \in \sigma(T_\Pi)$, as required.
\end{proof}

The three lemmas above imply the following result. 

\begin{coro} \label{cor: spectral gap}
We have $\sigma(T) = \sigma(T_\Pi) \sqcup \{1\}$ or, equivalently, $\sigma(M) = \sigma(M_\Pi) \sqcup \{r\}$.
\end{coro}

\begin{proof}[Proof of Theorem~\textnormal{\ref{thm: discrepancy}}]
We may write
\[
|\Exp(f,a,n) - \Act(f,a,n)| = |r^n\ell(a) \cdot \mu(f) - (M^n f)(a)| = |M^n (\mu(f)\ell - f)(a)|.
\]
Since $M = r T$ we thus have
\[
|\Exp(f,a,n) - \Act(f,a,n)| = r^n |T^n(\mu(f) \ell - f)(a)| = r^n|(T^n v)(a)|,
\]
where $v = \mu(f)\ell - f$. We have that $\mu(v) = \mu(\mu(f)\ell - f) = \mu(f)(\mu(\ell)) - \mu(f) = 0$ (since we normalise with $\mu(\ell) = 1$) so $v \in T_\Pi$.

By Corollary \ref{cor: spectral gap}, $r(T_\Pi) = r_2/r$. By Gelfand's formula for the spectral radius,
\[
\lim_{n \to \infty} \sqrt[n]{\|T_\Pi^n\|} = r(T_\Pi), \text{ hence } \sqrt[n]{\|T_\Pi^n\|} \leq \psi(n)(r_2/r),
\]
where $\psi(n) \to 1$ as $n \to \infty$. Taking $n$th powers and substituting into the above,
\[
|\Exp(f,a,n) - \Act(f,a,n)| = r^n|(T^n v)(a)| = r^n|(T_\Pi^n v)(a)| \leq r^n \|T_\Pi^n\|\|v\| \leq \|v\|(\psi(n))^n(r_2)^n,
\]
where $(\psi(n))^n$ has $n$th roots converging to $1$ as $n \to \infty$. Note that $\|v\| = \|\mu(f)\ell - f\|$ is bounded above by a constant $c$, since $\mu$ is bounded and $\|f\| \leq 1$, so the discrepancy estimate follows taking $\theta(n) = c (\psi(n))^n$.

Let $\delta > 0$ be arbitrary. Then $h(n) = (1+\delta)^n$ is such that $\sqrt[n]{h(n)} = (1+\delta) > \sqrt[n]{\theta(n)}$, and thus $h(n) > \theta(n)$, for sufficiently large $n$. So for such $n$,
\[
|\Exp(f,a,n) - \Act(f,a,n)| \leq \theta(n)(r_2)^n \leq h(n)(r_2)^n \leq (1+\delta)^n(r_2)^n = ((1+\delta)r_2)^n.
\]
Setting $0< \delta < \epsilon/r_2$ establishes the second result.
\end{proof}

\section{Higher dimensions} \label{sec: higher dimensions}

Many constructions and results seen here, excluding those pertaining to natural length functions, also apply to the higher dimensional setting. We make some initial observations on these extensions in this final section.

\subsection{ILC tilings and substitutions in higher dimensions}
In contrast to our symbolic setup in the \(d=1\) case, it makes sense to begin with the geometric substitution rule on a compact prototile set. Many of the following definitions are inspired by \cite{PFS:fusion-ILC} and are extensions of those in Definition~\ref{def: geometric hulls}. We let \(\A\) be a compact, Hausdorff space, now considered as a set of \emph{labels}. We write
\[
S(\R^d) \coloneqq \{X \subseteq \R^d \mid X \text{ is non-empty and compact}\},
\]
which we consider as a topological space, with topology induced by the Hausdorff distance (or equivalently the Vietoris topology). A \emph{prototile} (with label in \(\A\)) is then a pair \(p = (s,a)\), whose \emph{support} is denoted by \(\mathrm{supp}(p) \coloneqq s \in S(\R^d)\). We assume that each \(\mathrm{supp}(p)\) is homeomorphic to a closed ball in \(\R^d\) and that the origin lies in its interior. We further assume that the associated fusion rule is van Hove, in the sense of \cite{PFS:fusion-ILC}. This does not seem too restrictive and holds, loosely speaking, so long as the boundaries of tiles are not too wild. Certain examples can easily be seen to be van Hove. This includes the pinwheel below, since in this example all supertiles are inflates of a single (up to rigid motion) polygonal shape. We let \(\mathrm{label}(p) \coloneqq a\) denote the \emph{label} of prototile \(p\).

A \emph{prototile set} \(\mathcal{P}\) is then a set of prototiles which is compact as a subspace of \(S(\R^d) \times \A\). Without loss of generality, \(\mathrm{label} \colon \mathcal{P} \to \A\) is a homeomorphism (see Remark \ref{rem: homeo to labels}) i.e., all labels are realised and distinct prototiles have distinct labels (continuity of the inverse follows from compactness). Then \(\mathrm{supp} \circ \mathrm{label}^{-1} \colon \A \to S(\R^d)\) is continuous, since both \(\mathrm{supp}\) and \(\mathrm{label}^{-1}\) are continuous. This condition may be interpreted as saying that, if a sequence of labels converges, then so do the supports of their tiles, a condition required in \cite{PFS:fusion-ILC}.

\begin{remark} \label{rem: homeo to labels}
Restricting \(\A\) to be compact causes no loss of generality given compactness of \(\mathcal{P}\) (as otherwise we may restrict labels to \(\mathrm{label}(\mathcal{P})\)). Similarly, the labelling being a homeomorphism causes no loss of generality. Indeed, otherwise, replace \(\A\) with \(\mathcal{P}\); more precisely replace any prototile \(p = (s,a)\) with \((s,p)\) (i.e., labelling the prototile with itself). Indeed, the map
\[
f \colon \mathcal{P} \to S(\R^d) \times \mathcal{P}, \quad p = (s,a) \mapsto (\mathrm{supp}(p),p)
\]
is continuous and bijective onto the new prototile set \(\mathcal{P}' = \{(\mathrm{supp}(p),p) \mid p \in \mathcal{P}\}\), with labels now in \(\mathcal{P}\), so \(\mathcal{P} \cong \mathcal{P}'\) in a way preserving tile supports. An alternative is to drop the `labels' and always work with \(\mathcal{P}\) in place of \(\A\) (or allowing different shapes of tiles with the same label). However, we will choose here to use the conventions in \cite{PFS:fusion-ILC}.
\end{remark}

For a prototile \(p = (s,a)\) and a vector \(x \in \R^d\), we define \(p+x \coloneqq (s+x,a)\). Given a prototile set \(\mathcal{P}\), a \emph{tile} is a translate of a prototile. A \emph{patch} is a finite set of tiles whose supports have mutually disjoint interiors and connected union. We let \(\mathcal{P}^*\) denote the set of patches. This can be topologised as follows. Consider an arbitrary patch \(P \in \mathcal{P}^*\) with tiles \(t_1 = p_1 + x_1\), \ldots, \(t_n = p_n + x_n\), where \(p_i \in \mathcal{P}\) and \(x_i \in \R^d\). For arbitrary open neighbourhoods \(V_i \subset \mathcal{P}\) of each \(p_i\) and \(\epsilon > 0\), we define an open neighbourhood \(U(P,V_1,\ldots,V_n,\epsilon)\) of \(P\) as all patches \(P' \in \mathcal{P}^*\) having a bijection \(h\) from the tiles of \(P\) to those of \(P'\) so that, for each \(i\) and \(h(p_i + x_i) = p' + x'\), we have \(p' \in V_i\) and \(|x'-x_i| < \epsilon\). A \emph{tiling} is a covering of \(\R^d\) by tiles with mutually disjoint interiors. The above topology on \(\mathcal{P}^*\) also topologises the space of all tilings, taking a basic open set for each open subset \(U \subseteq \mathcal{P}^*\) to consist of all tilings that contain a patch in \(U\). So `small' open sets are defined by basic open sets \(U(P,V_1,\ldots,V_n,\epsilon)\) for \(P\) covering a large ball about the origin, each \(V_i\) a small neighbourhood of each prototile \(p_i\) and \(\epsilon > 0\) small; tilings in this basic open set thus contain a patch covering a large ball about the origin with tiles that may be paired with those of \(P\) in a way that only shifts tiles (translates and labels) a small amount.

A \emph{substitution rule} (or \emph{stone inflation}), with inflation factor \(\lambda > 1\), is defined to be a continuous map \(\sub \colon \mathcal{P} \to \mathcal{P}^*\) so that, for each \(p \in \mathcal{P}\), we have
\[
\lambda \ \mathrm{supp}(p) = \bigcup_{t \in \sub(p)} \mathrm{supp}(t).
\]
An \emph{\(n\)-supertile} is a patch \(\sub^n(p)\), for \(n \in \N_0\) and \(p \in \mathcal{P}\), where \(\sub\) is extended to patches in the obvious way. This again defines \emph{generated patches}, the subpatches of supertiles, and \emph{legal patches}, given as the closure of the collection of generated patches. We then obtain the \emph{continuous hull} (or \emph{tiling space}) \(\Omega_\sub\), consisting of all tilings whose finite patches are legal.

The space \(\Omega:=\Omega_\sub\) is compact and Hausdorff, and translation by \(\R^d\) acts continuously on it, making \((\Omega,\R^d)\) a topological dynamical system. By construction, 
elements of $\Omega$ are fusion tilings, and the results of \cite{PFS:fusion-ILC} still apply to this slightly generalised setting, where \(\A\) is compact Hausdorff (but is not given a metric).

\subsection{Results on unique ergodicity in higher dimensions}
Because our results on unique ergodicity mostly proceed in terms of the substitution operator, it turns out that essentially identical proofs of most results here follow also in the higher dimensional setting.

We still have the Banach space \(E = C(\mathcal{P})\) (recall that \(\mathcal{P} \cong \A\)) and the substitution operator \(M\) on it (Eq.~\eqref{eq: substitution operator}) where, for \(f \colon \mathcal{P} \to \R\), we define \((Mf)(p)\), for a prototile \(p \in \mathcal{P}\), to be
\[
(Mf)(p) \coloneqq \sum_{t \in \sub(p)} f(t),
\]
summing with multiplicities. Here, we evaluate \(f(p_i+x) = f(p_i)\) for all prototiles \(p_i \in \mathcal{P}\) and translates \(x \in \R^d\). The notion of irreducibility and primitivity may be defined just as before, and results equating them to properties of the substitution operator (such as Propositions \ref{prop: irreducible M} and \ref{prop: primitive <=> st. positive iterates}) still hold, with identical proofs. The notion of a `natural length function' no longer applies, but the function \(\ell \in E\), given by setting \(\ell(a)\) as the volume of \(a\), will be a positive eigenvector of \(M\), with eigenvalue \(\lambda^d\). 

Most importantly here, Theorem~\ref{thm: ILC-fusion-invmeas} relating invariant measures of the tiling dynamical system to sequences of volume normalised and transition consistent measures still apply. Then Theorem \ref{thm:unique ergod FS} continues to hold in this higher dimensional setting, as does the part of Theorem \ref{thm:inv measures - tiling} on unique ergodicity:

\begin{theorem}
Let \(\sub\) be an irreducible substitution in \(\R^d\), on a compact Hausdorff prototile set \(\mathcal{P}\) with \(T = M/r\) strongly power convergent. Then \((\Omega,\R^d)\) is uniquely ergodic.
\end{theorem}

The higher dimensional analogue of Theorem \ref{thm: quasi-compact subs} similarly follows by the same proofs. The important structure of constant length substitutions required for the proof of Theorem \ref{thm: CL UE} is really that of having continuous columns. Thus, we define a \emph{constant length substitution} (in \(\R^d\)) to be a substitution for which we have continuous maps \(\sub_i \colon \mathcal{P} \to \mathcal{P}\) for \(i=1\), \ldots, \(L\), the \emph{columns}, so that for each \(p \in \mathcal{P}\) there is a bijection between the tiles of \(\sub(p)\) and \(\{\sub_i(p)\}_{i=1}^L\). This is exactly what is needed for Equation \eqref{eq: constant length substitution operator} to apply with each column continuous. Then, using an identical proof to that of Theorem \ref{thm: CL UE}, we obtain:

\begin{theorem} \label{thm: CS UE}
Suppose that \(\sub\) is a primitive, constant length substitution in \(\R^d\) whose columns generate an equicontinuous semigroup \(\Phi\). Then \(T\) is strongly power convergent and the tiling dynamical system associated with \(\sub\) is uniquely ergodic.
\end{theorem}

\begin{remark}
For \(d=1\) the above notion of constant length coincides with the usual one i.e., \(|\sub(p)|\) is constant in \(p \in \mathcal{P}\). It follows from the definition of the topology on \(\mathcal{P}^*\) that \(|\sub(p)|\) is locally constant and thus globally constant if \(\mathcal{P}\) is connected. However, it is not clear to us that \(|\sub(p)|\) being globally constant is sufficient for the constant length condition defined above for \(d > 1\), which requires a type of continuous indexing of subtiles.
\end{remark}

We now demonstrate how unique ergodicity of the Conway--Radin pinwheel tilings \cite{R:pinwheel} follows in a straightforward way from our results. Interestingly this approach seems quite distinct from others \cite{Fre08,MPS06}, which make use of Weyl's criterion at some stage to show statistical circular symmetry. Instead, this fact follows from unique ergodicity here because, from unique ergodicity, the unique eigenmeasure on the prototile space must correspond to Lebesgue measure.

\begin{example}
The pinwheel substitution has prototile set \(\mathcal{P}\) consisting of a \(1\)--\(2\)--\(\sqrt{5}\) distinguished triangle \(p\) and all linear isometries \(gp\) of it, for \(g \in \mathrm{O}(2,\R)\). Thus, \(\mathcal{P}\) is homeomorphic to the disjoint union of two copies of \(S^1\). The substitution \(\sub\) is continuous and constant length: we may arbitrarily label the tiles of \(\sub(p)\)  with distinct labels from \(1\) to \(5\) to define \(\sub_i(p)\), and then we define \(\sub_i(gp) = g \, \sub_i(p)\) (note that this is well defined and continuous, since \(gp = p\) implies that \(g = \mathrm{id}\)).

Then each column \(\sub_i \colon \mathcal{P} \to \mathcal{P}\) is determined by some rigid motion \(g_i \in \mathrm{O}(2,\R)\), by setting \(\sub_i(gp) = g(g_i(p))\). Let \(\alpha\) be clockwise rotation by \(\arctan\tfrac{1}{2}\), \(\beta\) be anticlockwise rotation by \(\tfrac{\pi}{2}\) and \(\gamma\) be reflection about the \(y\)-axis, all acting on the left. The columns (see Figure \ref{fig: pinwheel}) are then determined by
\[
g_1 = \alpha, \quad \quad \ g_2 = \alpha \beta^2, \quad\quad  \ g_3 = g_4 = \alpha \beta^2 \gamma \quad\quad  \text{ and }\quad\quad   g_5 = \alpha \beta \gamma .
\]

\begin{figure}[h]
\centering
\def\svgwidth{.9\columnwidth}
\begingroup%
  \makeatletter%
  \providecommand\color[2][]{%
    \errmessage{(Inkscape) Color is used for the text in Inkscape, but the package 'color.sty' is not loaded}%
    \renewcommand\color[2][]{}%
  }%
  \providecommand\transparent[1]{%
    \errmessage{(Inkscape) Transparency is used (non-zero) for the text in Inkscape, but the package 'transparent.sty' is not loaded}%
    \renewcommand\transparent[1]{}%
  }%
  \providecommand\rotatebox[2]{#2}%
  \newcommand*\fsize{\dimexpr\f@size pt\relax}%
  \newcommand*\lineheight[1]{\fontsize{\fsize}{#1\fsize}\selectfont}%
  \ifx\svgwidth\undefined%
    \setlength{\unitlength}{850.39370079bp}%
    \ifx\svgscale\undefined%
      \relax%
    \else%
      \setlength{\unitlength}{\unitlength * \real{\svgscale}}%
    \fi%
  \else%
    \setlength{\unitlength}{\svgwidth}%
  \fi%
  \global\let\svgwidth\undefined%
  \global\let\svgscale\undefined%
  \makeatother%
  \begin{picture}(1,0.33333333)%
    \lineheight{1}%
    \setlength\tabcolsep{0pt}%
    \put(0,0){\includegraphics[width=\unitlength,page=1]{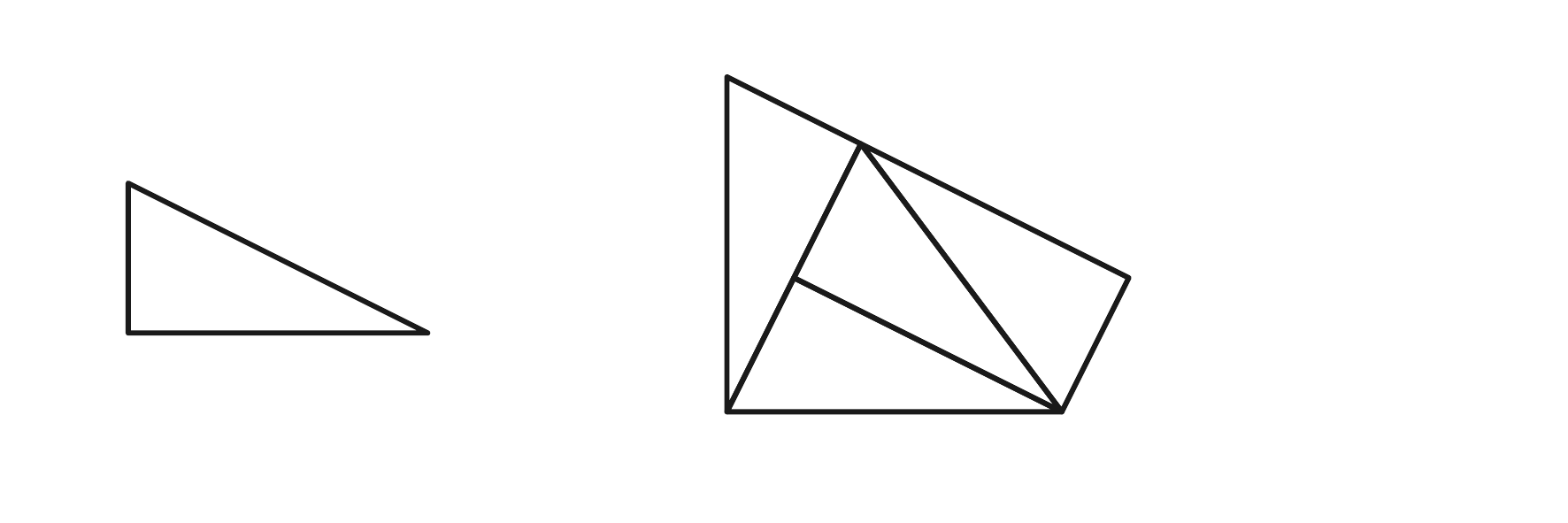}}%
    \put(0.51790812,0.0950607){\color[rgb]{0,0,0}\makebox(0,0)[lt]{\lineheight{1.25}\smash{\begin{tabular}[t]{l}$\alpha \beta^2 \gamma$\end{tabular}}}}%
    \put(0.54959508,0.16211953){\color[rgb]{0,0,0}\makebox(0,0)[lt]{\lineheight{1.25}\smash{\begin{tabular}[t]{l}$\alpha$\end{tabular}}}}%
    \put(0.64000984,0.14401014){\color[rgb]{0,0,0}\makebox(0,0)[lt]{\lineheight{1.25}\smash{\begin{tabular}[t]{l}$\alpha \beta^2$\end{tabular}}}}%
    \put(0.47140129,0.21277489){\color[rgb]{0,0,0}\makebox(0,0)[lt]{\lineheight{1.25}\smash{\begin{tabular}[t]{l}$\alpha \beta \gamma$\end{tabular}}}}%
    \put(0,0){\includegraphics[width=\unitlength,page=2]{pinwheel.pdf}}%
    \put(0.73155609,0.09506078){\color[rgb]{0,0,0}\makebox(0,0)[lt]{\lineheight{1.25}\smash{\begin{tabular}[t]{l}$\alpha \beta^2 \gamma$\end{tabular}}}}%
  \end{picture}%
\endgroup%

\caption{The pinwheel substitution \label{fig: pinwheel}}
\end{figure}

Note that we have the relations \(\alpha\beta = \beta\alpha\), \(\gamma^2 = \mathrm{id}\), \(\alpha\gamma = \gamma \alpha^{-1}\) and \(\beta\gamma = \gamma\beta^{-1}\). Since \(\mathrm{O}(2,\R)\) is a compact Lie group, the semigroup \(\Phi\) generated by the columns is uniformly equicontinuous. Primitivity of the substitution is easily established: \(\sub^2(p)\) contains (up to translation) a copy of \(p\) (since \(g_3^2(p) = p\)), a reflection of \(p\) (since \(g_4(g_2(p)) = \gamma(p)\)) and a rotation of \(p\) by the irrational angle \(2\arctan\tfrac{1}{2}\) (since \(g_1^2(p) = \alpha^2(p)\)), so that \(\sub^n(gp)\) fills \(\mathcal{P}\) arbitrarily densely, uniformly in \(g \in \mathrm{O}(2,\R)\). Hence, unique ergodicity follows from Theorem \ref{thm: CS UE}. In a similar way, any primitive generalised pinwheel tiling space \cite{S:pinwheel} is uniquely ergodic.
\end{example}

\section*{Acknowledgements}

The authors would like to thank Michael Baake, Joel Feinstein, Natalie Priebe Frank, Dirk Frettl\"oh, Alexey Garber, Philipp Gohlke, Gerhard Keller, Samuel Petite, Julia Slipantschuk, Nicolae Strungaru and Reem Yassawi for valuable discussions. We thank the referees for their valuable comments. NM is funded by the German Academic Exchange Service (DAAD) via the Postdoctoral Researchers International Mobility Experience (PRIME) Fellowship programme. DR 
acknowledges the support of the German Research Foundation (DFG) via SFB1283/1 2019-317210226. JW expresses his gratitude to the same SFB of the Faculty of Mathematics of Bielefeld University, where a part of this work was completed.

\bibliography{tilings}
\bibliographystyle{plain}

\end{document}